\documentclass[english]{article}
\usepackage[T1]{fontenc}
\usepackage[latin9]{inputenc}
\usepackage{array}
\usepackage{float}
\usepackage{multirow}
\usepackage{amsmath}
\usepackage{amsthm}
\usepackage{amssymb}
\usepackage{graphicx}

\makeatletter

\providecommand{\tabularnewline}{\\}

  \theoremstyle{definition}
  \newtheorem{defn}{\protect\definitionname}
\theoremstyle{plain}
\newtheorem{thm}{\protect\theoremname}
 \theoremstyle{definition}
  \newtheorem{example}{\protect\examplename}

\@ifundefined{showcaptionsetup}{}{%
 \PassOptionsToPackage{caption=false}{subfig}}
\usepackage{subfig}
\makeatother

\usepackage{babel}
  \providecommand{\definitionname}{Definition}
  \providecommand{\examplename}{Example}
\providecommand{\theoremname}{Theorem}

\begin{document}

\title{Nonlinear Normal Modes and Spectral Submanifolds: \\
Existence, Uniqueness and Use in Model Reduction}

\author{George Haller\thanks{Corresponding author. Email: georgehaller@ethz.ch}
~~and Sten Ponsioen}

\maketitle
\begin{center}
Institute for Mechanical Systems, ETH Zürich \\
Leonhardstrasse 21, 8092 Zürich, Switzerland
\par\end{center}
\begin{abstract}
We propose a unified approach to nonlinear modal analysis in dissipative
oscillatory systems. This approach eliminates conflicting definitions,
covers both autonomous and time-dependent systems, and provides exact
mathematical existence, uniqueness and robustness results. In this
setting, a nonlinear normal mode (NNM) is a set filled with small-amplitude
recurrent motions: a fixed point, a periodic orbit or the closure
of a quasiperiodic orbit. In contrast, a spectral submanifold (SSM)
is an invariant manifold asymptotic to a NNM, serving as the smoothest
nonlinear continuation of a spectral subspace of the linearized system
along the NNM. The existence and uniqueness of SSMs turns out to depend
on a spectral quotient computed from the real part of the spectrum
of the linearized system. This quotient may well be large even for
small dissipation, thus the inclusion of damping is essential for
firm conclusions about NNMs, SSMs and the reduced-order models they
yield.
\end{abstract}

\section{Introduction}

Decomposing nonlinear oscillations in analogy with linear modal analysis
has been an exciting perspective for several decades in multiple disciplines.
In the engineering mechanics literature, this approach was initiated
by Rosenberg \cite{rosenberg62}, who defines a nonlinear normal mode
in a conservative system as a synchronous periodic oscillation that
reaches its maximum in all modal coordinates at the same time. Shaw
and Pierre \cite{shaw93} offers an elegant alternative, envisioning
nonlinear normal modes as invariant manifolds that are locally graphs
over two-dimensional modal subspaces of the linearized system. These
definitions have subsequently been relaxed and generalized to different
settings, as surveyed by the recent reviews of Avramov and Mikhlin
\cite{avramov10,avramov13}, Kerschen \cite{kerschen14} and Renson
et al. \cite{renson16}.

In conservative autonomous systems, a relationship between the above
two views on nonlinear normal modes is established by the subcenter-manifold
theorem of Lyapunov \cite{lyapunov92}. In its strongest version due
to Kelley \cite{kelley69}, this theorem guarantees that unique and
analytic invariant manifolds tangent to two-dimensional modal subspaces
of the linearized system at an elliptic fixed point persist in an
analytic nonlinear system under appropriate nonresonance conditions.
These persisting manifolds are in turn filled with periodic orbits.
Roughly speaking, therefore, conservative Shaw--Pierre-type normal
modes are just surfaces composed of Rosenberg-type normal modes, if
one relaxes Rosenberg's synchrony requirement, as is routinely done
in the literature. 

A similar relationship, however, is absent between the two normal
mode concepts for non-conservative or non-autonomous systems.  In
such settings, periodic orbits become rare and isolated in the phase
space. At the same time, either no or infinitely many invariant manifolds
tangent to eigenspaces may exist, most often without containing any
periodic orbit. Having then identical terminology for two such vastly
different concepts is clearly less than optimal. Furthermore, while
both dissipative normal mode concepts are inspired by nonlinear dynamical
systems theory, neither of the two has been placed on firm mathematical
foundations comparable to other classic concepts in nonlinear dynamics,
such as stable, unstable and center manifolds near equilibria (see,
e.g., Guckenheimer and Holmes \cite{guckenheimer83} for a survey).

Indeed, as Neild et al. \cite{neild15} observe, the envisioned Shaw--Pierre-type
invariant surfaces are already non-unique in the linearized system,
and there is no known result guaranteeing their persistence as nonlinear
normal modes in the full nonlinear system. These authors propose normal
form theory as a more expedient computational tool to investigate
near-equilibrium dynamics for model reduction purposes. Truncated
normal forms, however, offer no a priori guarantee for the actual
existence of the structures they predict either. Rather, the persistence
of such structures needs to be investigated on a case by case basis
either numerically or via mathemtical analysis. 

Cirillo et al. \cite{cirillo15a,cirillo15b} also observe the non-uniqueness
of invariant manifolds tangent to eigenspaces in a two-dimensional
linear example. They point out that only one of these manifolds is
infinitely many times differentiable, then state without further analysis
that there is a unique, analytic Pierre--Shaw type invariant surface
tangent to any two-dimensional modal subspace of a nonlinear system.
While a proof of this claim is yet to be provided, the authors also
put forward a computational technique for the construction of invariant
manifolds on larger domains of the phase space. Their proposed approach
is actually a special case of the classic \emph{parametrization method}
(see, e.g., Cabré et al. \cite{cabre05} for a historical and technical
survey), which forms the basis of some of the rigorous invariant manifold
results we will use in the present paper. 

The above concerns about an ambiguity in the definition of Shaw--Pierre
type normal modes have been sporadic in the literature. One reason
might be the general expectation that if one manages to compute arbitrarily
many terms in the Taylor series approximation of an envisioned invariant
surface, then that surface is bound to exist and be unique. While
the success of a low-order numerical or Taylor approximation to an
envisioned invariant manifold is certainly encouraging, by no means
does it give any guarantee for the existence of a unique manifold.
This classic issue is well-documented for the divergence of Lindstedt
series for invariant tori in conservative systems (Arnold \cite{arnold89}).
For dissipative systems, an early example of a divergent expansion
for an invariant manifold was already pointed out by Euler \cite{euler24}
(cf. Arnold \cite{arnold88}). 

We recall Euler's example here briefly in a slightly altered form
relevant for damped vibrations. Consider the planar dynamical system
\begin{eqnarray}
\dot{x} & = & -x^{2},\nonumber \\
\dot{y} & = & -y+x,\label{eq:Euler's example}
\end{eqnarray}
whose right-hand side is analytic on the whole $(x,y)$ plane. A formal
Taylor series for a center manifold tangent to the $x$ axis at the
origin is computable up to any order, but diverges for any $x\neq0$.
Therefore, the formal Taylor expansion of the center manifold does
not converge to any analytic invariant manifold (cf. Appendix \ref{sub:Euler's-example-of}
for details). Accordingly, there is a continuous family of non-unique,
non-analytic center manifolds with vastly different global shapes
for $x>0$ (cf. Fig. \ref{fig:Eulerfig}). None of these manifolds
is distinguished in any way. Approximating any one of them analytically
or numerically, then reducing the full system to this approximation
leads to a highly arbitrary reduced model outside a neighborhood of
the fixed point.

\begin{figure}[H]
\begin{centering}
\includegraphics[width=0.4\textwidth]{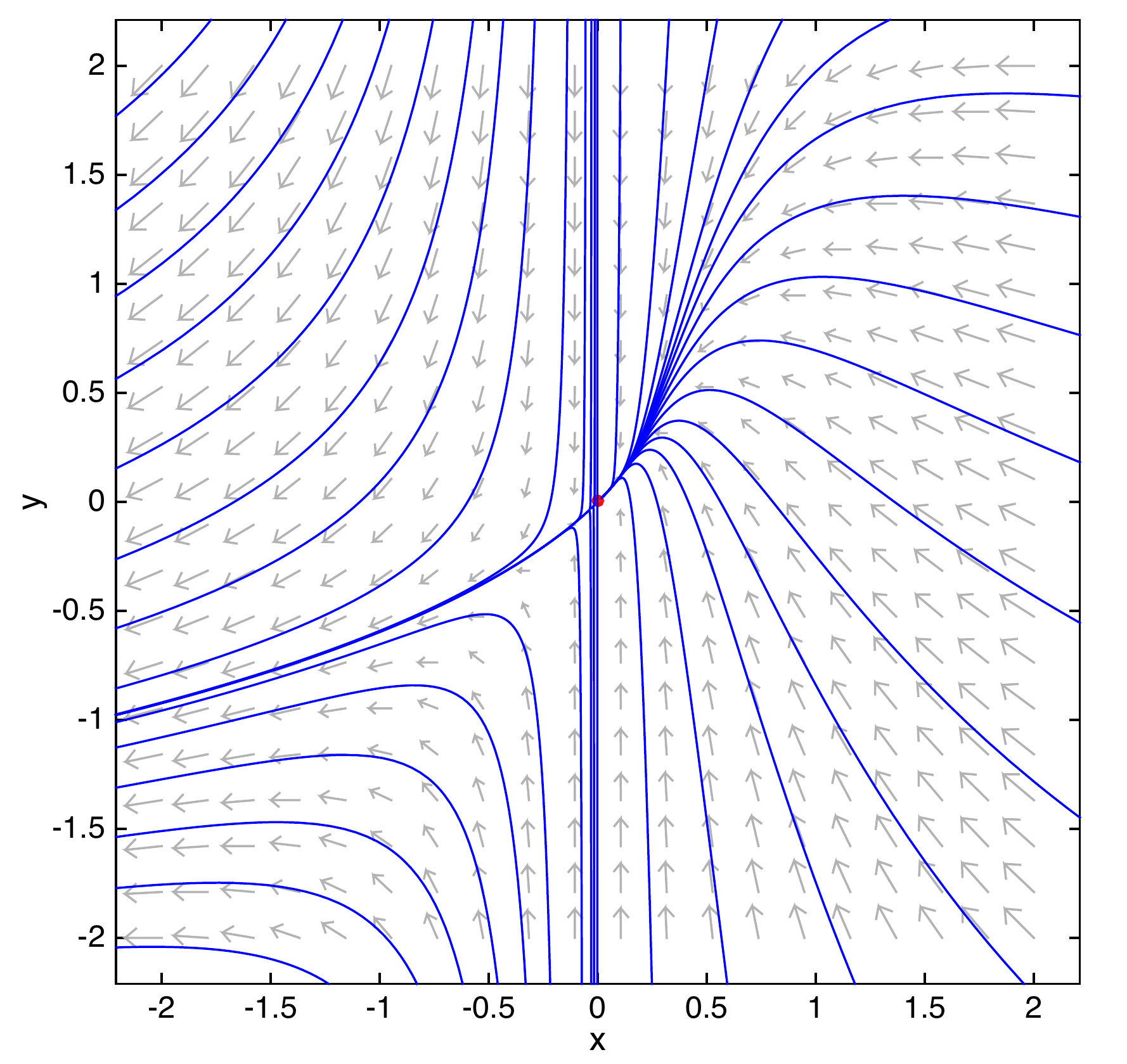}\caption{Phase portrait of the dynamical system \eqref{eq:Euler's example},
showing infinitely many $C^{\infty}$ invariant manifolds with vastly
different global behaviors. The formal computability of the common
Taylor expansion of these manifolds up to any order, therefore,  does
not imply their uniqueness.\label{fig:Eulerfig}}

\par\end{centering}

\end{figure}

The global phase space dynamics of higher-dimensional systems cannot
be visualized in such a simple way as in Fig. \ref{fig:Eulerfig}.
Accordingly, the non-uniqueness of Shaw--Pierre-type invariant surfaces
is often overlooked or ignored in computational studies for multi-degree-of-freeedom
problems (see Renson et al. \cite{renson16} for a recent review).
Some of these approaches solve a PDE for the invariant manifold with
ill-posed boundary conditions; others use the modal subspaces of the
linearization to set boundary conditions away from the fixed point;
yet others envision a uniquely defined boundary condition that they
determine by minimizing an ad hoc cost function. (cf. Appendix \ref{sub:Existence-and-uniqueness issues for numerically solved PDEs}
for details). In all cases, the computed invariant manifold depends
on the choice of basis functions, or domain boundaries, or cost functions
used in the process. The resulting ambiguities in the solutions are
small close to the equilibrium, but are vastly amplified over larger
domains where nonlinear normal mode analysis is meant to surpass the
results from linearization (cf. Fig. \ref{fig:Eulerfig}). 

Here we discuss a unified mathematical approach to nonlinear normal
modes in dissipative systems to address these issues. First, we propose
eliminating the ambiguity in the terminology itself. Borrowing the
original concept of Rosenberg \cite{rosenberg62} from conservative
systems, we call a near-equilibrium quasiperiodic motion in a dissipative,
nonlinear system a \emph{nonlinear normal mode} (NNM). Such NNMs are
certainly special, but the invariant surfaces envisioned in the seminal
work of Shaw and Pierre \cite{shaw93} are arguably more influential
for the overall system dynamics, and can be viewed as invariant surfaces
asymptoting to eigenspaces along a NNM. To emphasize this distinction,
we will refer to the smoothest member of an invariant manifold family
tangent to a modal subbundle along an NNM as a \emph{spectral submanifold}
(SSM). Our precise definitions of NNMs and SSMs (to be given in Definitions
1 and 2) are general enough to apply to both autonomous and externally
forced systems with finitely many forcing frequencies.

With this terminology at hand, we employ classical invariant manifold
results of Fenichel \cite{fenichel71} and more recent invariant manifold
results of Cabré et al. \cite{cabre03} and Haro and de la Llave \cite{haro06}
to deduce existence, uniqueness, regularity, and robustness theorems
for NNMs and SSMs, respectively. The conditions of these theorems
are computable solely from the spectrum of the linearized system.
Contrary to common expectation in vibration theory, however, the mathematical
conditions for NNMs and SSMs are more affected by the real part of
the spectrum, rather than the imaginary part (i.e., frequencies) of
the oscillations. Therefore, even weak damping should be carefully
considered and analyzed, rather than ignored, if one wishes to construct
robust SSMs for model reduction purposes. We illustrate our results
on simple, low-dimensional examples, and discuss the relevance of
our findings for model reduction. More detailed numerical examples
of higher-dimensional mechanical systems will be treated elsewhere.

\section{Set up}

Our study is motivated by, but not restricted to, $n$-degree of freedom
mechanical systems of the form
\begin{eqnarray}
M\ddot{q}+\left(C+G\right)\dot{q}+\left(K+B\right)q & = & F_{0}(q,\dot{q})+\epsilon F_{1}(q,\dot{q,}\Omega_{1}t,\ldots,\Omega_{k}t;\epsilon),\qquad0\leq\epsilon\ll1,\label{eq:system}\\
F_{0}(q,\dot{q}) & = & \mathcal{O}\left(\left|q\right|^{2},\left|q\right|\left|\dot{q}\right|,\left|\dot{q}\right|^{2}\right),
\end{eqnarray}
where $q=(q_{1},\ldots q_{n})\in U\subset\mathbb{R}^{n}$ is the vector
of generalized coordinates defined on an open set $U$; $M=M^{T}\in\mathbb{R}^{n\times n}$
is the positive definite mass matrix; $C=C^{T}\in\mathbb{R}^{n\times n}$
is a positive semi-definite damping matrix; $G=-G^{T}\in\mathbb{R}^{n\times n}$
is the gyroscopic matrix; $K=K{}^{T}\in\mathbb{R}^{n\times n}$ is
a positive semidefinite stiffness matrix; $B=-B^{T}\in\mathbb{R}^{n\times n}$
is the coefficient matrix of follower forces; the vector $F_{0}\in\mathbb{R}^{n}$
represents autonomous nonlinearities; and the vector $F_{1}\in\mathbb{R}^{n}$
denotes external forcing with the frequency vector $\Omega=\left(\Omega_{1},\ldots,\Omega_{k}\right)\in\mathbb{R}^{k}$
with $k\geq0.$ Note that $F_{1}(q,\dot{q,}\Omega_{1}t,\ldots,\Omega_{k}t)$
is not necessarily nonlinear, and hence can in principle be large
even when $\left|q\right|$ and $\left|\dot{q}\right|$ are small.
In the special case of $k=0$, the external forcing is autonomous,
while in the case of $k=1$, the external forcing is time-periodic.
For $k>1$, the external forcing is quasiperiodic if at least two
of the frequencies $\Omega_{j}$ are rationally incommensurate. We
assume both $F_{0}$ and $F_{1}$ to be of class $C^{r}$ in their
arguments, where $r$ is either a nonnegative integer, $\infty$,
or equal to $a$, with $C^{a}$ referring to analytic functions. In
short, we assume
\begin{equation}
r\in\mathbb{N}^{+}\cup\left\{ \infty,a\right\} .\label{eq:r-def}
\end{equation}

For $\epsilon=0$, system \eqref{eq:system} has an equilibrium point
at $q=0.$ Linear oscillations around this equilibrium point are governed
by the spectral properties of the linearized system on the left-hand
side of \eqref{eq:system}. Our main interest here is the relevance
of these linear oscillations for the dynamics of the full system \eqref{eq:system}.
A strict mathematical relationship between linear and nonlinear oscillations
can only be expected near the equilibrium (i..e, for small values
of $\left|q\right|$ and $\left|\dot{q}\right|$) and for small values
of the forcing parameter $\epsilon$. We seek to establish, however,
the existence of nonlinear sets of solutions near the equilibrium
that continue to extend to larger domains of the phase space and hence
exert a more global influence on the system dynamics.

After the change of variables $x_{1}=q$, $x_{2}=\dot{q}$, the evolution
of the vector $x=(x_{1},x_{2})\in\mathcal{U}=U\times\mathbb{R}^{n}$
is governed by the first-order differential equation
\begin{equation}
\dot{x}=Ax+f_{0}(x)+\epsilon f_{1}(x,\Omega t;\epsilon),\qquad f_{0}(x)=\mathcal{O}(\left|x\right|^{2}),\qquad0\leq\epsilon\ll1,\label{eq:1storder_system}
\end{equation}
with a constant matrix $A\in\mathbb{R}^{N\times N},$ and with the
class $C^{r}$ functions $f_{0}\colon\mathcal{U}\to\mathbb{R}^{N}$
and $f_{1}\colon\mathcal{U}\times\mathbb{T}^{k}\to\mathbb{R}^{N}$,
where $\mathbb{T}^{k}=S^{1}\times\ldots\times S^{1}$ is the $k$-dimensional
torus. 

As long as $A$, $f_{0}$ and $f_{1}$ are of the general form stated
above, their specific form will be unimportant for our forthcoming
discussion, as we state all results in terms of the ODE \eqref{eq:1storder_system}.
If, however, the ODE \eqref{eq:1storder_system} arises from the mechanical
system \eqref{eq:system}, then we specifically have $N=2n$ and 
\[
A=\left(\begin{array}{cc}
0 & I\\
-M^{-1}(K+B) & -M^{-1}\left(C+G\right)
\end{array}\right),
\]
\[
f_{0}(x)=\left(\begin{array}{c}
0\\
M^{-1}F_{0}(x_{1},x_{2})
\end{array}\right),\qquad f_{1}(x,\Omega t)=\left(\begin{array}{c}
0\\
M^{-1}F_{1}(x_{1},x_{2},\Omega_{1}t,\ldots,\Omega_{k}t)
\end{array}\right).
\]

\section{Linear spectral geometry: Eigenspaces, normal modes, spectral subspaces,
and invariant manifolds}

\subsection{Eigenvalues}

The linear, unperturbed part of system \eqref{eq:1storder_system}
is 
\begin{equation}
\dot{x}=Ax.\label{eq:linearization}
\end{equation}
The matrix $A$ has $N$ eigenvalues $\lambda_{j}\in\mathbb{C}$,
$j=1,\ldots,N,$ with multiplicities counted. We order these eigenvalues
so that their real parts form a decreasing sequence under increasing
$j:$
\begin{equation}
\mathrm{Re}\lambda_{N}\leq\mathrm{Re}\lambda_{N-1}\leq\ldots\ldots\leq\mathrm{Re}\lambda_{1}.\label{eq:eigenvalue_ordering}
\end{equation}
We denote the algebraic multiplicity of $\lambda_{j}$ (i.e., its
multiplicity as a root of the characteristic equation of $A$) by
$\mathrm{alg}\,(\lambda_{j})$, and its geometric multiplicity (i.e.,
the number of independent eigenvectors corresponding to $\lambda_{j}$)
by $\mathrm{geo}\,(\lambda_{j}).$ We recall that $A$ is called semisimple
if $\mathrm{alg}\,(\lambda_{j})=\mathrm{geo}\,(\lambda_{j})$ holds
for all $\lambda_{j}$. This is always the case if all eigenvalues
are distinct or $A$ is symmetric. When $A$ is not semisimple, then
some of its eigenvalues satisfy $\mathrm{alg}\,(\lambda_{j})>\mathrm{geo}\,(\lambda_{j}),$
leading to nontrivial blocks in the Jordan decomposition of $A$.
A good reference for this and other forthcoming aspects of linear
dynamical systems is Hirsch, Smale and Devaney \cite{hirsch13}.

\subsection{Eigenspaces}

For each distinct eigenvalue $\lambda_{j}$, there exists a real \emph{eigenspace}
$E_{j}\subset\mathrm{\mathbb{R}}^{N}$ spanned by the imaginary and
real parts of the corresponding eigenvectors and generalized eigenvectors
of $A$. We have $\dim E_{j}=\mathrm{alg}\,(\lambda_{j})$ in case
$\mathrm{Im\,}\lambda_{j}=0$, while we have $\dim E_{j}=2\times\mathrm{alg}\,(\lambda_{j})$
in case $\mathrm{Im\,}\lambda_{j}\neq0.$ In the latter case, $E_{j}\equiv E_{j+1}$
because $\lambda_{j}=\bar{\lambda}_{j+1}$. That is, the real eigenspaces
associated with each of two complex conjugate eigenvalues coincide
with each other.

An eigenspace $E_{j}$ also represents an invariant subspace for the
linearized system \eqref{eq:linearization}, filled with trajectories
of this system corresponding to the eigenvalue $\lambda_{j}$. Specifically,
we have 
\begin{equation}
E_{j}=\underset{t\in\mathbb{R}}{\mathrm{span}}\left\{ \,\,\,\,e^{\mathrm{Re}\lambda_{j}t}\cos\left[\mathrm{Im}\left(\lambda_{j}\right)t\right]\sum_{\alpha=1}^{\mathrm{alg}\,(\lambda_{j})}a_{j}^{\alpha}t^{\alpha-1};\,\,\,\,\,e^{\mathrm{Re}\lambda_{j}t}\sin\left[\mathrm{Im}\left(\lambda_{j}\right)t\right]\sum_{\alpha=1}^{\mathrm{alg}\,(\lambda_{j})}b_{j}^{\alpha}t^{\alpha-1}\,\,\,\,\right\} \label{eq:eigenspace0}
\end{equation}
for appropriate real vectors $a_{j}^{\alpha},b_{j}^{\alpha}\in\mathbb{R}^{N}$.
In the generic case, $\lambda_{j}$ is a simple real or simple complex
eigenvalue, in which case $E_{j}$ is one- or two-dimensional, respectively.

\subsection{Linear normal modes}

The classic definition of a\emph{ linear normal mode} refers to a
periodic solution of the linear system \eqref{eq:linearization},
arising from an eigenvalue $\lambda_{j}$ with $\mathrm{Re}\lambda_{j}=0$
and $\mathrm{alg}\,(\lambda_{j})=\mathrm{geo}\,(\lambda_{j})$. In
this case, normal modes fill the full eigenspace of $\lambda_{j}$,
i.e., we have 
\begin{equation}
E_{j}=\underset{t\in\mathbb{R}}{\mathrm{span}}\left\{ \,\,\,a_{j}^{1}\cos\left[\mathrm{Im}\left(\lambda_{j}\right)t\right],\ldots,a_{j}^{\mathrm{alg}\,(\lambda_{j})}\cos\left[\mathrm{Im}\left(\lambda_{j}\right)t\right];\,\,b_{j}^{1}\sin\left[\mathrm{Im}\left(\lambda_{j}\right)t\right],\ldots,b_{j}^{\mathrm{alg}\,(\lambda_{j})}\sin\left[\mathrm{Im}\left(\lambda_{j}\right)t\right]\,\,\,\right\} ,\label{eq:eigenspace}
\end{equation}
with the vectors $a_{j}^{\alpha},b_{j}^{\alpha}$ appearing in \eqref{eq:eigenspace0},
and with $\dim E_{j}=2\times\mathrm{alg}\,(\lambda_{j})$. In case
of a linear mechanical system without symmetries, the eigenvalues
$\lambda_{j}=i\omega_{j}$ generating normal modes are typically simple.
In that case, we have $\mathrm{alg}\,(\lambda_{j})=\mathrm{geo}\,(\lambda_{j})=1$,
and $\dim E_{j}=2$. The normal mode family of period $T_{j}=2\pi/\omega_{j}$
then spans the two-dimensional invariant plane $E_{j}$ in the phase
space of the linear system \eqref{eq:linearization}

The fixed point $x=0$ of the linear system \eqref{eq:linearization}
can also be considered as a singular normal mode when viewed as a
periodic motion of arbitrary period. This \emph{trivial normal mode},
however, is isolated and does not form a family spanning a nontrivial
subspace. Yet, this representation of the fixed point as a periodic
orbit becomes useful when we seek its continuation under small forcing
$(\epsilon>0$) in the perturbed equation \eqref{eq:1storder_system}.
The fixed point will generally not survive, but a unique periodic
or quasiperiodic orbit mimicking the stability of the fixed point
will often exist, as we discuss below.

\subsection{Spectral subspaces}

By linearity, a subspace spanned by any combination of eigenspaces
is also invariant under the dynamics of the linear system \eqref{eq:linearization}.
Specifically, a \emph{spectral subspace
\begin{equation}
E_{j_{1},\ldots,j_{q}}=E_{j_{1}}\oplus E_{j_{2}}\oplus\ldots\oplus E_{j_{q}}=\left\{ v\in\mathbb{R}^{N}\,:\,v=\sum_{i=1}^{q}v_{i},\quad v_{i}\in E_{j_{i}},\quad E_{j_{l}}\neq E_{j_{k}},\,\,\,\,k,l=1,\ldots,q\,\,\right\} ,\label{eq:spectral subspace}
\end{equation}
}with\emph{ $\oplus$} denoting the direct sum of vector spaces, is
an invariant subspace of system\emph{ }\eqref{eq:linearization}.
The definition \eqref{eq:spectral subspace} avoids double-counting
the same real eigenspace corresponding to complex conjugate eigenvalues.
Also, by definition, any single eigenspace $E_{j}$ is also a spectral
subspace. 

Classic examples of spectral subspaces include the \emph{stable subspace}
$E^{s}$, the \emph{unstable subspace} $E^{u}$ and the \emph{center
subspace} $E^{c}$. In the presence of $n_{s}$, $n_{u}$ and $n_{c}$
eigenvalues with negative, positive and zero real parts, respectively,
these classic spectral subspaces are defined as 
\begin{eqnarray}
E^{s} & = & \left\{ v\in\mathbb{R}^{N}\,:\,v=\sum_{i=1}^{n_{s}}v_{i},\quad v_{i}\in E_{j_{i}},\quad\mathrm{Re}\lambda_{j_{i}}<0,\quad i=1,\ldots,n_{s}\right\} ,\nonumber \\
E^{u} & = & \left\{ v\in\mathbb{R}^{N}\,:\,v=\sum_{i=1}^{n_{u}}v_{i},\quad v_{i}\in E_{j_{i}},\quad\mathrm{Re}\lambda_{j_{i}}>0,\quad i=1,\ldots,n_{u}\right\} ,\label{eq:invariant subspaces}\\
E^{c} & = & \left\{ v\in\mathbb{R}^{N}\,:\,v=\sum_{i=1}^{n_{c}}v_{i},\quad v_{i}\in E_{j_{i}},\quad\mathrm{Re}\lambda_{j_{i}}=0,\quad i=1,\ldots,n_{c}\right\} .\nonumber 
\end{eqnarray}

Linearized oscillatory systems in mechanics often have only decaying
solutions due to the presence of damping on an otherwise conservative
system of oscillators. In that case, $E^{s}=\mathbb{R}^{N}$ and $E^{u}=E^{c}=\emptyset.$
If, in addition, all eigenvalues $\lambda_{j}$ are distinct and complex,
then the minimal spectral subspaces are formed by the two-dimensional
eigenspaces $E_{j}$. Again, any direct sum of these two-dimensional
eigenspaces is a spectral subspace by the above definition.

\subsection{Invariant manifolds in the linearized system\label{sub:Spectral-submanifolds-in-linear-systems}}

For simplicity, we assume here that the matrix $A$ has only distinct
eigenvalues. We make this assumption here only for ease of exposition,
and will drop it later in our results for the full nonlinear system.

In its eigenbasis, $A$ is then diagonal and the linearized system
\eqref{eq:linearization} can be written in the complexified form
\begin{equation}
\dot{y}=\Lambda y,\qquad\Lambda=\mathrm{diag}\left(\lambda_{1},\ldots,\lambda_{N}\right),\label{eq:diagonalized linearization}
\end{equation}
where $y\in\mathbb{C}^{N}$ is a complex vector, with its $j^{th}$
coordinate $y_{j}$ denoting a coordinate along the (generally complex)
eigenvector $e_{j}$ of $A$. Complexified equivalents of all real
eigenspaces $E_{j}$ and spectral subspaces $E_{j_{1},\ldots,j_{q}}$
are again invariant subspaces for the linearized dynamics \eqref{eq:diagonalized linearization}.
As invariant manifolds, not only are all these subspaces infinitely
many times differentiable but also analytic. Indeed, their coordinate
representations are given by the analytic graphs $y_{l}=f_{l}(y_{j_{1}},\ldots,y_{j_{q}})\equiv0$,
for all $l\notin\left\{ j_{1},\ldots,j_{q}\right\} $, over any spectral
subspace $E_{j_{1},\ldots,j_{q}}$. 

There are, however, generally infinitely many other invariant manifolds
in the linearized system \eqref{eq:diagonalized linearization} that
are also graphs over $E_{j_{1},\ldots,j_{q}}$ and are tangent $E_{j_{1},\ldots,j_{q}}$
at the origin. Indeed, as we show in Appendix \ref{sub:Uniqueness-and-analyticity},
along \emph{any} codimension-one surface $\Gamma\subset E_{j_{1},\ldots,j_{q}}$,
intersected transversely by the linear vector field \eqref{eq:diagonalized linearization}
within $E_{j_{1},\ldots,j_{q}}$, we can prescribe the $y_{l}$ coordinates
of an invariant manifold via arbitrary smooth functions $y_{l}\vert_{\Gamma}=f_{l}^{0}(\Gamma)$
with $l\notin\left\{ j_{1},\ldots,j_{q}\right\} $, and obtain (under
non-resonance conditions) a unique manifold satisfying this boundary
condition. For two-dimensional systems, this arbitrariness in the
boundary conditions leads to a one-parameter family of invariant surfaces
(see. Fig. \ref{fig:nonuniqueness}a). In the multi-dimensional case,
illustrated in Fig. \ref{fig:nonuniqueness}b, there is a substantially
higher degree of non-uniqueness for invariant manifolds tangent to
individual spectral subspaces. Indeed, both the choice of the codimension-one
boundary surface $\Gamma$ and the choice of the boundary values $f_{l}^{0}(\Gamma)$
of the invariant manifold are arbitrary, as long as $\Gamma$ is transverse
to the linear vector field.

\begin{figure}[H]
\centering{}\includegraphics[width=0.7\textwidth]{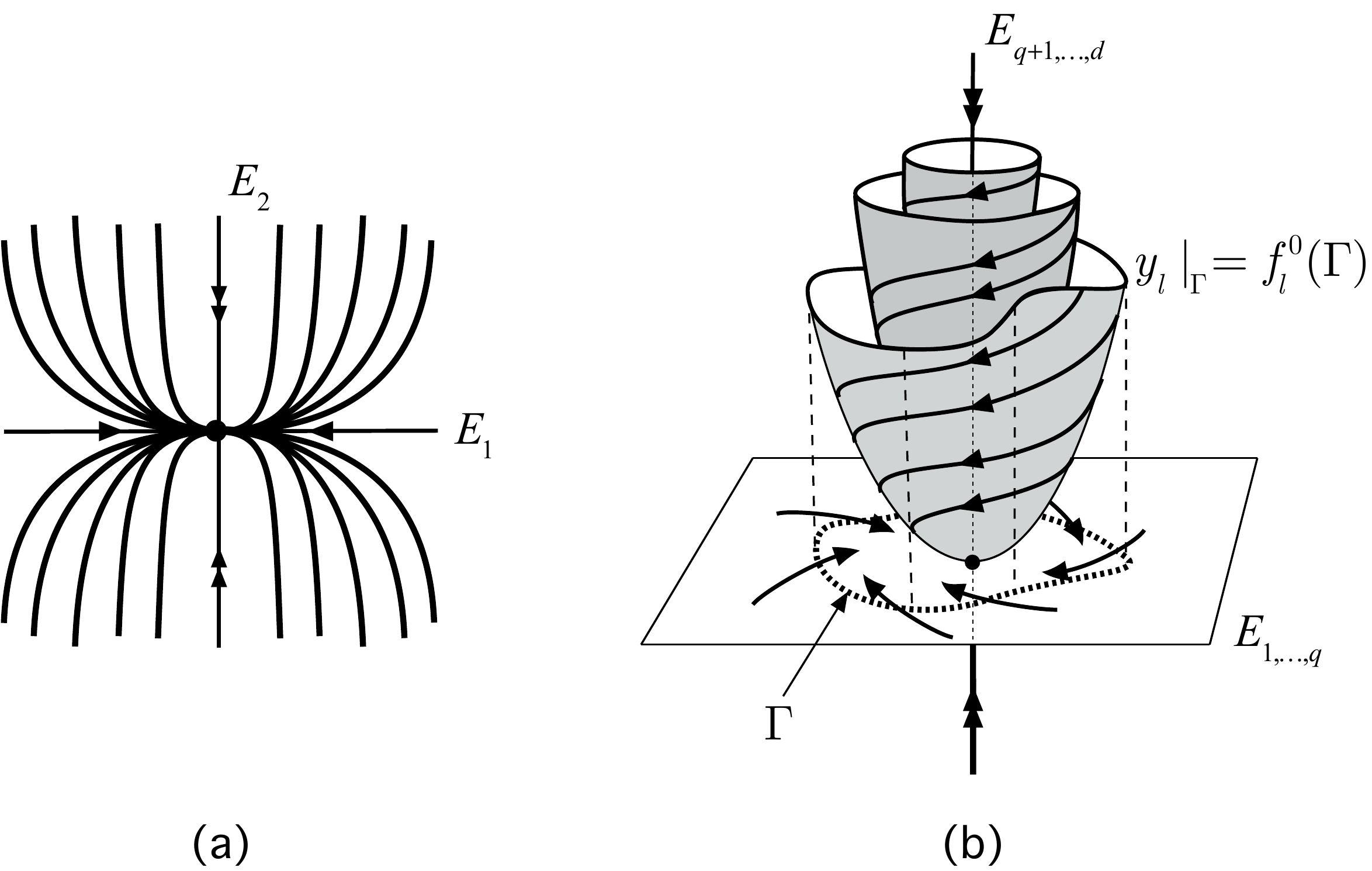}\caption{(a) Non-uniqueness of invariant manifolds tangent to the slower-decaying
spectral subspace of a planar, linear dynamical system. Note the uniqueness
of the invariant manifold tangent to the faster-decaying spectral
subspace (b) Non-uniqueness of invariant manifolds tangent to the
direct product $E_{1,\ldots,q}$ of $q$ slowest-decaying spectral
subspaces of a higher-dimensional, linear dynamical system.\emph{
}Under appropriate nonresonance conditions (cf. Appendix \ref{sub:Uniqueness-and-analyticity}),
any codimension-one boundary surface $\Gamma$ transverse to the flow
within $E_{1,\ldots,q}$ yields an invariant manifold tangent to $E_{1,\ldots,q}$
at the fixed point, for \emph{any} choice of the smooth functions
$y_{l}=f_{l}^{0}(\Gamma),$ with $l\protect\notin\left\{ j_{1},\ldots,j_{q}\right\} $.
Again, note the uniqueness of the invariant manifold tangent to the
spectral subspace of the remaining faster-decaying modes.\label{fig:nonuniqueness}}
\end{figure}

A subset of these infinitely many solutions is simple to write down
in the case of underdamped mechanical vibrations whereby we have $\mathrm{Im}\lambda_{j}\neq0$
for all eigenvalues. Passing to amplitude-phase variables $(r_{j},\varphi_{j})$
by letting $(y_{j},\bar{y}_{j})\equiv(y_{j},y_{j+1})=r_{j}e^{i\varphi_{j}}$,
we can re-write system \eqref{eq:diagonalized linearization} in the
simple amplitude-phase form
\[
\dot{r}_{j}=-\mathrm{Re}\lambda_{j}\,r,\quad\dot{\varphi}_{j}=\mathrm{Im}\lambda_{j},\quad j=1,\ldots,n=N/2,
\]
with $n$ denoting the number of degrees of freedom in the system
\eqref{eq:system}. In this case, a family of invariant manifolds
tangent to the spectral subspace $E_{j_{1},\ldots,j_{q}}$ is given
explicitly by the equations 
\begin{eqnarray}
r_{l} & = & f_{r_{l}}(r_{j_{1}},\varphi_{j_{1}},\ldots,r_{j_{q}},\varphi_{j_{q}}):=\sum_{i=1}^{q}C_{l}^{j_{i}}r_{j_{i}}^{\frac{\mathrm{Re}\lambda_{l}}{\mathrm{Re}\lambda_{j_{i}}}},\nonumber \\
\phi_{l} & = & f_{\varphi_{l}}(r_{j_{1}},\varphi_{j_{1}},\ldots,r_{j_{q}},\varphi_{j_{q}}):=D_{l}^{j_{i}}+\frac{\mathrm{Im}\lambda_{l}}{\mathrm{Im}\lambda_{j_{1}}}\varphi_{j_{1}},\label{eq:flatsolution}
\end{eqnarray}
for all $l\notin\left\{ j_{1},\ldots,j_{q}\right\} $, with $C_{l}^{j_{i}}\in\mathbb{R}$
and $D_{l}^{j_{i}}\in[0,2\pi)$ denoting arbitrary constants. Under
the nonresonance conditions $\lambda_{l}/\lambda_{j_{i}}\not\in\mathbb{N}^{+}$
, if
\begin{equation}
\mathrm{Re}\lambda_{l}<\mathrm{Re}\lambda_{j_{i}}<0,\quad i=1,\ldots,q,\quad l\notin\left\{ j_{1},\ldots,j_{q}\right\} \label{eq:dominance}
\end{equation}
holds, then any nonzero solution \eqref{eq:flatsolution} has only
finitely many continuous derivatives at the origin. The only exceptions
are the identically zero solutions for which $C_{l}^{j_{i}}=0$ holds
for all $j_{i}$ and $l$ values, giving $f_{r_{l}}(r_{j_{1}},\varphi_{j_{1}},\ldots,r_{j_{q}},\varphi_{j_{q}})\equiv0$.
These zero solutions are, in fact the unique smoothest $(C^{\infty}$
and even $C^{a}$) member of the solution family \eqref{eq:flatsolution},
representing the invariant spectral subspace $E_{j_{1},\ldots,j_{q}}$
itself. 

Condition \eqref{eq:dominance}, however, never holds in the case
of $E_{j_{1},\ldots,j_{q}}=E_{N-q+1,\ldots,N}$, i.e., when the invariant
manifold is sought as a graph over the spectral subspace of the $q$
fastest decaying modes. In this case, $\mathrm{Re}\lambda_{j_{i}}<\mathrm{Re}\lambda_{l}<0$
hold for all indices involved, and the only differentiable member
of the solution family \eqref{eq:flatsolution} at the origin is $f_{l}(y_{N-q+1},\ldots,y_{N})\equiv0.$
This is unique differentiable invariant manifold over $E_{N-q+1,\ldots,N}$
also happens to be analytic. The uniqueness of $E_{N-q+1,\ldots,N}$
as a smooth invariant manifold with the prescribed tangency property
does not just hold within the special solution family \eqref{eq:flatsolution}.
Indeed, the classic strong stable manifold theorem (see, e.g., Hirsch,
Pugh and Shub \cite{hirsch77}) applied to the linear system \eqref{eq:diagonalized linearization}
implies uniqueness for $E_{N-q+1,\ldots,N}$ among \emph{all} invariant
manifolds tangent to $E_{N-q+1,\ldots,N}$ at the origin. This uniqueness
of fast invariant manifolds is also illustrated in Fig. \ref{fig:nonuniqueness}a
for the two-dimensional case, and in Fig. \ref{fig:nonuniqueness}b
for the multi-dimensional case.

In summary, under appropriate nonresonance assumptions on the eigenvalues,
there are infinitely many Shaw--Pierre-type invariant manifolds tangent
to any non-fast spectral subspace $E_{j_{1},\ldots,j_{q}}$ at the
origin of the linearized system \eqref{eq:linearization}. Clearly,
one cannot expect such manifolds to be unique in the nonlinear context
studied by Shaw and Pierre \cite{shaw93} either. Thus, the common
assumption in the nonlinear normal modes literature, that invariant
manifolds tangent to eigenspaces will uniquely emerge from approximate
operational procedures, is generally unjustified. 

Observe, however, that despite the non-uniqueness of invariant manifolds
tangent to a non-fast spectral subspace $E_{j_{1},\ldots,j_{q}}$
at the origin of the linear system \eqref{eq:diagonalized linearization},
the flat boundary condition $y_{l}=f_{l}(y_{j_{1}},\ldots,y_{j_{q}})\equiv0$,
with $l\notin\left\{ j_{1},\ldots,j_{q}\right\} $, yields the unique
\emph{analytic invariant manifold}, $E_{j_{1},\ldots,j_{q}}$, provided
that the nonresonance conditions $\lambda_{p}/\lambda_{j_{i}}\not\in\mathbb{N^{+}}$
hold (see Appendix \ref{sub:Uniqueness-and-analyticity} for details.)
This gives hope that perhaps there is a unique analytic (or at least
a unique smoothest) continuation of spectral subspaces of the linearized
system to locally smoothest manifolds in the nonlinear system \eqref{eq:1storder_system}
near the origin. As we show in later sections, this expectation turns
out to be justified under certain conditions.

\subsection{Spectral quotients}

As we observed above, nontrivial solutions of the form \eqref{eq:flatsolution}
have only a finite number of continuous derivatives at the origin.
Namely, if the graph is constructed over the spectral subspace $E_{j_{1},\ldots,j_{q}}$,
then only $\mathrm{Int}\left[\mathrm{Re}\lambda_{l}/\mathrm{Re}\lambda_{j_{i}}\right]$
continuous derivatives exist for the $r_{l}$ coordinate function,
with $\mathrm{Int}\left[\,\cdot\,\right]$ denoting the integer part
of a real number. 

The smoothest non-flat invariant graphs in the family \eqref{eq:flatsolution},
therefore, satisfy 
\begin{eqnarray*}
r_{L} & = & C_{L}^{j_{I}}r_{j_{I}}^{\frac{\mathrm{Re}\lambda_{L}}{\mathrm{Re}\lambda_{j_{I}}}},\quad L=\arg\max_{l\notin\left\{ j_{1},\ldots,j_{q}\right\} }\left|\mathrm{Re}\lambda_{L}\right|,\quad I=\arg\min_{i\in\left\{ 1,\ldots,q\right\} }\left|\mathrm{Re}\lambda_{j_{i}}\right|,\\
r_{l} & \equiv & 0,\quad l\neq L,
\end{eqnarray*}
with their degree of smoothness at the origin equal to $\mathrm{Int}\left[\mathrm{Re}_{L}/\mathrm{Re}\lambda_{j_{I}}\right]$.
This is the maximal degree of smoothness that any non-flat member
of the solution family \eqref{eq:flatsolution} can attain. The only
smoother invariant graph over $E_{j_{1},\ldots,j_{q}}$ in the graph
family \eqref{eq:flatsolution} is the subspace $E_{j_{1},\ldots,j_{q}}$
itself. 

This maximal smoothness of the invariant graphs \eqref{eq:flatsolution}
is purely determined by the ratio of the fastest decay exponent outside
$E_{j_{1},\ldots,j_{q}}$ to the slowest decay exponent within $E_{j_{1},\ldots,j_{q}}$.
For later purposes, we now give a formal definition of the integer
part of this ratio for any spectral subspace $E$ of the operator
$A$. We also define another version of the same quotient, with the
numerator replaced by the fastest decay exponent in the whole spectrum
of $A$. Our notation for the full spectrum of $A$ is $\mathrm{Spect(}A),$
whereas we denote the spectrum of the restriction of $A$ to its spectral
subspace $E$ by $\mathrm{Spect(}A\vert_{E}).$
\begin{defn}
\label{def:spectral quotients}For any spectral subspace of the linear
operator$A$, we define the \textbf{relative spectral quotient} $\sigma(E)$
and the \textbf{absolute spectral quotient} $\Sigma(E)$ as 
\begin{eqnarray}
\sigma(E) & = & \mathrm{Int}\,\left[\frac{\min_{\lambda\in\mathrm{Spect(}A)-\mathrm{Spect(}A\vert_{E})}\mathrm{Re}\lambda}{\max_{\lambda\in\mathrm{Spect(}A\vert_{E})}\mathrm{Re}\lambda}\right],\label{eq:relative_sigma}\\
\Sigma(E) & =\mathrm{Int}\, & \left[\frac{\min_{\lambda\in\mathrm{Spect(}A)}\mathrm{Re}\lambda}{\max_{\lambda\in\mathrm{Spect(}A\vert_{E})}\mathrm{Re}\lambda}\right].\label{eq:absolute_sigma}
\end{eqnarray}
These spectral quotients will play a major role in later sections
when we discuss the existence and uniqueness of nonlinear continuations
of invariant manifolds of the linearized system. 
\end{defn}

\section{Nonlinear spectral geometry: Nonlinear normal modes and spectral
submanifolds}

The fundamental assumption of nonlinear modal analysis is that appropriate
generalizations of invariant manifolds of the linearized system persist
under the full system \eqref{eq:1storder_system} (see, e.g., Vakakis
\cite{vakakis01}, Kerschen et al. \cite{kerschen09}, Peeters et
al. \cite{peters09}, and Avramov and Mikhlin \cite{avramov10,avramov13}
for reviews). 

The classic definition of Rosenberg \cite{rosenberg62} for autonomous,
conservative systems states that nonlinear normal modes are synchronous
periodic orbits, i.e., periodic motions that reach their extrema along
all modal coordinate directions at the same time. A useful relaxation
of this concept allows for general (not necessarily synchronous) periodic
orbits in autonomous systems (see, e.g., Peeters et al. \cite{peters09}).

Here we relax Rosenberg's definition even further for general dissipative
systems, allowing a nonlinear normal to be a recurrent motion with
a discrete Fourier spectrum of $f$ frequencies.\footnote{Recurrent motions are typical in conservative systems with compact
energy surfaces. Thus, recurrence by itself can only distinguish nonlinear
normal modes in dissipative systems.} If $f>1$ and the frequencies of the motion are rationally independent,
then the motion is quasi-periodic and forms a non-compact set in the
phase space. To this end, we use the closure of such a trajectory
in our normal mode definition (with the closure including the trajectory
as well as all its limit points). Specifically, the closure of a periodic
orbit is just the periodic orbit itself, while the closure of a quasiperiodic
orbit contains further points outside the trajectory, forming an invariant
torus densely filled by the trajectory.
\begin{defn}
\emph{\label{def:NNM}A nonlinear normal mode (NNM) }is the closure
of a multi-frequency solution
\[
x(t)=\sum_{\left|m\right|=1}^{\infty}x_{m}e^{i\left\langle m,\Omega\right\rangle t},\qquad m\in\mathbb{N}^{f},\quad\Omega\in\mathbb{R}^{f},
\]
of the nonlinear system \eqref{eq:1storder_system}. Here $f\in\mathbb{N}$
is the number of frequencies; the vector $m$ is a multi-index of
$f$ nonnegative integers; $x_{m}\in\mathbb{C}^{n}$ are the complex
Fourier amplitudes of the real solution $x(t)$ with respect to the
frequencies in the frequency vector $\Omega=\left(\Omega_{1},\ldots,\Omega_{f}\right)$.
Special cases of NNMs include (see Fig. \ref{fig:NNMs=000026SSMs}):\end{defn}
\begin{description}
\item [{\emph{(1)}}] \emph{trivial NNM ($f=0$)}: a fixed point 
\item [{\emph{(2)}}] \emph{periodic NNM }(either $f=1$, or $f>1$ and
the elements of $\Omega$ are rationally commensurate): a periodic
orbit
\item [{(3)}] \emph{quasiperiodic NNM} ($f>1$ and the elements of $\Omega$
are rationally incommensurate): an $f$-dimensional invariant torus
\end{description}

\begin{figure}[H]
\begin{centering}
\includegraphics[width=0.8\textwidth]{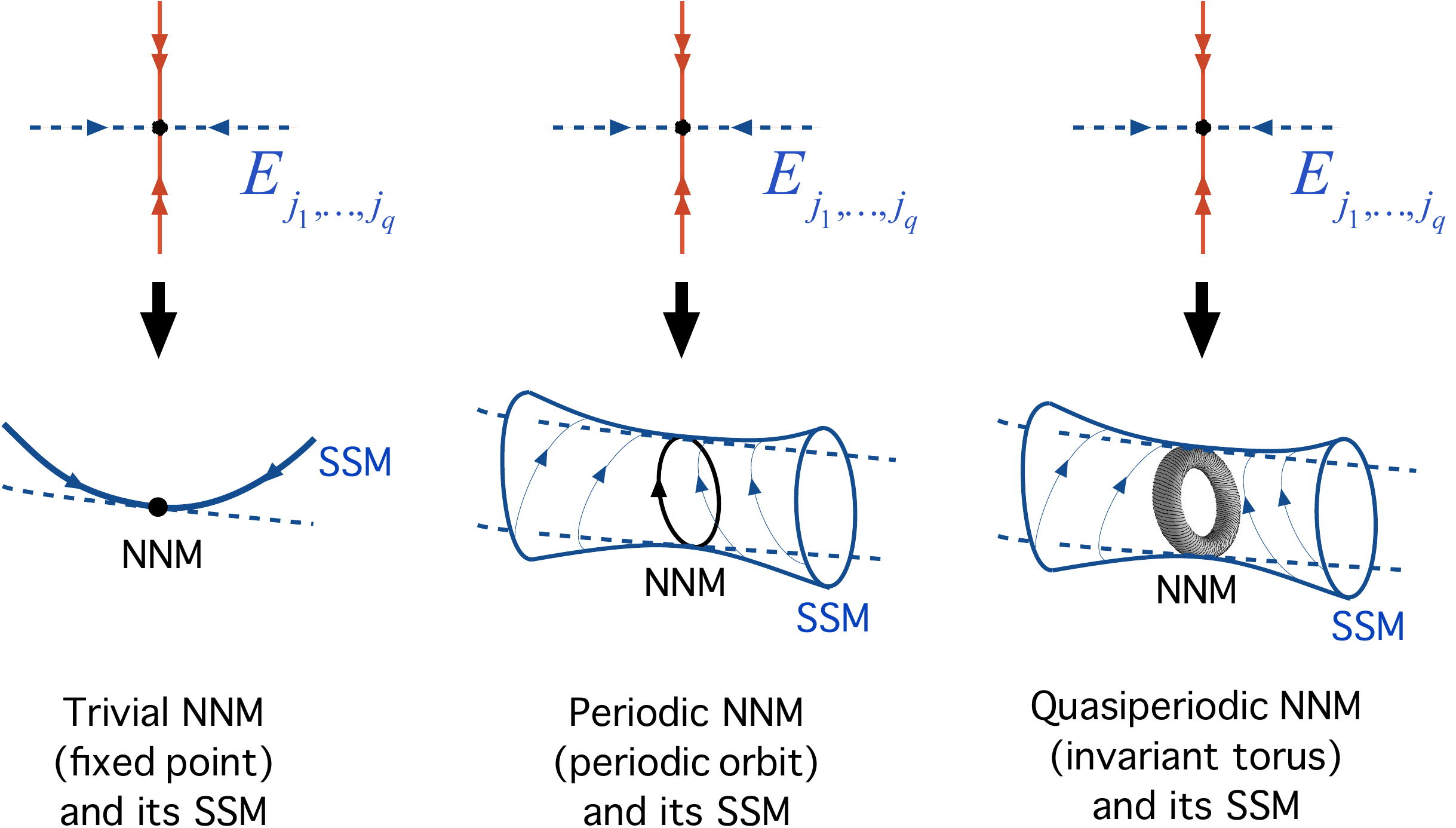}
\par\end{centering}

\caption{Schematics of the three main types of NNMs (trivial, periodic and
quasiperiodic) and their corresponding SSMs (autonomous, periodic
and quasiperiodic). In all cases, the NNM are, or are born out of,
perturbations of a fixed point. The SSMs are always tangent to a sub-bundle
along the NNM whose fibers are close to a specific spectral subspace
$E_{j_{1},\ldots,j_{q}}$ of the linearized system.\label{fig:NNMs=000026SSMs}}

\end{figure}

A further expectation in the nonlinear vibrations literature--put
forward first by Shaw and Pierre \cite{shaw93} in its simplest form,
then extended by Pescheck et al. \cite{peschek01}, Shaw, Peschek
and Pierre\cite{shaw99}, Jiang, Pierre and Shaw \cite{jiang05}--is
that an arbitrary spectral subspace $E_{j_{1},\ldots,j_{q}}$ of the
$x=0$ fixed point will also persist under the addition of nonlinear
and time-dependent terms in system \eqref{eq:invariant subspaces}.
This would lead to a nonlinear continuation of the spectral subspace
$E_{j_{1},\ldots,j_{q}}$ into an invariant manifold $W_{j_{1},\ldots,j_{q}}(\mathcal{\mathcal{N}})$
along $\mathcal{N}.$ While Shaw and Pierre \cite{shaw93} calls such
a $W_{j_{1},\ldots,j_{q}}(\mathcal{\mathcal{N}})$ a nonlinear normal
mode, the dynamics in $W_{j_{1},\ldots,j_{q}}(\mathcal{\mathcal{N}})$
will not inherit the forward- and backward-bounded, recurrent nature
of linear normal modes even in the simplest dissipative examples.
To make this distinction from classic normal modes clear, we refer
here to $W_{j_{1},\ldots,j_{q}}(\mathcal{\mathcal{N}})$ as a spectral
submanifold. 
\begin{defn}
\emph{\label{def:SSM}A spectral submanifold (SSM)} of a NNM, $\mathcal{\mathcal{N},}$
is an invariant manifold $W(\mathcal{\mathcal{N}})$ of system \eqref{eq:1storder_system}
such that\end{defn}
\begin{description}
\item [{(i)}] $W(\mathcal{\mathcal{N}})$ is a subbundle of the normal
bundle $N\mathcal{N}$ of $\mathcal{\mathcal{N}}$, satisfying $\dim W(\mathcal{\mathcal{N}})=\dim E+\dim\mathcal{\mathcal{N}}$
for some spectral subspace $E$ of the operator $A$. 
\item [{(ii)}] The fibers of the bundle $W(\mathcal{\mathcal{N}})$ perturb
smoothly from the spectral subspace $E$ of the linearized system
\eqref{eq:linearization} under the addition of the nonlinear and
$\mathcal{O}(\epsilon)$ terms in system \eqref{eq:1storder_system}.
\item [{(iii)}] $W(\mathcal{\mathcal{N}})$ has strictly more continuous
derivatives along $\mathcal{\mathcal{N}}$ than any other invariant
manifold satisfying (i) and (ii).
\end{description}
More specifically, in the case of zero external forcing ($\epsilon=0$),
an SSM is the smoothest invariant manifold $W(0)$ out of all invariant
manifolds that are tangent to a spectral submanifold $E$ at $x=0$
and have the same dimension as $E$. In the case of nonzero external
forcing $(\epsilon\neq0)$, an SSM is the smoothest invariant manifold
$W(\mathcal{N})$ out of all invariant manifolds that are $\mathcal{O}(\epsilon$)
$C^{1}$-close to the set $\mathcal{N}\times E$ along $\mathcal{N}$
and have the same dimension as $\mathcal{N}\times E$ does.

To be clear, there is no a priori guarantee that a unique smoothest
member in a family of surfaces satisfying (i) and (ii) of Definition
\ref{def:SSM} actually exists. Indeed, no smooth surface might exist,
or those that exist may be equally smooth. We will need to derive
conditions under which SSMs are unique and hence well-defined in the
sense of Definition \ref{def:SSM}.

Special cases of SSMs include \textbf{(}see Fig. \ref{fig:NNMs=000026SSMs}):

\emph{(1) autonomous SSM ($f=0$)}: nonlinear continuations of spectral
submanifolds discussed for linear systems in Section \ref{sub:Spectral-submanifolds-in-linear-systems}.

\emph{(2) periodic SSM }(either $f=1$, or $f>1$ and the elements
of $\Omega$ are rationally commensurate): a three-dimensional invariant
manifold tangent to a spectral subbundle along a hyperbolic periodic
orbit

(3) \emph{quasiperiodic SSM} ($f>1$ and the elements of $\Omega$
are rationally incommensurate): an invariant manifold tangent to a
spectral subbundle of a hyperbolic invariant torus.

Classic examples of autonomous SSMs include the stable manifold $W^{s}(\mathcal{N})$
and the unstable manifold $W^{u}(\mathcal{N})$ of a fixed point $\mathcal{N}$
(i.e., of a trivial NNM). Classic examples of non-autonomous SSMs
include the stable manifold $W^{s}(\mathcal{N})$ and the unstable
manifold $W^{u}(\mathcal{N})$ of a periodic or quasiperiodic orbit
$\mathcal{N}$. The SSMs of interest here are submanifolds of $W^{s}(\mathcal{N})$
that perturb smoothly from spectral subspaces within $\mathcal{N}\times E^{s}$.
The construction of these surfaces has been the main question in the
nonlinear modal analysis of autonomous and non-autonomous systems,
to be discussed in detail in our Theorems \ref{theo:SSM unforced}
and \ref{theo:SSM forced} below.

There is a clear geometric distinction between our NNM definition
(a generalization of the normal mode concept of Rosenberg) and our
SSM definition (a generalization of the normal-mode concept of Shaw
and Pierre, with the highest smoothness requirement added). Both concepts
are helpful, but refer to highly different dynamical structures in
dissipative dynamical systems.

\section{Existence and uniqueness of NNMs\label{sec:NNM results}}

As mentioned before, the survival of the trivial NNMs in the form
of a nearby perturbed solution in system \eqref{eq:1storder_system}
is broadly expected in the nonlinear normal modes literature. These
perturbed NNMs are routinely sought via formal asymptotic expansions
with various a priori postulated time scales (see, e.g., Nayfeh \cite{nayfeh04}
for a survey of such intuitive methods). There is generally limited
concern for the the validity of these formal approximations (see Verhulst
\cite{verhulst15} for a discussion). Formal computability of the
first few terms of the assumed asymptotic expansion for NNMs, however,
does not imply that the targeted structure actually exists, as we
discussed in the Introduction. 

Here, we would like to fill this conceptual gap by clarifying the
existence and uniqueness of NNMs using classical invariant manifold
theory. The same theory also allows us to conclude the existence of
a special SSM, the stable manifold of the NNM. Here we only consider
damped mechanical vibrations for which 
\begin{equation}
\mathrm{Re}\lambda_{j}<0,\quad j=1,\ldots,N\label{eq:negative spectrum}
\end{equation}
holds in the linearized system\eqref{eq:linearization}. This assumption
ensures that we are in the dissipative setting in which our NNM and
SSM definitions are meaningful.

\subsection{Trivial NNM under autonomous external forcing ($k=0$)}

For time-independent external forcing, \eqref{eq:1storder_system}
remains autonomous even under the inclusion of the remaining $\mathcal{O}(\epsilon)$
forcing terms. Because these autonomous forcing terms are not assumed
to vanish at $x=0,$ the full system will generally no longer have
a fixed point at $x=0.$ The following theorem nevertheless guarantees
the existence of a nearby trivial NNM with spectral properties mimicking
that of the origin. 
\begin{thm}
\label{theo:unforced NNM}{[}\textbf{Existence, uniqueness and persistence
of autonomous NNMs}{]} Assume that the external forcing is autonomous
($k=0$) in \eqref{eq:1storder_system}. Assume further that \eqref{eq:negative spectrum}
holds for the eigenvalues of the matrix $A$. 

Then, for $\epsilon\neq0$ small enough, there exists a unique, trivial
NNM, $x_{\epsilon}=\tau(\epsilon)$, with $\tau(0)=0$, in system
\eqref{eq:1storder_system}. This NNM attracts all nearby trajectories
and depends on $\epsilon$ in a $C^{r}$ fashion. \end{thm}
\begin{proof}
Since no zero eigenvalues are allowed for the linearized system, a
unique, smoothly persisting fixed point (trivial NNM) will persist
for small enough $\epsilon$ by the implicit function theorem. This
persisting fixed point will be attracting by the classic stable manifold
theorem applied to system \eqref{eq:1storder_system}, as described,
e.g., in Guckenheimer and Holmes \cite{guckenheimer83}. 
\end{proof}

\subsection{Periodic and quasiperiodic NNM under non-autonomous external forcing
($k\geq1$)\label{sub:Periodic-and-quasiperiodic-NNM-under-non-autonomous-forcing}}

The existence of a small-amplitude periodic solution under purely
periodic forcing in system \eqref{eq:1storder_system} is also routinely
assumed in the nonlinear vibrations literature. These solutions are
then sought via numerical continuation or finite Fourier expansions.
Conditions guaranteeing the success of these formal procedures are
generally omitted.

Next we deduce general mathematical conditions for system \eqref{eq:1storder_system}
under which the existence, uniqueness and even the stability type
of a nontrivial NNM follows under general quasiperiodic forcing, including
the case of periodic forcing ($k=1$).
\begin{thm}
\label{theo:forced NNM}{[}\textbf{Existence, uniqueness and persistence
of non-autonomous NNMs}{]} Assume that the external forcing $f_{1}$
is quasi-periodic with $k\geq1$ frequencies, and the eigenvalues
of the matrix $A$ satisfy \eqref{eq:negative spectrum}. 

Then, for $\epsilon\neq0$ small enough, there exists a unique NNM,
$x_{\epsilon}(t)=\epsilon\tau(\Omega_{1}t,\ldots,\Omega_{k}t;\epsilon)$
in the system \eqref{eq:1storder_system}, where the function $\tau$
is $2\pi$-periodic in each of its first $k$ arguments. This NNM
attracts all nearby trajectories and depends on $\epsilon$ in a $C^{r}$
fashion. \end{thm}
\begin{proof}
For $r\in\mathbb{N}^{+}$ , the theorem can be proven using classic
invariant manifold results, as detailed in Appendix B. Proving case
for $r\in\left\{ 0,\infty,a\right\} $ requires use of the existence
results of Haro and de la Llave \cite{haro06} for invariant tori
which are directly applicable here.
\end{proof}
Theorem \ref{theo:forced NNM} gives a mathematical foundation to
various formal expansion techniques (two-timing, harmonic balance,
etc) and numerical continuation techniques used in the nonlinear vibrations
literature. The existence of the NNMs and their domain of attraction
are independent of any possible resonances between the forcing frequencies
$\Omega_{j}$ and the imaginary parts of the eigenvalues of $A$.
The nature of the NNM (periodic or quasiperiodic) will depend on the
actual value of $\epsilon$, and will be captured by general multi-mode
Fourier expansions, as we describe in Section \ref{sub:Expansions-for-NNMs}.

Of relevance here is the recent work of Kuether et al. \cite{kuether15},
who call a periodic NNM (as defined in Definition \ref{def:NNM})
nonlinear forced response, to distinguish it from nonlinear normal
modes (defined as not necessarily synchronous periodic orbits of the
unforced and undamped nonlinear system). Kuether et al. \cite{kuether15}
investigate connections between NNM and forced responses via intuitive
techniques. A firm connection between the quasiperiodic or periodic
NNM and the $x=0$ equilibrium is offered by Theorem \ref{theo:forced NNM}
for small $\left|\epsilon\right|$ values. For large values of $\left|\epsilon\right|$,
such a connection no longer exists, as the local phase space structure
near the former equilibrium is drastically altered by large perturbations.

\section{Spectral submanifolds in autonomous systems ($k=0$)\label{sec:Spectral-submanifolds-in-autonomous-systems}}

In this section, we discuss spectral submanifolds in the sense of
Definition \ref{def:SSM}, i.e, smoothest nonlinear continuations
of spectral subspaces $E_{j_{1},\ldots,j_{q}}$ in the nonlinear system
\eqref{eq:1storder_system}. We assume here that $k=0$ holds, in
which case, after a possible shift of coordinates, all autonomous
terms contained in the function $f_{1}$ on the right-hand side of
system \eqref{eq:1storder_system} can be subsumed either into the
linear term $Ax$ or the autonomous nonlinear term $f_{0}(x).$ Thus,
without any loss of generality, we can write the $k=0$ case of system
\eqref{eq:1storder_system} in the form 
\begin{equation}
\dot{x}=Ax+f_{0}(x),\qquad\qquad f_{0}(x)=\mathcal{O}(\left|x\right|^{2}),\quad f_{0}\in C^{r},\label{eq:unforced}
\end{equation}
where $r$ is selected as in \eqref{eq:r-def}.

\subsection{Main result}

The idea of seeking two-dimensional spectral submanifolds in system
\eqref{eq:unforced} is originally due to Shaw and Pierre \cite{shaw93}.
They called such spectral submanifolds nonlinear normal modes, even
though these surfaces generally do not contain periodic or even recurrent
motions in the presence of damping. Shaw and Pierre \cite{shaw94}
later extended their original idea to infinite-dimensional evolutionary
equations arising in continuum oscillations. Furthermore, Pescheck
et al. \cite{peschek01} extended the original Shaw--Pierre concept
to the nonlinear continuation of an arbitrary, finite-dimensional
spectral subspace. More recent reviews of the approach and its applications
are given by Kerschen et al. \cite{kerschen09} and Avramov and Mikhlin
\cite{avramov10,avramov13}.

We restrict here the discussion to the case of a stable underlying
NNM, the context in which the Shaw--Pierre invariant manifold concept
was originally proposed. We thus assume throughout this section that
\begin{equation}
\mathrm{Re}\lambda_{j}<0,\quad j=1,\ldots,N,\label{eq:ass1}
\end{equation}
implying that the origin is an asymptotically stable fixed point.
By reversing the direction of time, we obtain similar results for
unstable NNMs (repelling fixed points) with $\mathrm{Re}\lambda_{j}>0,\quad j=1,\ldots,N.$ 

To describe appropriate nonresonance conditions for a spectral subspace
$E$, we will use linear combinations of eigenvalues associated with
a spectral subspace $E$ with nonnegative integers $m_{i}$. Specifically,
for a $q$-dimensional spectral subspace $E$, we denote such linear
combinations as 
\[
\left\langle m,\lambda\right\rangle _{E}:=m_{1}\lambda_{j_{1}}+\ldots+m_{q}\lambda_{j_{q}},\quad\;\lambda_{j_{k}}\in\mathrm{Spect}(A\vert_{E}),\quad m\in\mathrm{\mathbb{N}}^{q},\quad q=\dim E.
\]
We define the order of the nonnegative integer vector $m$ as
\[
\left|m\right|:=m_{1}+\ldots+m_{q}.
\]

\begin{thm}
\label{theo:SSM unforced}{[}\textbf{Existence, uniqueness and persistence
of autonomous SSM}{]} \emph{Consider a spectral subspace $E$ and
assume that the low-order nonresonance conditions 
\begin{equation}
\left\langle m,\lambda\right\rangle _{E}\neq\lambda_{l},\quad\lambda_{l}\not\in\mathrm{Spect}(A\vert_{E}),\quad2\leq\left|m\right|\leq\sigma(E)\label{eq:unforced_nonresonance}
\end{equation}
hold for all eigenvalues of $\lambda_{l}$ of $A$ that lie outside
the spectrum of $A\vert_{E}$. }
\end{thm}
\emph{Then the following statements hold:}
\begin{description}
\item [{\emph{(i)}}] \emph{There exists a class $C^{r}$ SSM, $W(0)$,
tangent to the spectral subspace $E$ at the trivial NNM, $x=0$.
Furthermore, $\dim W(0)=\dim E.$}
\item [{\emph{(ii)}}] W(0)\emph{ is unique among all $C^{\sigma(E)+1}$
invariant manifolds with the properties listed in (i).}
\item [{\emph{(iii)}}] \emph{If $f_{0}$ is jointly $C^{r}$ in $x$ and
an additional parameter vector $\mu$, then the SSM $W(0)$ is jointly
$C^{r}$ in $x$ and $\mu.$ In particular, if $f_{0}(x,\mu)$ is
$C^{\infty}$ or analytic, then $W(0)$ persists under small perturbations
in the parameter $\mu$, and will depend on these perturbations in
a $C^{\infty}$ or analytic fashion, respectively.}\end{description}
\begin{proof}
We deduce the results from a more general theorem of Cabré, Fontich
and de la Llave \cite{cabre03} in Appendix \ref{sub:proof SSM unforced}.
\end{proof}
In short, Theorem \ref{theo:SSM unforced} states that a unique smoothest
Shaw--Pierre-type invariant surface, i.e., an SSM in the sense of
Definition \ref{def:SSM}, exists and persists, as long as no low-order
resonances arise between the master modes and the enslaved modes.
The order of these nonresonance conditions varies from one type of
SSM to the other, as we discuss next.

\subsection{Application to specific spectral subspaces}

We now spell out the meaning of Theorem \ref{theo:SSM unforced} for
different choices of the spectral subspace $E$. We specifically consider
spectral subspaces $E_{j_{1},\ldots,j_{q}}$, where the selected $q$
eigenvalues $\lambda_{j_{1}},\ldots,\lambda_{j_{q}}$ are ordered
so that their real parts form a nondecreasing sequence: 
\begin{equation}
\mathrm{Re}\lambda_{j_{1}}\leq\ldots\mathrm{\leq Re}\lambda_{j_{q}}<0.\label{eq:q eigenvalues ordering}
\end{equation}
We order the real parts of the remaining $N-q$ eigenvalues as 
\begin{equation}
\mathrm{Re}\lambda_{j_{q+1}}\leq\ldots\mathrm{\leq Re}\lambda_{j_{d}}<0.\label{eq:rest of eigenvalues ordering}
\end{equation}
Here $\mathrm{Re}\lambda_{j_{q+1}}$ may be larger or smaller than
the real parts of any of the eigenvalues listed in \eqref{eq:q eigenvalues ordering}. 

We distinguish three types of SSMs in our discussion (cf. Fig. \ref{fig:fast_slow_subspaces}).
\begin{itemize}
\item \emph{A fast spectral submanifold} (\emph{fast SSM})\emph{,} $W_{N-q+1,\ldots,N}(0)$,
is an SSM in the sense of Definition \ref{def:SSM}, with $E_{N-q+1,\ldots,N}$
chosen as the subspace of the $q$ strongest decaying modes of the
linearized system. Here $q\leq N$, with $q=N$ marking the special
case of a fast spectral submanifold that coincides with the domain
of attraction of the fixed point at $x=0$.
\item \emph{An intermediate spectral submanifold (intermediate SSM),} $W_{j_{1},\ldots,j_{q}}(0)$
is an SSM in the sense of Definition \ref{def:SSM}, serving as the
nonlinear continuation of 
\begin{equation}
E_{j_{1},\ldots,j_{q}}=E_{j_{1}}\oplus E_{j_{2}}\oplus\ldots\oplus E_{j_{q}}\label{eq:intermediate SSM}
\end{equation}
for a general choice of the $q<N$ eigenspaces $E_{j_{1}},\ldots,E_{j_{q}}$. 
\item \emph{A slow spectral submanifold (slow SSM),} $W_{1,\ldots,q}(0)$,
is an SSM in the sense of Definition \ref{def:SSM}, with the underlying
spectral subspace $E_{1,\ldots,q}$ chosen as the subspace of the
$q<N$ slowest decaying modes of the linearized system. 
\end{itemize}
In Figure \ref{fig:fast_slow_subspaces}, we illustrate parts of the
spectrum of $A$ that generate fast, intermediate and slow spectral
subspaces, whose smoothest nonlinear continuations are the fast, intermediate
and slow SSMs. 
\begin{figure}[H]
\begin{centering}
\includegraphics[width=0.8\textwidth]{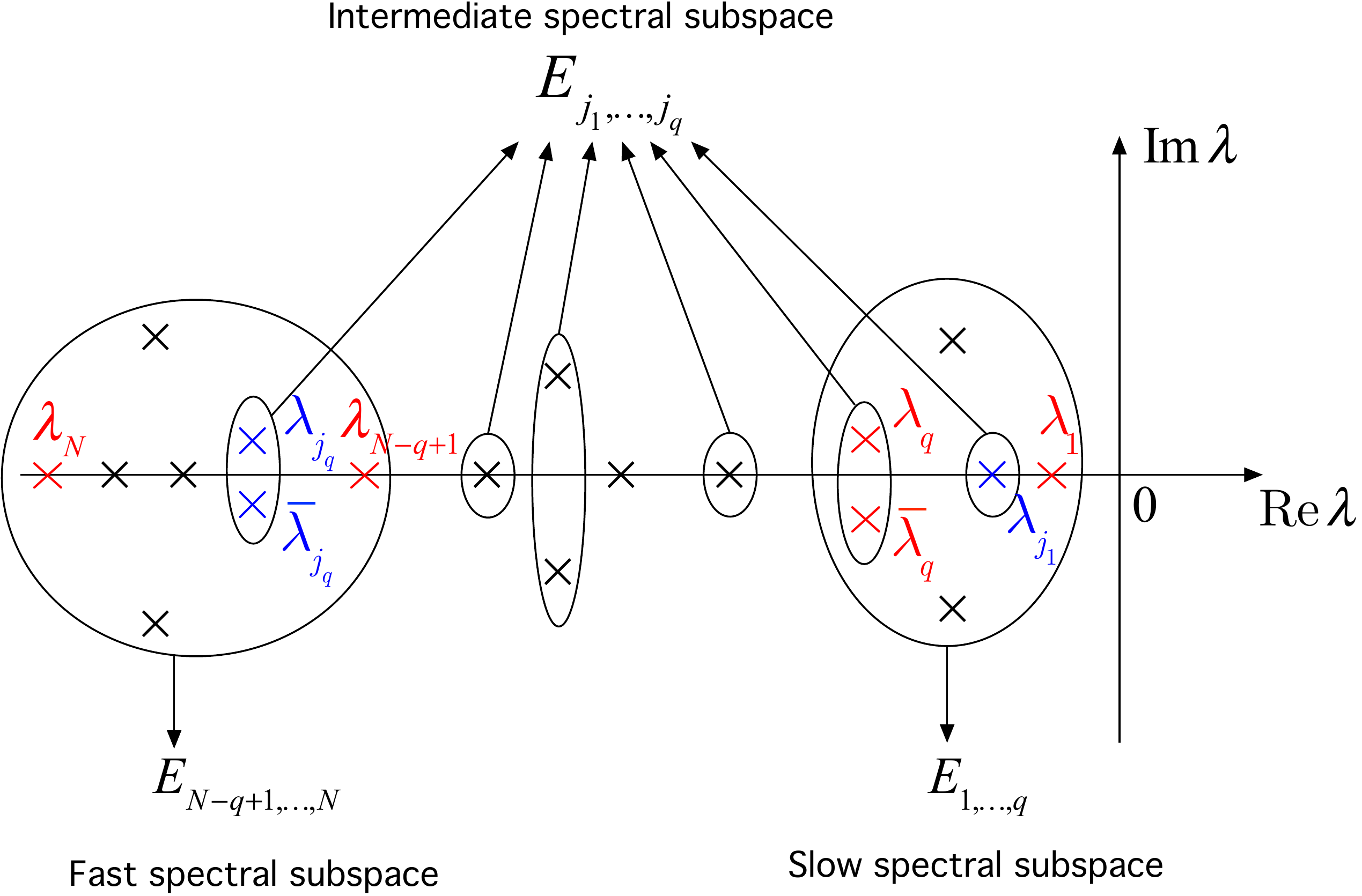}\caption{Fast, intermediate and slow spectral subspaces identified from the
spectrum of $A$. The smoothest nonlinear continuations of these along
an NNM are fast, intermediate and slow SSMs of the NNM.\label{fig:fast_slow_subspaces}}

\par\end{centering}

\end{figure}

Table 1 summarizes relevant  relative spectral quotients and nonresonance
conditions, as obtained from a direct application of Theorem \ref{theo:SSM unforced}
to fast, intermediate and slow spectral subspaces.

\noindent \begin{flushleft}
\begin{tabular}{|l|c|c|c|}
\hline 
 & \textbf{\small{}Fast SSM } & \textbf{\small{}Intermediate SSM } & \textbf{\small{}Slow SSM }\tabularnewline
\hline 
\hline 
{\footnotesize{}$E$} & {\footnotesize{}$E_{N-q+1,\ldots,N}$} & {\footnotesize{}$E_{j_{1},\ldots,j_{q}}$} & {\footnotesize{}$E_{1,\ldots,q}$}\tabularnewline
\hline 
{\footnotesize{}$\sigma(E)$} & {\footnotesize{}0} & {\footnotesize{}$\mathrm{Int}\left[\mathrm{Re}\lambda_{j_{q+1}}/\mathrm{Re}\lambda_{j_{q}}\right]$} & {\footnotesize{}$\mathrm{Int}\left[\mathrm{Re}\lambda_{N}/\mathrm{Re}\lambda_{1}\right]$}\tabularnewline
\hline 
\multirow{2}{*}{{\footnotesize{}Nonresonance:}} & \multirow{2}{*}{{\footnotesize{}-}} & {\footnotesize{}$\sum_{i=1}^{q}\left(a_{i}\lambda_{j_{i}}+b_{i}\bar{\lambda}_{j_{i}}\right)\neq\lambda_{l}$} & {\footnotesize{}$\sum_{i=1}^{q}\left(a_{i}\lambda_{i}+b_{i}\bar{\lambda}_{i}\right)\neq\lambda_{l}$}\tabularnewline
\cline{3-4} 
 &  & {\footnotesize{}$\left|a\right|+\left|b\right|\in[2,\sigma(E)],\,\,\,\,l\in\left[j_{q+1},j_{N}\right]$ } & {\footnotesize{}$\left|a\right|+\left|b\right|\in[2,\sigma(E)],\,\,\,\,l\in\left[q+1,N\right]$ }\tabularnewline
\hline 
{\footnotesize{}$W(0)$} & {\footnotesize{}$W_{N-q+1,\ldots,N}(0)$} & {\footnotesize{}$W_{j_{1},\ldots,j_{q}}(0)$} & {\footnotesize{}$W_{1,\ldots,q}(0)$}\tabularnewline
\hline 
\end{tabular}
\par\end{flushleft}

\begin{center}
Table 1: Conditions for different types of SSMs obtained from Theorem
\ref{theo:SSM unforced}, with parameters $a,b\in\mathbb{N}^{q}$
and $l\in\mathbb{N}$.
\par\end{center}

For fast SSMs, Table 1 requires no non-resonance condition, giving
just a sharpened version of a classic result in dynamical systems,
the strong stable manifold theorem (see, e.g., Hirsch, Pugh and Shub
\cite{hirsch77}). If the nonlinear function $f_{0}$ is analytic
(class $C^{a}$) in a neighborhood of the origin, then so is the unique
fast SSM, $W_{N-q+1,\ldots,N}(0)$. In that case, seeking the unique
fast SSM as a Taylor-expanded graph over the fast stable subspace
$E_{N-q+1,\ldots,N}$ leads to a convergent Taylor series for $W_{N-q+1,\ldots,N}(0)$.
By statement (iii) of Theorem \ref{theo:SSM unforced}, the same holds
for Taylor expansions with respect to any parameter $\mu$ on which
the system may depend analytically.

That said, the relevance of fast SSMs for model reduction is generally
limited. These manifolds contain atypical trajectories that reach
the origin in the shortest possible time, practically unaffected by
the remaining slower modes. Special cases of relevance may arise,
for instance, if one wishes to control general motions that exhibit
the fastest possible decay to the equilibrium. 

The two-dimensional invariant manifolds originally envisioned by Shaw
and Pierre \cite{shaw93} generally fall in the category of intermediate
SSMs, with $q=1,$ $\mathrm{Im\,\lambda}_{j_{1}}\neq0,$ and $\dim E_{j_{1}}=2.$
In the later work by Peschek et al. \cite{peschek01}, invariant surfaces
defined over an arbitrary $q\geq1$ number of internally resonant
modes are envisioned, although the resonance among these modes is
not exploited in the construction. By Table 1, all these intermediate
SSMs exist in a rigorous mathematical sense, as long as the spectral
subspaces over which they are constructed exhibit no low-order resonances
with the remaining modes (resonances \emph{within} those spectral
subspaces are allowed). A low-order resonance is one whose order $\left|a\right|+\left|b\right|$
does not exceed $\sigma(E)=\mathrm{Int}\left[\mathrm{Re}\lambda_{j_{q+1}}/\mathrm{Re}\lambda_{j_{q}}\right]$.
Any such intermediate SSM is of class $C^{r}$, but is already unique
in the class of $C^{\sigma(E)+1}$ invariant surfaces tangent to $E_{j_{1},\ldots,j_{q}}$.
This means that a Taylor expansion of order $\sigma(E)+1$ or higher
is only valid for a unique intermediate SSM. 

Slow SSMs exist by Theorem \ref{theo:SSM unforced} under the conditions
detailed in the last column of Table 1. Again, no low-order resonances
are allowed between the $q$ slowest decaying modes in $E_{1,\ldots,q}$
and the remaining faster modes outside{\footnotesize{} }$E_{1,\ldots,q}${\footnotesize{}.}
The order of the resonance is low if it does not exceed the relative
spectral quotient $\sigma(E)=\mathrm{Int}\left[\mathrm{Re}\lambda_{N}/\mathrm{Re}\lambda_{1}\right]$.
Interestingly, this non-resonance order has no dependence on the number
$q$ of slow modes considered. As intermediate SSMs, slow SSMs are
unique among class $C^{\sigma(E)+1}$ invariant manifolds tangent
to $E_{1,\ldots,q}$ at the trivial normal mode $x=0$. For model
reduction purposes, slow SSMs offer the most promising option, as
we discuss in Section \ref{sec:modelreduction}.

Shaw and Pierre \cite{shaw93,shaw94}, Elmegard \cite{elmegard14},
and Renson et al. \cite{renson16} allude to the theory of normally
hyperbolic invariant manifolds by Fenichel \cite{fenichel71} as justification
for the numerical computation of general SSMs. Another hint in the
literature at a rigorous existence result for two-dimensional autonomous
SSMs in analytic systems is given by Cirillo et al. \cite{cirillo15b},
who invoke a classic analytic linearization theorem by Poincaré \cite{poincare79}.
A closer inspection of these results reveals, however, that the applicability
of the theorems of Fenichel and Poincaré is substantially limited
in practical settings (see Appendices \ref{sub:NHIM proof unforced}
and \ref{sub:Poincare proof unforced} for details).
\begin{example}
{[}\emph{Application of Theorem }\ref{theo:SSM unforced}{]} Consider
the planar system 
\begin{eqnarray}
\dot{x} & = & -x,\nonumber \\
\dot{y} & = & -\sqrt{24}y+x^{2}+x^{3}+x^{4}+x^{5},\label{eq:resonant example-1}
\end{eqnarray}
which is analytic on the whole plane, i.e., we have $r=a$ in the
notation of Theorem \ref{theo:unforced NNM}. The eigenvalues of the
linearized system at the origin are $\lambda_{2}=-\sqrt{24}$ and
$\lambda_{1}=-1,$ giving $N=2$ and $q=1$ for the construction of
a slow SSM $W_{1}(0)$ over the slow subspace $E_{1}=\left\{ (x,y)\,:\,y=0\right\} $.
The required order of nonresonance from Table 1 is, therefore, 
\[
\sigma(E)=\mathrm{Int}\left[\mathrm{Re}\lambda_{N}/\mathrm{Re}\lambda_{1}\right]=\mathrm{Int}\left[\sqrt{24}\right]=4,
\]
up to which the non-resonance condition 
\[
a_{1}\cdot(-1)\neq-\sqrt{24},\quad a_{1}=2,3,4
\]
is satisfied. Then Theorem \ref{theo:SSM unforced} guarantees the
existence of an analytic (class $C^{a})$ slow SSM, $W_{1}(0)$, that
is unique among all class $C^{5}$ invariant manifolds tangent to
the $x$ axis at the origin. We seek this slow SSM in the form 
\begin{equation}
y=h(x)=a_{2}x^{2}+a_{3}x^{3}+a_{4}x^{4}+a_{5}x^{5}+\ldots,\label{eq:analyticSSM}
\end{equation}
the minimal Taylor expansion that only exists for the analytic SSM
but not for the other invariant manifolds. Differentiation of \eqref{eq:analyticSSM}
in time gives
\begin{equation}
\dot{y}=\left[2a_{2}x+3a_{3}x^{2}+4a_{4}x^{3}+5a_{5}x^{4}+\mathcal{O}(x^{5})\right]\dot{x}=-2a_{2}x^{2}-3a_{3}x^{3}-4a_{4}x^{4}-5a_{5}x^{5}+\mathcal{O}(x^{6}),\label{eq:resonant eq ydot1-1}
\end{equation}
while substitution of \eqref{eq:analyticSSM} into the second equation
in \eqref{eq:resonant example-1} gives
\begin{equation}
\dot{y}=\left(1-\sqrt{24}a_{2}\right)x^{2}+\left(1-\sqrt{24}a_{3}\right)x^{3}+\left(1-\sqrt{24}a_{4}\right)x^{4}+\left(1-\sqrt{24}a_{5}\right)x^{5}+\mathcal{O}(x^{6}).\label{eq:resonant eq ydot2-1}
\end{equation}
Equating \eqref{eq:resonant eq ydot1-1} and \eqref{eq:resonant eq ydot2-1}
gives
\begin{equation}
a_{2}=\frac{1}{\sqrt{24}-2},\quad a_{3}=\frac{1}{\sqrt{24}-3},\quad a_{4}=\frac{1}{\sqrt{24}-4},\quad a_{5}=\frac{1}{\sqrt{24}-5},\qquad a_{j}=0,\quad j\geq6.\label{eq:slowmanifold coefficients}
\end{equation}
 We also observe that the ODE \eqref{eq:resonant example-1} is explicitly
solvable: a direct integration gives $x(t)$ which, upon substitution
into the $y$ equation, yields an inhomogeneous linear ODE for $y(t).$
Combining the expressions for $x(t)$ and $y(t)$ enables us to eliminate
the time variable $t$, giving the equation of trajectories in the
form
\begin{eqnarray*}
y(x;x_{0},y_{0}) & = & K(x_{0},y_{0})x^{\sqrt{24}}+\frac{x^{2}}{\sqrt{24}-2}+\frac{x^{3}}{\sqrt{24}-3}+\frac{x^{4}}{\sqrt{24}-4}+\frac{x^{5}}{\sqrt{24}-5},\\
K(x_{0},y_{0}) & = & \frac{y_{0}}{x_{0}^{\sqrt{24}}}-\frac{x_{0}^{2-\sqrt{24}}}{\sqrt{24}-2}-\frac{x_{0}^{3-\sqrt{24}}}{\sqrt{24}-3}-\frac{x_{0}^{4-\sqrt{24}}}{\sqrt{24}-4}-\frac{x_{0}^{5-\sqrt{24}}}{\sqrt{24}-5},
\end{eqnarray*}
with $(x_{0},y_{0})$ denoting an arbitrary initial condition on the
trajectory. This shows that the graph $y(\,\cdot\,;x_{0},y_{0})$
of the slow SSM is generally only of class $C^{4}$, as the term $K(x_{0},y_{0})x^{\sqrt{24}}$
admits only four continuous derivatives at the origin. The only exception
is the case $K(x_{0},y_{0})=0,$ for which $y(\,\cdot\,;x_{0},y_{0})$
becomes a quintic polynomial in $x$ and hence analytic over the whole
plane. But $K(x_{0},y_{0})=0$ holds only along the points 
\begin{equation}
y_{0}=\frac{x_{0}^{2}}{\sqrt{24}-2}+\frac{x_{0}^{3}}{\sqrt{24}-3}+\frac{x_{0}^{4}}{\sqrt{24}-4}+\frac{x_{0}^{5}}{\sqrt{24}-5},\label{eq:SSMexample}
\end{equation}
which lie precisely on the SSM, $W_{1}(0)$, whose Taylor expansion
we computed in \eqref{eq:slowmanifold coefficients}. This example,
therefore, illustrates the sharpness of the results of Theorem \ref{theo:SSM unforced}:
the analytic slow SSM, $W_{1}(0)$, is indeed unique among all five
times continuously differentiable invariant manifolds tangent to the
$x$ axis at the origin. We plot in red the unique analytic SSM for
this example in Fig. \ref{fig:resonant}a.
\begin{figure}[H]
\begin{centering}
\includegraphics[width=0.6\textwidth]{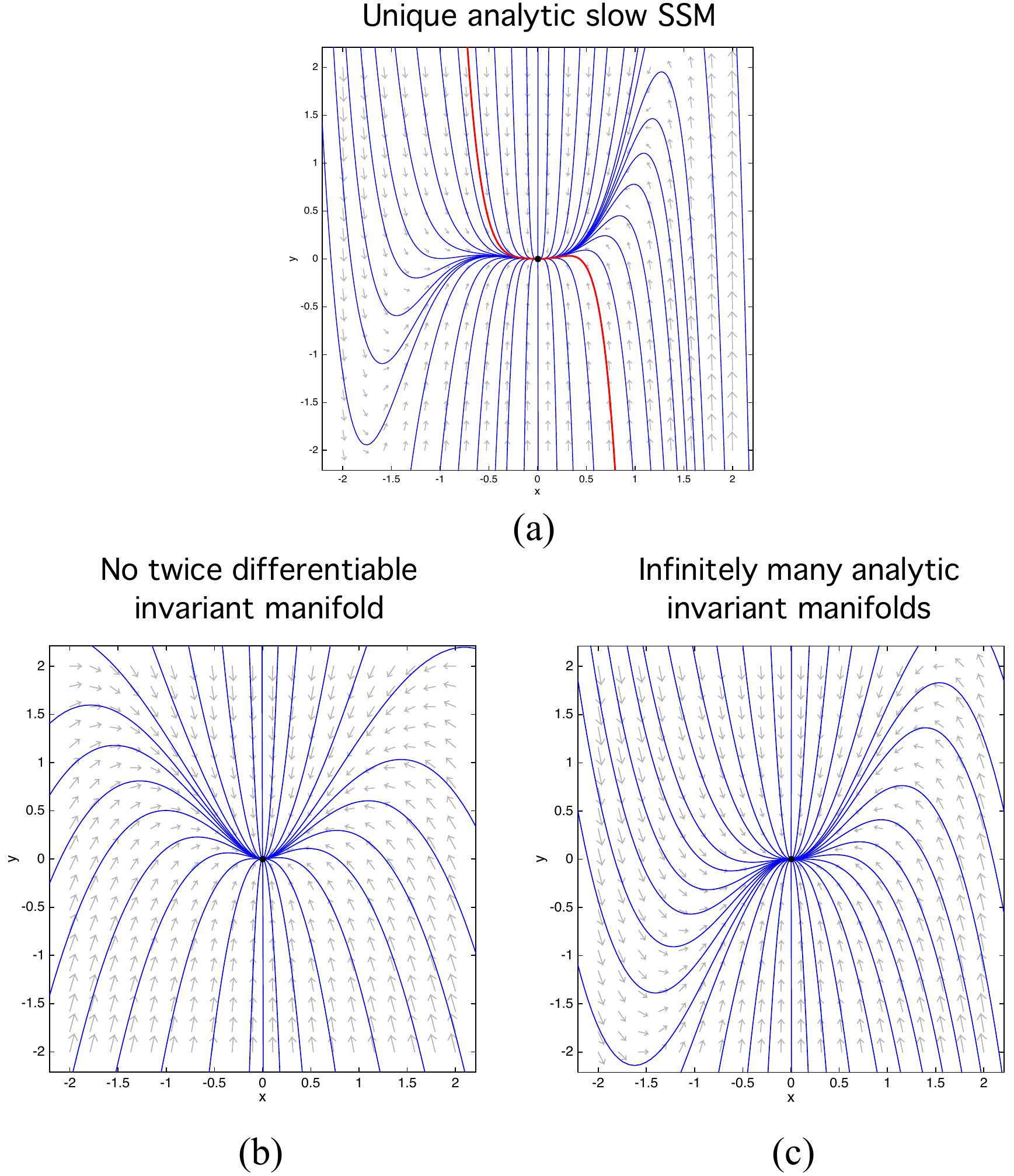}
\par\end{centering}

\caption{(a) Phase portrait of system \eqref{eq:resonant example-1}, with
the unique analytic SSM guaranteed by Theorem \ref{theo:SSM unforced}
computed explicitly (red) (b) Phase portrait of system \eqref{eq:resonant example}.
(c) Phase portrait of system \eqref{eq:resonant example2}. For all
three plots: trajectories are shown in blue and the vector field is
indicated with grey arrows.\label{fig:resonant}}
\end{figure}

\end{example}

\begin{example}
{[}\emph{Optimality of Theorem \ref{theo:SSM unforced}}{]} Consider
the planar dynamical system
\begin{eqnarray}
\dot{x} & = & -x,\nonumber \\
\dot{y} & = & -2y+x^{2},\label{eq:resonant example}
\end{eqnarray}
with its phase portrait shown in Fig. \ref{fig:resonant}b. The system
is analytic over the whole plane, and has a stable node-type fixed
point at the origin with eigenvalues $\lambda_{2}=-2$ and $\lambda_{1}=-1.$
This system, therefore, falls into the slow SSM case of Table 1 with
$\sigma(E)=2.$ The corresponding nonresonance condition is, however,
violated because 
\[
a_{1}\cdot(-1)=-2,\quad a_{1}=2.
\]
Theorem \ref{theo:SSM unforced}, therefore, fails to apply, and hence
we have no a priori mathematical guarantee for the existence or uniqueness
of an at least $C^{2}$ slow SSM. To see if such a manifold nevertheless
exists, we again seek a slow SSM in the form 
\begin{equation}
y=h(x)=a_{2}x^{2}+a_{3}x^{3}+\ldots,\label{eq:resonant submanifold}
\end{equation}
a graph with quadratic tangency to $E_{1}$ at the origin. Differentiation
of this graph in time gives
\begin{equation}
\dot{y}=\left[2a_{2}x+\mathcal{O}(x^{2})\right]\dot{x}=-2a_{2}x^{2}+\mathcal{O}(x^{3}),\label{eq:resonant eq ydot1}
\end{equation}
while substitution of the graph into the second equation in \eqref{eq:resonant example}
gives
\begin{equation}
\dot{y}=\left(-2a_{2}+1\right)x^{2}+\mathcal{O}(x^{3}).\label{eq:resonant eq ydot2}
\end{equation}
Equating \eqref{eq:resonant eq ydot1} and \eqref{eq:resonant eq ydot2}
gives no solution for $a_{2}$, and hence no $C^{2}$ invariant manifold
tangent to $E_{1}$ exists in this example. There are infinitely many
invariant manifolds tangent to the spectral subspace $E_{1}$ but
none of them is smoother than the other one: they all just have one
continuous derivative at the origin. As a consequence, no SSM exists
by Definition \ref{def:SSM}. Next, consider the slightly different
dynamical system 
\begin{eqnarray}
\dot{x} & = & -x,\nonumber \\
\dot{y} & = & -2y+x^{3},\label{eq:resonant example2}
\end{eqnarray}
with its phase portrait shown in Fig. \ref{fig:resonant}c, which
violates the same nonresonance condition as \ref{eq:resonant example}.
This time, we find infinitely many analytic invariant manifolds tangent
to the spectral subspace $E_{1}$. Indeed, any member of the analytic
manifold family $y(x)=Cx^{2}-x^{3}$, with the parameter $C\in\mathbb{R}$,
is invariant and tangent to the spectral subspace $E_{1}$ of \eqref{eq:resonant example2}
at the origin. Thus, the violation of the nonresonance condition in
the slow case of Table 1 may either lead to the non-existence of a
single $C^{2}$ invariant manifold, or to a high degree of non-uniqueness
of smooth (even analytic) invariant manifolds.
\end{example}

\begin{example}
\label{ex:shaw-pierre1}{[}\emph{Illustration of Theorem \ref{theo:SSM unforced}
on a mechanical example}{]} We reconsider here the damped nonlinear
mechanical system studied by Shaw and Pierre \cite{shaw93}. Shown
in Fig. \ref{fig:shaw-pierre-model1}, this two-degree-of-freedom
mechanical system consists of two masses connected via springs to
each other and to their environment. Two of the springs are linearly
elastic and linearly damped, while the remaining spring is still elastic
but has a cubic nonlinearity as well. The displacements $q_{1}$ and
$q_{2}$, as well as the damping coefficient $c$, the spring constant
$k$, and the coefficient $\gamma$ of the cubic nonlinearity, are
all non-dimensionalized.

\begin{figure}[H]
\begin{centering}
\includegraphics[width=0.6\textwidth]{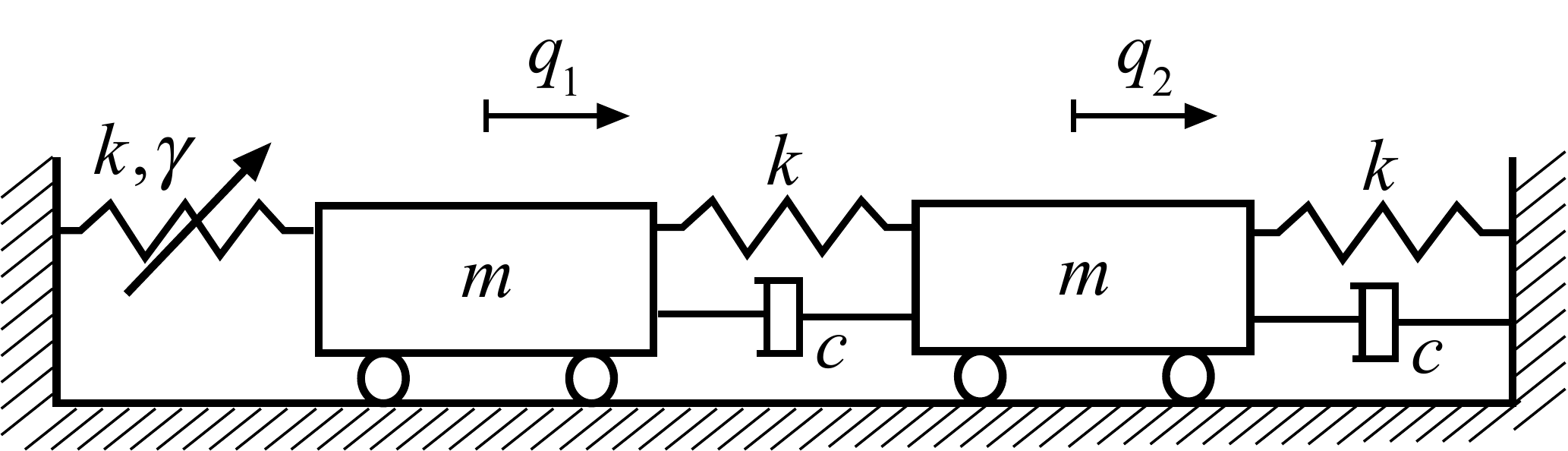}
\par\end{centering}

\caption{The two-degree-of-freedom mechanical model considered by Shaw and
Pierre \cite{shaw93}.\label{fig:shaw-pierre-model1}}
\end{figure}
The equations of motion for this system are of the general form \eqref{eq:system}
with $n=N/2=2$ and $q=(q_{1},q_{2}),$ and with the quantities
\[
M=\left(\begin{array}{cc}
m & 0\\
0 & m
\end{array}\right),\quad C=\left(\begin{array}{cc}
c & -c\\
-c & 2c
\end{array}\right),\quad K=\left(\begin{array}{cc}
2k & -k\\
-k & 2k
\end{array}\right),\quad G=B=\left(\begin{array}{cc}
0 & 0\\
0 & 0
\end{array}\right),
\]
\[
F_{0}(q,\dot{q})=\left(\begin{array}{c}
-\gamma q_{1}^{3}\\
0
\end{array}\right),\quad\epsilon F_{1}=\left(\begin{array}{c}
0\\
0
\end{array}\right).
\]
In the variables $x_{1}=q_{1},x_{2}=\dot{q}_{1}$, $x_{3}=q_{2},x_{4}=\dot{q}_{2}$,
the first-order form \eqref{eq:1storder_system} of the system has
\[
A=\left(\begin{array}{cccc}
0 & 1 & 0 & 0\\
-\frac{2k}{m} & -\frac{c}{m} & \frac{k}{m} & \frac{c}{m}\\
0 & 0 & 0 & 1\\
\frac{k}{m} & \frac{c}{m} & -\frac{2k}{m} & -\frac{2c}{m}
\end{array}\right),\quad f_{0}(x)=\left(\begin{array}{c}
0\\
-\gamma x_{1}^{3}\\
0\\
0
\end{array}\right),\quad\epsilon f_{1}=\left(\begin{array}{c}
0\\
0\\
0\\
0
\end{array}\right).
\]
Shaw and Pierre \cite{shaw93} fixed the parameter values
\begin{equation}
c=0.3,\quad k=1,\quad m=1,\quad\gamma=0.5,\label{eq:S-P parameter values}
\end{equation}
and reported for this parameter setting the eigenvalues
\begin{equation}
\lambda_{1}=-0.0741\pm1.0027i\qquad\lambda_{2}=-0.3759\pm1.6812i.\label{eq:SP eigenvalues}
\end{equation}
This implies the existence of two two-dimensional real invariant subspaces,
$E_{1}$ and $E_{2}$, for the linearized system. 

Shaw and Pierre calculated a formal cubic-order Taylor expansion for
SSMs tangent to these subspaces at the origin. Since the function
$f_{0}(x)$ is analytic on the whole phase space, Theorem \eqref{theo:SSM unforced}
guarantees the existence of an analytic fast SSM, $W_{2}(0)$. Furthermore,
since $\sigma(E_{2})=0$ holds by Table 1, $W_{2}(0)$ is unique among
all $C^{1}$ invariant manifolds tangent to the fast spectral subspace
$E_{2}$ at the fixed point $x=0$. Theorem \eqref{theo:SSM unforced}
also guarantees the existence of a unique analytic slow SSM, $W_{1}(0)$,
as long as no resonance conditions (listed in the last column of Table
1) up to order 
\[
\sigma(E_{1})=\mathrm{Int}\text{\ensuremath{\left[\frac{\mathrm{Re}\,\lambda_{2}}{\mathrm{Re}\,\lambda_{1}}\right]}=5 }
\]
 hold. These nonresonance conditions take the specific form
\[
-0.0741\left(a_{1}+b_{1}\right)+1.0027\left(a_{1}-b_{1}\right)i\neq-0.3759\pm1.6812i,\quad\left|a_{1}\right|+\left|b_{1}\right|=2,3,4,5,
\]
which are all satisfied, as seen by inspection. 

We conclude that the analytic slow SSM $W_{1}(0)$ exists, and is
unique among all $C^{6}$ invariant manifolds tangent to the slow
spectral subspace $E_{1}$ at the fixed point $x=0$. Therefore, the
cubic-order Taylor expansion of Shaw and Pierre \cite{shaw93} more
than captures the fast SSM $W_{2}(0)$ uniquely, but fails to capture
the slow SSM uniquely. Indeed, the latter cubic expansion holds for
infinitely many $C^{5}$ invariant manifolds tangent to the origin
along the slow spectral subspace. A $6^{th}$ order Taylor-expansion
would hold only for the unique analytic slow SSM, for which the expansion
can continued up to any order, giving a convergent power series in
a neighborhood of the origin. The required order of expansion remains
the $6^{th}$ for general underdamped parameter values, but increases
sharply with increasing overdamping (see Fig. \ref{fig:required order of expansion for SP1}).
\begin{figure}[H]
\centering{}\includegraphics[width=0.4\textwidth]{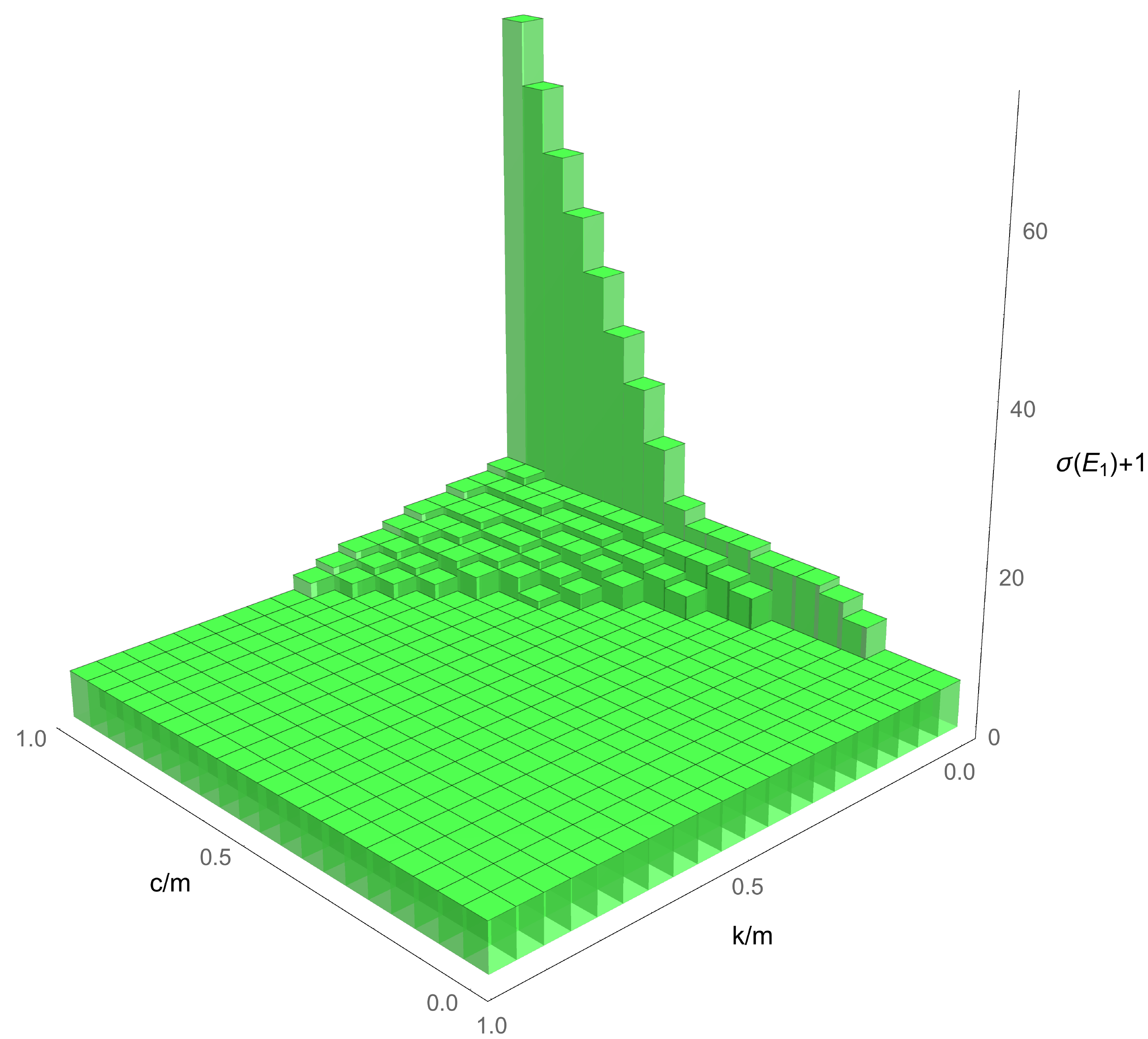}\caption{Dependence of the uniqueness class of the slow SSM in Example \ref{ex:shaw-pierre1}
on the parameters $k/m$ and $c/m$. \label{fig:required order of expansion for SP1}}
\end{figure}

We now carry out the computation of the slow SSM in detail for the
parameter values \eqref{eq:S-P parameter values}. Applying a linear
change of coordinates, we split the state vector $x\in\mathbb{R}^{4}$
as
\end{example}
\begin{equation}
x=(y,z)\in E_{1}\times E_{2},\label{eq:xyz}
\end{equation}
which results in the transformed equations of motion

\begin{equation}
\dot{y}=A_{y}y+f_{0y}(y,z)=\left(\begin{array}{cc}
-0.0741 & 1.0027\\
-1.0027 & -0.0741
\end{array}\right)y+\left(\begin{array}{c}
1.0148\\
-0.2162
\end{array}\right)p(y,z),\label{eq:ysep}
\end{equation}

\begin{equation}
\dot{z}=A_{z}z+f_{0z}(y,z)=\left(\begin{array}{cc}
-0.3759 & 1.6812\\
-1.6812 & -0.3759
\end{array}\right)z+\left(\begin{array}{c}
0.8046\\
-0.1685
\end{array}\right)p(y,z),\label{eq:zsep}
\end{equation}

\[
p(y,z)=-0.5\left(-0.0374\thinspace y_{1}-0.5055\thinspace y_{2}-0.1526\thinspace z_{1}-0.3052\thinspace z_{2}\right)^{3}.
\]

We seek the slow SSM, $W_{1}(0),$ within the class of $C^{6}$ function
in which the analytic SSM is already unique. This requires finding
the coefficients in the $6^{th}$ order Taylor-expansion

\begin{equation}
z=h(y)=\sum_{\left|p\right|=1}^{6}h_{p}y^{p},\quad p=\left(p_{1},p_{2}\right)\in\mathbb{N}^{2},\quad y^{p}=y_{1}^{p_{1}}y_{2}^{p_{2}},\quad h_{p}\in\mathbb{R}^{2}.\label{eq:z poly}
\end{equation}
Differentiating this expression with respect to time and substituting
$\dot{z}$ from \eqref{eq:zsep} gives

\[
\frac{\partial h(y)}{\partial y}\left[A_{y}y+f_{0y}\left(y,h(y)\right)\right]=A_{z}h(y)+f_{0z}\left(y,h(y)\right).
\]
Equating powers of $y$ on both sides of this last expression, we
obtain the unknown coefficients $h_{p}$ in \eqref{eq:z poly} and
hence the slow SSM in the form

\begin{eqnarray*}
z_{1} & = & -0.0278y_{1}^{3}+0.0011y_{1}^{2}y_{2}-0.0026y_{1}y_{2}^{2}+0.0009y_{2}^{3}\\
 &  & +0.0023y_{1}^{5}-0.0006y_{1}^{4}y_{2}+0.0026y_{1}^{3}y_{2}^{2}-0.0007y_{1}^{2}y_{2}^{3}-0.0010y_{1}y_{2}^{4}+0.0002y_{2}^{5},\\
z_{2} & = & -0.0032y_{1}^{3}-0.0470y_{1}^{2}y_{2}-0.0074y_{1}y_{2}^{2}-0.0323y_{2}^{3}\\
 &  & +0.0004y_{1}^{5}+0.0039y_{1}^{4}y_{2}+0.0004y_{1}^{3}y_{2}^{2}+0.0065y_{1}^{2}y_{2}^{3}-0.0005y_{1}y_{2}^{4}+0.0011y_{2}^{5}.
\end{eqnarray*}
Note that the $6^{th}$-order terms (as well as any other odd-order
terms) vanish due to the particular form of the nonlinearity in this
example. The slow SSM obtained in this fashion is shown in Fig. \ref{fig:SSM_minipage}.

\begin{figure}[H]
~~~~~~~~~~~~~~~~~\subfloat[\label{fig:ssm_auto_1}]{\centering{}\includegraphics[width=0.4\textwidth]{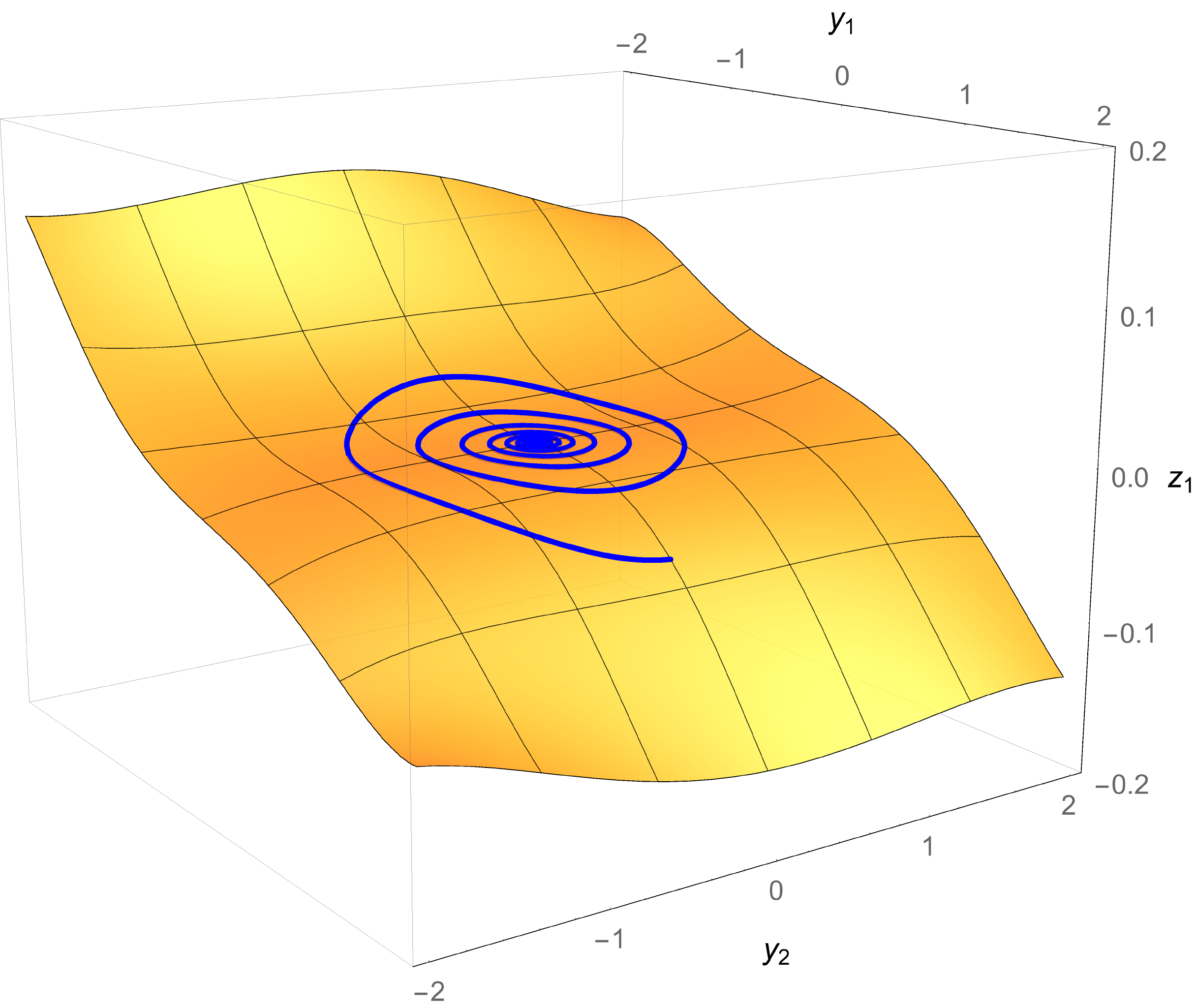}}~~~~~~~~~~~~~~~~\subfloat[\label{fig:ssm_auto_2}]{\centering{}\includegraphics[width=0.4\textwidth]{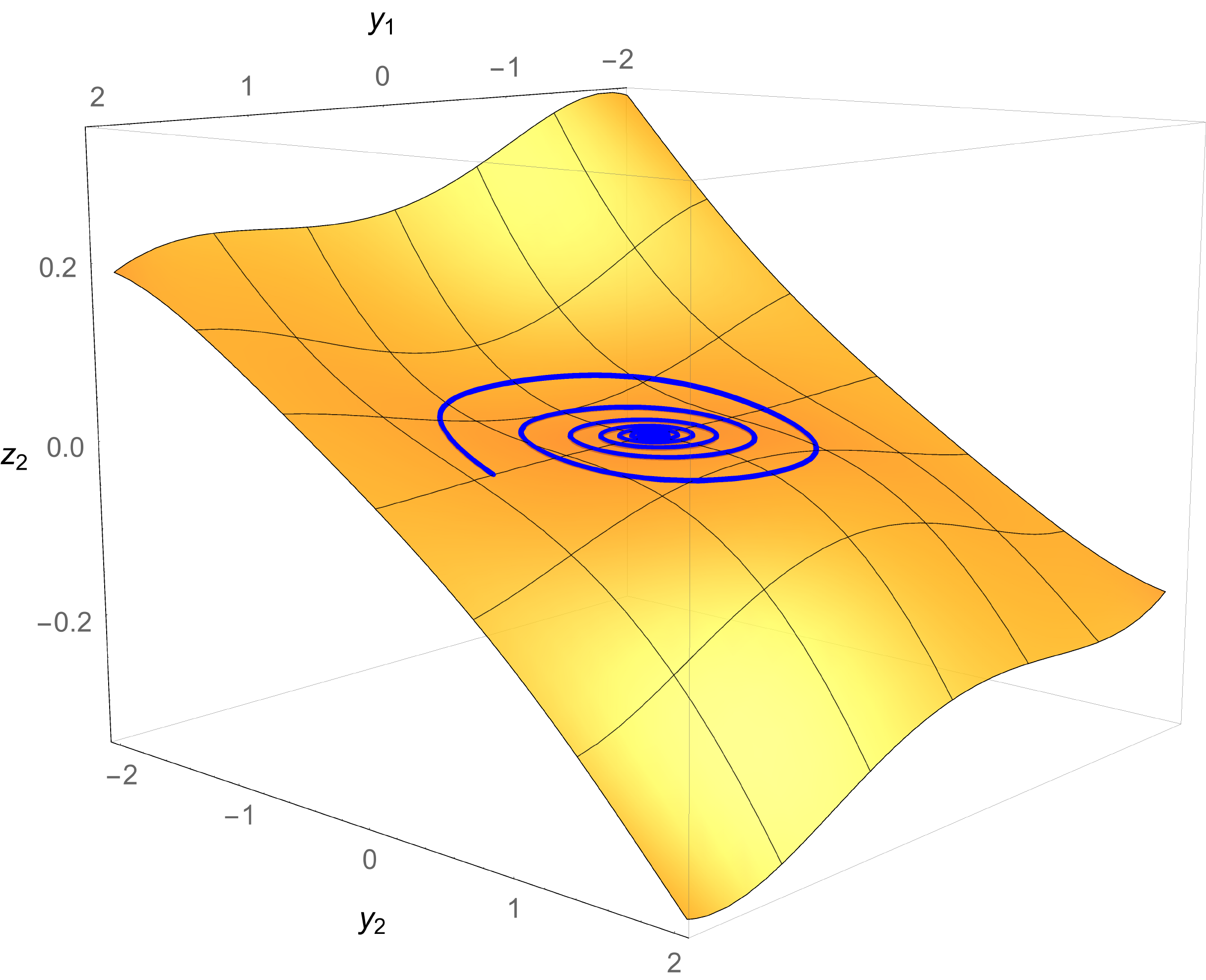}}

\caption{Two views of the slow SSM for the nonlinear oscillator system \eqref{eq:ysep}-\eqref{eq:zsep}.
In the plots (\ref{fig:ssm_auto_1})-(\ref{fig:ssm_auto_2}), $z_{1}$
and $z_{2}$ are shown as a function of $y$ respectively. The blue
curve is a trajectory starting on the SSM from the initial conditions
($y_{1}(0)=1.2$, $y_{2}(0)=0$, $z_{1}(y_{1}(0),y_{2}(0))=-0.042$,
$z_{2}(y_{1}(0),y_{2}(0))=-0.0045$). The trajectory remains close
to the SSM and converges to a trivial NNM, the $(y,z)=(0,0)$ fixed
point. \label{fig:SSM_minipage}}
\end{figure}

We now compare the accuracy of the third-order approximation employed
by Shaw and Pierre \cite{shaw93} to the fifth-order approximation
used here. By the nature of the nonlinearity, this is in fact just
one step up in accuracy, as the fourth-order terms are absent in the
Taylor expansion of the SSM. Figure \ref{fig:Poincare map 1} shows
a Poincaré-map view of our comparison, with dots indicating the intersection
of representative trajectories launched from the approximate slow
SSMs with the $y_{2}=0$ Poincaré section. We conclude that the sixth-order
(which is the same as the fifth-order) approximation to the slow SSM
brings a major improvement in its accuracy. This is evidenced by significantly
reduced trajectory oscillations arising from the lack of exact invariance
of the approximate SSM.

\begin{figure}[H]
\centering{}\includegraphics[width=0.6\textwidth]{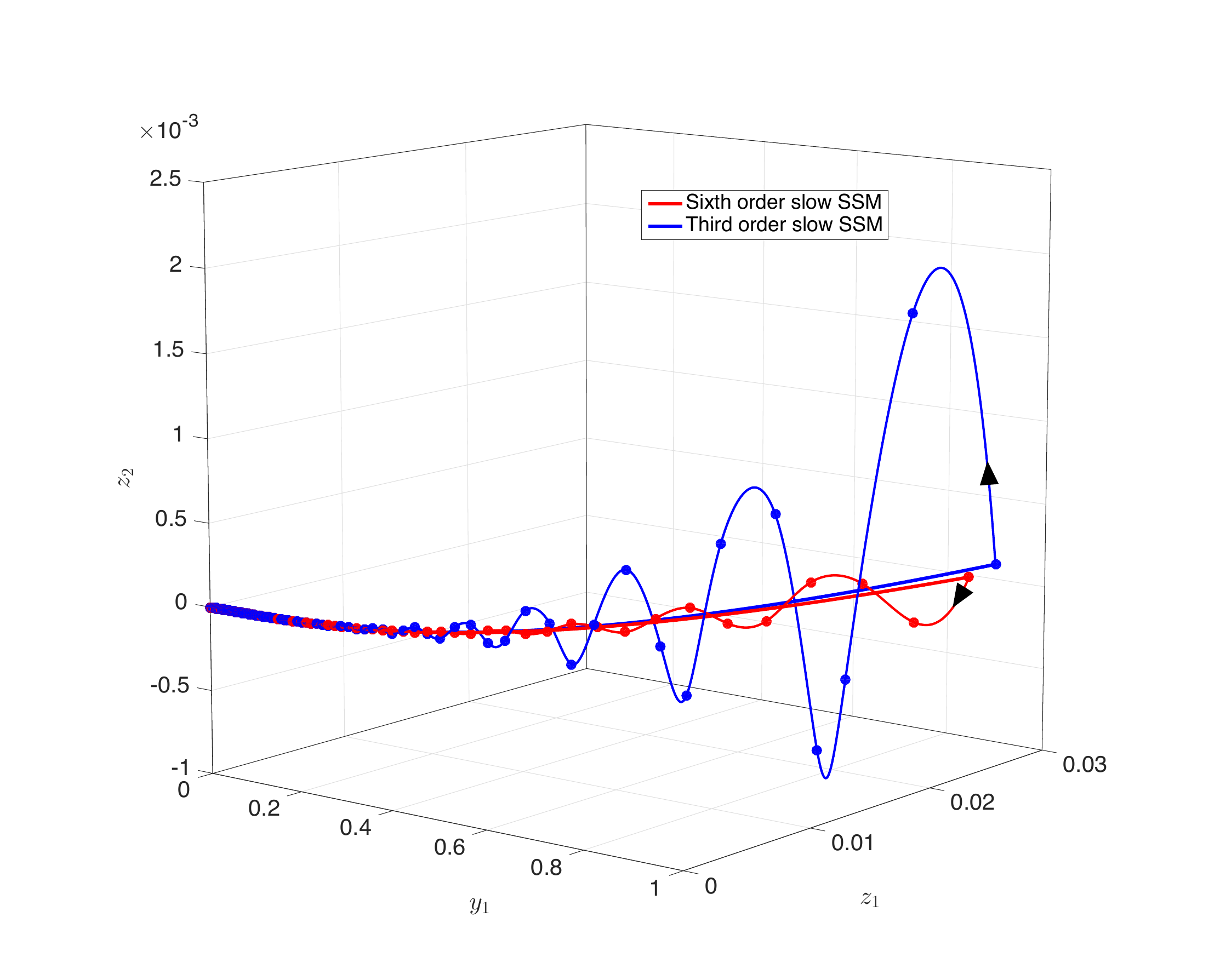}\caption{Poincaré map for trajectories launched on different approximation
to the slow SSM. Dots indicate intersections of the trajectories with
the $y_{2}=0$ hyperplane. Arcs connecting adjacent intersections
are for illustration only, to give a sense of the trajectory evolution.
\label{fig:Poincare map 1}}
\end{figure}

\section{Spectral submanifolds in non-autonomous systems ($k>0$)}

The idea of periodic SSMs ($k=1$) was proposed first by Shaw, Peschek
and Pierre \cite{shaw99} for undamped oscillatory systems, then later
extended by Jiang, Pierre and Shaw \cite{jiang05} for systems with
damping. In these studies, the periodic time-dependence appears as
a perturbation, as in our equation \eqref{eq:fullsystemagain}. As
a parallel development, Sinha, Redkar and Butcher \cite{sinha05}
considered systems with a time-periodic linear part, and applied a
Lyapunov--Floquet transformation to bring this linear part to an autonomous
form before applying the SSM approach of Shaw, Pierre et al. This
treatment appears to be the first one to give a general non-resonance
condition for the Fourier expansion of the SSM to be at least formally
computable (without consideration of convergence) up to a given order. 

In later work, Redkar and Sinha \cite{redkar08} assume single-frequency
external forcing and select the master modes (i.e., those constituting
the spectral subspace of interest) as the ones in resonance or near-resonance
with the external forcing. Gabale and Sinha \cite{gabale11} develop
this approach further, selecting the master modes to be either in
near-resonance with the forcing, or to be those with eigenvalues that
have dominant negative real parts (fast NNMs). The authors provide
nonresonance conditions for formal computability up to any order,
but the actual convergence of the approximation to a true invariant
manifold is not discussed. As noted before, such a convergence is
not guaranteed, as a PDE for an invariant surface can always be written
down for any system, but it may not have a solution under the prescribed
boundary conditions. Gabale and Sinha \cite{gabale11} also discuss
the case of a time-periodic linear part, using a Lyapunov--Floquet
transformation. This appears to be the first reference where the Shaw--Pierre
invariant manifold approach is formally applied in the presence of
two frequencies. 

In summary, as in the autonomous case, only formal calculations of
non-autonomous SSMs have appeared in the literature without mathematical
arguments for existence and uniqueness. Unlike in the autonomous case,
however, the connection of the assumed non-autonomous SSM to any surviving
NNM (periodic orbit or invariant torus) has remained unexplored. It
is therefore unclear in the literature what the orbits in the envisioned
invariant manifolds should asymptote to. In the following, we address
these conceptual gaps in the theory of non-autonomous SSMs.

\subsection{Main result}

We consider the full, perturbed non-autonomous dynamical system 
\begin{equation}
\dot{x}=Ax+f_{0}(x)+\epsilon f_{1}(x,\Omega t;\epsilon),\qquad f_{0}(x)=\mathcal{O}(\left|x\right|^{2}),\quad\Omega\in\mathbb{R}^{k},\,\,k\geq1;\quad0<\epsilon\ll1.\label{eq:fullsystemagain}
\end{equation}
Our smoothness assumptions on $f_{0}$ and $f_{1}$ will be spelled
out in our main result below.

We continue to assume that the linear part of this system is asymptotically
stable, i.e., 
\begin{equation}
\mathrm{Re}\lambda_{j}<0,\quad j=1,\ldots,N.\label{eq:ass1-1}
\end{equation}
As already noted in the autonomous case, this assumption on the dissipative
nature of the system ensures that our NNM and SSM definitions indeed
capture distinguished solution sets of the nonlinear oscillatory system.
\begin{thm}
\label{theo:SSM forced}{[}\textbf{Existence, uniqueness and persistence
of non-autonomous SSM}{]} \emph{Consider a spectral subspace $E$
and assume that the low-order nonresonance conditions 
\begin{equation}
\left\langle m,\mathrm{Re}\lambda\right\rangle _{E}\neq\mathrm{Re}\lambda_{l},\quad\lambda_{l}\not\in\mathrm{Spect}(A\vert_{E}),\quad2\leq\left|m\right|\leq\Sigma(E)\label{eq:forced_nonresonance}
\end{equation}
hold for all eigenvalues $\lambda_{l}$ of $A$ that lie outside the
spectrum of $A\vert_{E}$. }
\end{thm}
\emph{Then the following hold:}
\begin{description}
\item [{\emph{(i)}}] \emph{There exists an SSM, $W(x_{\epsilon}(t))$that
is of class $C^{\Sigma(E)+1}$in the variable $x$. }For any fixed
time $t_{0}$, the time slice $W(x_{\epsilon}(t_{0}))$ of the SSM
is $\mathcal{O}(\epsilon)$ $C^{r}$-close to\emph{ $E$ along the
quasiperiodic NNM, $x_{\epsilon}(t)=\epsilon\tau(\Omega t;\epsilon)$.
Furthermore, $\dim W(x_{\epsilon}(t))=\dim E+k$.}
\item [{\emph{(ii)}}] \emph{$W(x_{\epsilon}(t))$is unique among all invariant
manifolds that satisfy the properties listed in (i) and are at least
of class $C^{\Sigma(E)+1}$ with respect to the $x$ variable along
the NNM $x_{\epsilon}(t)$, }
\item [{(iii)}] If the functions $f_{0}$ and $f_{1}$ are $C^{\infty}$
or analytic, then $W(x_{\epsilon}(t))$ will depend on $\epsilon$
in a $C^{\infty}$ or analytic fashion, respectively.\end{description}
\begin{proof}
The results can be deduced from a more general result of Haro and
de la Llave \cite{haro06}, as we show in Appendix \ref{sub:proof SSM forced}.
\end{proof}
According to Theorem \ref{theo:SSM forced}, under the appropriate
nonresonance conditions between the modes in the spectral subspace
$E$ and those outside $E$, a well-defined periodic or quasiperiodic
SSM attached to a periodic or quasiperiodic NNM exists. This gives
precise mathematical conditions for the existence and uniqueness of
the invariant surfaces envisioned by Jiang, Pierre and Shaw \cite{jiang05}
for the time-periodic case, and extends their existence to the case
of quasiperiodic forcing. The SSMs obtained in this fashion are unique
among invariant surfaces that are at least $\Sigma(E)+1$-times continuously
differentiable in the $x$ direction along the NNM.

\subsection{Applications to specific spectral subspaces}

We again consider a select group of $q$ master modes of the linearized
system with 
\begin{equation}
\mathrm{Re}\lambda_{j_{1}}\leq\ldots\mathrm{\leq Re}\lambda_{j_{q}}<0,\label{eq:q eigenvalues ordering-1}
\end{equation}
 and with the remaining modes ordered as 
\[
\mathrm{Re}\lambda_{j_{q+1}}\leq\ldots\mathrm{\leq Re}\lambda_{j_{N}}<0.
\]

In analogy with the autonomous case, we distinguish three types of
non-autonomous SSMs in our discussion (cf. Fig. \ref{fig:fast_slow_subspaces}):
\begin{itemize}
\item \emph{A fast spectral submanifold} (\emph{fast SSM})\emph{,} $W_{N-q+1,\ldots,N}(x_{\epsilon}(t))$,
is an SSM in the sense of Definition \ref{def:SSM}, with the underlying
spectral subspace chosen as $E_{N-q+1,\ldots,N}$, the subspace of
the $q<N$ fastest decaying modes of the linearized system. The SSM
$W_{N-q+1,\ldots,N}(x_{\epsilon}(t))$ is time-periodic if either
$k=1$ or the elements of the frequency $\Omega$ are rationally commensurate
for $k>1$. In all cases, $W_{N-q+1,\ldots,N}(x_{\epsilon}(t))$ is
a surface in which trajectories are asymptotic to the nontrivial NNM,
$x_{\epsilon}(\Omega t)$.
\item \emph{An intermediate spectral submanifold (intermediate SSM),} $W_{j_{1},\ldots,j_{q}}(x_{\epsilon}(t))$,
is an SSM in the sense of Definition \ref{def:SSM}, serving as the
nonlinear continuation of 
\begin{equation}
E_{j_{1},\ldots,j_{q}}=E_{j_{1}}\oplus E_{j_{2}}\oplus\ldots\oplus E_{j_{q}}\label{eq:intermediate SSM-1}
\end{equation}
for a general choice of the $q<N$ eigenspaces $E_{j_{1}},\ldots,E_{j_{q}}$.
Trajectories in $W_{j_{1},\ldots,j_{q}}(x_{\epsilon}(t))$ are asymptotic
to the nontrivial NNM, $x_{\epsilon}(t)$.
\item \emph{A slow spectral submanifold (slow SSM),} $W_{1,\ldots,q}(x_{\epsilon}(t))$,
is an SSM in the sense of Definition \ref{def:SSM}, with the underlying
spectral subspace chosen as $E_{1,\ldots,q}$, the subspace of the
$q<N$ slowest decaying modes of the linearized system. Again, trajectories
in $W_{1,\ldots,q}(x_{\epsilon}(t))$ are asymptotic to the nontrivial
NNM, $x_{\epsilon}(t)$.
\end{itemize}
Table 2 summarizes the relevant absolute spectral quotients and nonresonance
conditions, as deduced from Theorem \ref{theo:SSM forced}, for specific
choices of the spectral subspace $E$.

\begin{center}
\begin{tabular}{|c|c|c|c|}
\hline 
 & \textbf{\small{}Fast SSM} & \textbf{\small{}Intermediate SSM} & \textbf{\small{}Slow SSM}\tabularnewline
\hline 
\hline 
{\footnotesize{}$E$ } & {\footnotesize{}$E_{N-q+1,\ldots,N}$} & {\footnotesize{}$E_{j_{1},\ldots,j_{q}}$} & {\footnotesize{}$E_{1,\ldots,q}$}\tabularnewline
\hline 
{\footnotesize{}$\Sigma(E)$} & {\footnotesize{}$\mathrm{Int}\left[\mathrm{Re}\lambda_{N}/\mathrm{Re}\lambda_{N-q+1}\right]$} & {\footnotesize{}$\mathrm{Int}\left[\mathrm{Re}\lambda_{N}/\mathrm{Re}\lambda_{j_{q}}\right]$} & {\footnotesize{}$\mathrm{Int}\left[\mathrm{Re}\lambda_{N}/\mathrm{Re}\lambda_{1}\right]$}\tabularnewline
\hline 
\multirow{2}{*}{{\footnotesize{}Nonresonance: }} & {\footnotesize{}$\sum_{i=N-q+1}^{N}a_{i}\mathrm{Re}\lambda_{i}\neq\mathrm{Re}\lambda_{l}$} & {\footnotesize{}$\sum_{i=1}^{q}a_{i}\mathrm{Re}\lambda_{j_{i}}\neq\mathrm{Re}\lambda_{l}$} & $\sum_{i=1}^{q}a_{i}\mathrm{Re}\lambda_{i}\neq\mathrm{Re}\lambda_{l}$\tabularnewline
\cline{2-4} 
 & {\footnotesize{}$\left|a\right|\in[2,\Sigma(E)],\,\,\,\,l\in\left[1,N-q\right]$ } & {\footnotesize{}$\left|a\right|\in[2,\Sigma(E)],\,\,\,\,l\in\left[j_{q+1},j_{N}\right]$ } & {\footnotesize{}$\left|a\right|\in[2,\Sigma(E)],\,\,\,\,l\in\left[q+1,N\right]$ }\tabularnewline
\hline 
{\footnotesize{}$W(x_{\epsilon}(t))$} & {\footnotesize{}$W_{N-q+1,\ldots,N}(x_{\epsilon}(t))$} & {\footnotesize{}$W_{j_{1},\ldots,j_{q}}(x_{\epsilon}(t))$} & {\footnotesize{}$W_{1,\ldots,q}(x_{\epsilon}(t))$}\tabularnewline
\hline 
\end{tabular}
\par\end{center}

\begin{center}
Table 2: Conditions for different types of non-autonomous SSMs appearing
in Theorem \ref{theo:SSM forced}, with parameters $a\in\mathbb{N}^{q}$
and $l\in\mathbb{N}$.
\par\end{center}

Much of our general discussion after Theorem \ref{theo:SSM unforced}
on the various choices of $E$ remains valid in the present non-autonomous
context, with two main differences. First, even the existence of fast
SSMs now requires a low-order non-resonance condition (cf. the first
column of Table 2). Accordingly, a non-autonomous fast SSM is only
guaranteed to be unique among at least $C^{\Sigma(E)+1}$ smooth invariant
manifolds. Second, Table 2 only requires the real parts of the eigenvalues
inside $E$ to be in non-resonance with the real parts of those outside
$E$. Resonances, therefore, occur with a larger likelihood than those
listed for the autonomous case in Table 1, since they now only involve
a condition on the real parts of the eigenvalues. 

Being as far as possible from resonances is also more important here
than in the autonomous case, as the exact nonresonance condition ensuring
the convergence of the Taylor approximation for non-autonomous SSMs
is not explicitly known. Rather, this condition is only known to be
$\mathcal{O}(\epsilon)$ close to that listed in the appropriate column
of Table 2. This is because the spectrum of the infinite-dimensional
transfer operator arising in the proof of the Theorem \ref{theo:SSM forced}
is generally only computable for $\epsilon=0$, giving the nonresonance
conditions listed in Table 2 (see Appendix \ref{sub:proof SSM forced}).

As in the autonomous case, one might ask if the existence of SSMs
guaranteed by Theorem \ref{theo:SSM forced} could also be deduced
directly from more classical dynamical systems results (cf. our related
discussion in Appendices \ref{sub:NHIM proof unforced} and \ref{sub:Poincare proof unforced}
for the autonomous case). It turns out that the shortcomings of Fenichel's
invariant manifold theorem would be the same as in the autonomous
case, while the non-autonomous extensions of Poincare's analytic linearization
theorem would be even more restrictive than in the autonomous case
(cf. Appendices \ref{sub:NHIM proof forced} and \ref{sub:Poincare proof forced}
for details.)
\begin{example}
{[}\emph{Periodic SSM from the application of Theorem }\ref{theo:SSM forced}{]}
Consider a periodically forced version of Example 1, given by 
\begin{eqnarray}
\dot{x} & = & -x,\nonumber \\
\dot{y} & = & -\sqrt{24}y+x^{2}+x^{3}+x^{4}+x^{5}+\epsilon\sin\Omega_{1}t,\label{eq:resonant example-1-1}
\end{eqnarray}
 with $\Omega_{1}=1.$ This system is analytic in all variables, and
hence we again have $r=a$ in the notation of Theorem \ref{theo:unforced NNM}.
The same non-resonance conditions are satisfied as in Example 1. Therefore,
Theorem \ref{theo:SSM forced} guarantees the existence of an analytic
(i.e, class $C^{a})$ slow SSM, $W_{1}(x_{\epsilon}(t))$ that is
unique among all class $C^{5}$ (in $x$) invariant manifolds tangent
to the horizontal axis along the NNM. Near the origin, this slow SSM
is guaranteed to be of the form 
\begin{equation}
y=h(x,t)=a_{0}(t)+a_{2}(t)x^{2}+a_{3}(t)x^{3}+a_{4}(t)x^{4}+a_{5}(t)x^{5}+\ldots,\qquad a_{j}(t+2\pi)=a_{j}(t).\label{eq:analyticSSM-1}
\end{equation}
This is the minimal Taylor expansion that only exists for the analytic
SSM but not for the other invariant manifolds tangent to the slow
subbundle along the NNM. Differentiation of \eqref{eq:analyticSSM-1}
in time gives
\begin{eqnarray}
\dot{y} & = & \dot{a}_{0}+\left[2a_{2}x+3a_{3}x^{2}+4a_{4}x^{3}+5a_{5}x^{4}+\mathcal{O}(x^{5})\right]\dot{x}+\dot{a}_{2}x^{2}+\dot{a}_{3}x^{3}+\dot{a}_{4}x^{4}+\dot{a}_{5}x^{5}+\mathcal{O}(x^{6})\nonumber \\
 & = & \left(\dot{a}_{2}-2a_{2}\right)x^{2}+\left(\dot{a}_{3}-3a_{3}\right)x^{3}+\left(\dot{a}_{4}-4a_{4}\right)x^{4}+\left(\dot{a}_{5}-5a_{5}\right)x^{5}+\mathcal{O}(x^{6}),\label{eq:resonant eq ydot1-1-1}
\end{eqnarray}
while substitution of \eqref{eq:analyticSSM-1} into the second equation
in \eqref{eq:resonant example-1-1} gives
\begin{equation}
\dot{y}=-\sqrt{24}a_{0}+\epsilon\sin t+\left(1-\sqrt{24}a_{2}\right)x^{2}+\left(1-\sqrt{24}a_{3}\right)x^{3}+\left(1-\sqrt{24}a_{4}\right)x^{4}+\left(1-\sqrt{24}a_{5}\right)x^{5}+\mathcal{O}(x^{6}).\label{eq:resonant eq ydot2-1-1}
\end{equation}
Equating \eqref{eq:resonant eq ydot1-1-1} and \eqref{eq:resonant eq ydot2-1-1}
gives
\begin{equation}
\dot{a}_{0}=-\sqrt{24}a_{0}+\epsilon\sin t,\quad\dot{a}_{j}=\left(j-\sqrt{24}\right)a_{j}+1,\qquad j=2,3,4,5,\quad\dot{a}_{k}=\left(j-\sqrt{24}\right)a_{k},\quad k\geq6.\label{eq:slowmanifold coefficients-1}
\end{equation}
 The requirement of $2\pi$-periodicity on $a_{i}(t)$ given in \eqref{eq:analyticSSM-1}
defines a boundary-value problem for the ODEs in \eqref{eq:slowmanifold coefficients-1},
whose unique solutions are
\[
a_{0}(t)=\epsilon\frac{\sqrt{24}}{25}\left(\sin t-\frac{1}{\sqrt{24}}\cos t\right),\quad a_{j}(t)\equiv\frac{1}{\sqrt{24}-j},\qquad j=2,3,4,5,\quad a_{k}(t)\equiv0,\quad k\geq6
\]
Just as in Example 1, the ODE \eqref{eq:resonant example-1-1} is
explicitly solvable: a direct integration gives $x(t)$ which, upon
substitution into the $y$ equation, yields an inhomogeneous linear
ODE for $y(t).$ Combining the expressions for $x(t)$ and $y(t)$
gives the solutions in the form 
\begin{eqnarray}
y(x;x_{0},y_{0},t) & = & K(x_{0},y_{0})x^{\sqrt{24}}+\frac{x^{2}}{\sqrt{24}-2}+\frac{x^{3}}{\sqrt{24}-3}+\frac{x^{4}}{\sqrt{24}-4}+\frac{x^{5}}{\sqrt{24}-5}\label{eq:y-x-periodic example}\\
 &  & +\epsilon\frac{\sqrt{24}}{25}\left[\sin t-\frac{1}{\sqrt{24}}\cos t\right],\nonumber \\
K(x_{0},y_{0},t_{0}) & = & \frac{y_{0}-\epsilon\frac{\sqrt{24}}{25}\left(\sin t_{0}-\frac{1}{\sqrt{24}}\cos t_{0}\right)}{x_{0}^{\sqrt{24}}}-\frac{x_{0}^{2-\sqrt{24}}}{\sqrt{24}-2}-\frac{x_{0}^{3-\sqrt{24}}}{\sqrt{24}-3}-\frac{x_{0}^{4-\sqrt{24}}}{\sqrt{24}-4}-\frac{x_{0}^{5-\sqrt{24}}}{\sqrt{24}-5},\nonumber 
\end{eqnarray}
with $(x_{0},y_{0})$ denoting an arbitrary initial conditions for
the solution at the initial time $t_{0}$.

This confirms the existence of a unique periodic NNM guaranteed by
Theorem \ref{theo:forced NNM}. Specifically, 
\[
\left(\begin{array}{c}
x_{\epsilon}(t)\\
y_{\epsilon}(t)
\end{array}\right)=\epsilon\tau(\Omega_{1}t;\epsilon)=\epsilon\left(\begin{array}{c}
0\\
\frac{\sqrt{24}}{25}\left[\sin t-\frac{1}{\sqrt{24}}\cos t\right]
\end{array}\right)
\]
attracts all solutions from SSM, which can be seen with $x=x_{0}e^{-t}$
substituted into \eqref{eq:y-x-periodic example}. The graph $y(\,\cdot\,;x_{0},y_{0};t)$
is a time-dependent representation of all invariant manifolds tangent
to the slow subbundle of this NNM, which is parallel to the $x$ axis.
As in the autonomous case, these invariant manifolds are generally
only of class $C^{4}$, because the term $K(x_{0},y_{0},t_{0})x^{\sqrt{24}}$
admits only four continuous derivatives along the NNM (which satisfies
$x\equiv0$). The only exception is the case $K(x_{0},y_{0})=0,$
for which $y(\,\cdot\,;x_{0},y_{0})$ becomes a quintic polynomial
in $x$ plus sine and cosine functions of $t$, all which are analytic.
But $K(x_{0},y_{0})=0$ holds only along the points 
\begin{equation}
y_{0}(x_{0},t_{0})=\frac{x_{0}^{2}}{\sqrt{24}-2}+\frac{x_{0}^{3}}{\sqrt{24}-3}+\frac{x_{0}^{4}}{\sqrt{24}-4}+\frac{x_{0}^{5}}{\sqrt{24}-5}+\epsilon\frac{\sqrt{24}}{25}\left(\sin t_{0}-\frac{1}{\sqrt{24}}\cos t_{0}\right),\label{eq:SSMexample-1}
\end{equation}
which lie precisely on the $W_{1}(\tau_{\epsilon}(t_{0}))$ slice
(or fiber) of the SSM, $W_{1}(\tau_{\epsilon}(t))$, whose Taylor
expansion we computed in \eqref{eq:slowmanifold coefficients-1}.
We show the unique analytic SSM for this example in Fig. \ref{fig:resonant_periodic}.
\begin{figure}[H]
\centering{}\includegraphics[width=0.9\textwidth]{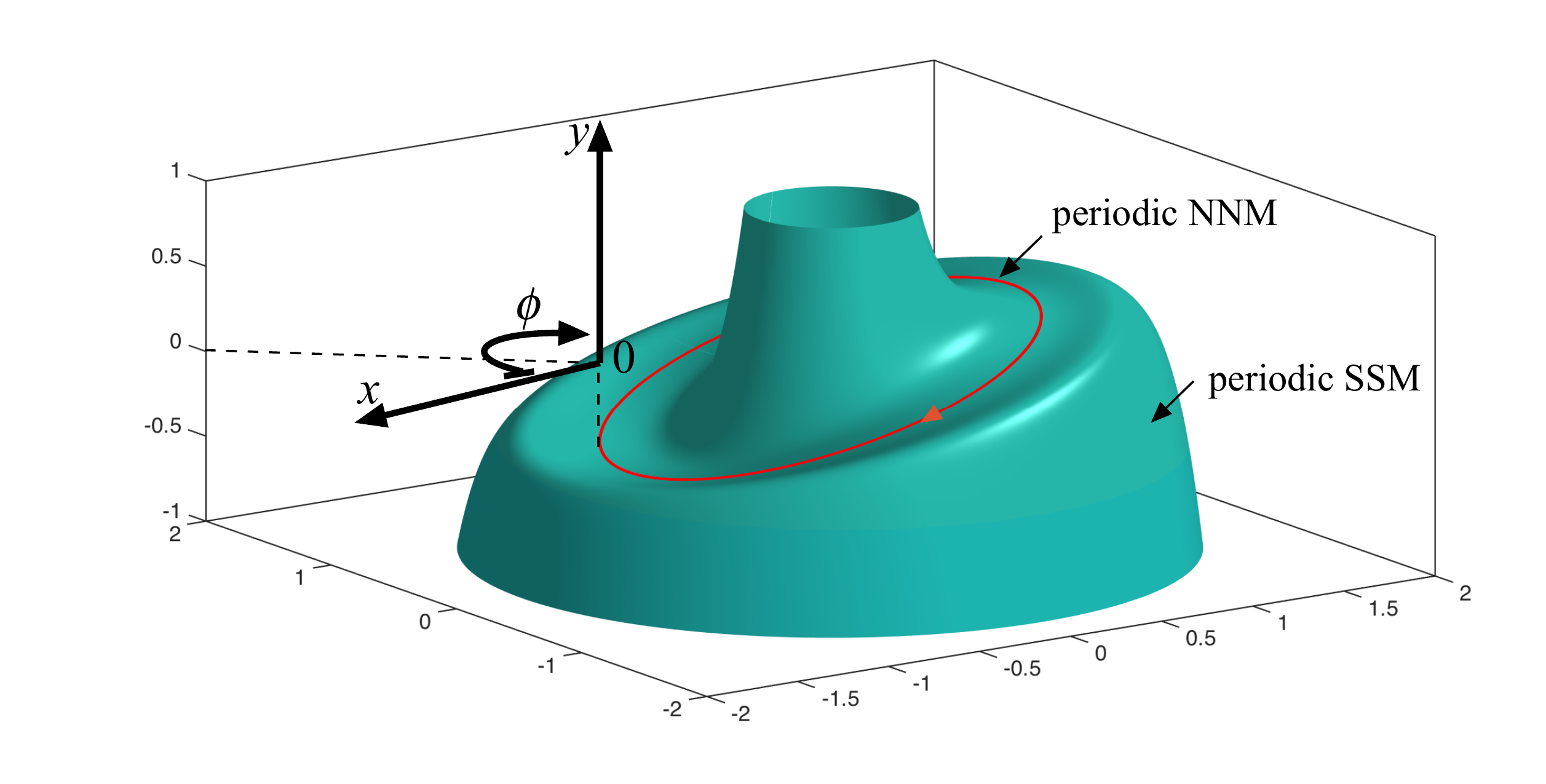}\caption{Phase portrait of system \eqref{eq:resonant example-1-1} in the extended
phase space of $(x,y,\phi)$, where $\phi=t\mod2\pi$. The green surface
is the unique analytic SSM guaranteed by Theorem \ref{theo:SSM forced},
emanating from the unique NNM (red) guaranteed by Theorem \ref{theo:forced NNM}.
The forcing parameter is selected as $\epsilon=2.$\label{fig:resonant_periodic}}
\end{figure}

\end{example}

\begin{example}
{[}\emph{Quasiperiodic SSM from the application of Theorem }\ref{theo:SSM forced}{]}
Consider the system 
\begin{eqnarray}
\dot{x} & = & -x,\nonumber \\
\dot{y} & = & -\sqrt{24}y+x^{2}+x^{3}+x^{4}+x^{5}+\epsilon(\sin\Omega_{1}t+\sin\Omega_{2}t),\label{eq:resonant example-1-1-1}
\end{eqnarray}
with $\Omega_{1}=1$ and $\Omega_{2}=\sqrt{2}$. This is just the
quasi-periodically forced version of Example 3. Based on the same
reasoning as in that example, we conclude from Theorem \ref{theo:SSM forced}
the existence of a unique quasiperiodic SSM in the form 
\begin{eqnarray}
y & = & h(x,\phi_{1},\phi_{2})=a_{0}(\phi_{1},\phi_{2})+a_{2}(\phi_{1},\phi_{2})x^{2}+a_{3}(\phi_{1},\phi_{2})x^{3}\label{eq:analyticSSM-1-1}\\
 &  & +a_{4}(\phi_{1},\phi_{2})x^{4}+a_{5}(\phi_{1},\phi_{2})x^{5}+\ldots,\nonumber \\
a_{j}(\phi_{1},\phi_{2}) & = & a_{j}(\phi_{1}+2\pi/\Omega_{1},\phi_{2}),\qquad a_{j}(\phi_{1},\phi_{2})=a_{j}(\phi_{1},\phi_{2}+2\pi/\sqrt{2}),\nonumber 
\end{eqnarray}
with the phase variables satisfying $\dot{\phi}_{1}=1,$ $\dot{\phi}_{2}=\sqrt{2}$.
Differentiation of \eqref{eq:analyticSSM-1-1} in time gives
\begin{eqnarray}
\dot{y} & = & \dot{a}_{0}+\left[2a_{2}x+3a_{3}x^{2}+4a_{4}x^{3}+5a_{5}x^{4}+\mathcal{O}(x^{5})\right]\dot{x}+\dot{a}_{2}x^{2}+\dot{a}_{3}x^{3}+\dot{a}_{4}x^{4}+\dot{a}_{5}x^{5}+\mathcal{O}(x^{6})\nonumber \\
 & = & \left(\dot{a}_{2}-2a_{2}\right)x^{2}+\left(\dot{a}_{3}-3a_{3}\right)x^{3}+\left(\dot{a}_{4}-4a_{4}\right)x^{4}+\left(\dot{a}_{5}-5a_{5}\right)x^{5}+\mathcal{O}(x^{6})\label{eq:resonant eq ydot1-1-1-1}
\end{eqnarray}
while substitution of \eqref{eq:analyticSSM-1-1} into the second
equation in \eqref{eq:resonant example-1-1-1} gives
\begin{eqnarray}
\dot{y} & = & -\sqrt{24}a_{0}+\epsilon\left(\sin t+\sin\sqrt{2}t\right)+\left(1-\sqrt{24}a_{2}\right)x^{2}+\left(1-\sqrt{24}a_{3}\right)x^{3}\label{eq:resonant eq ydot2-1-1-1}\\
 &  & +\left(1-\sqrt{24}a_{4}\right)x^{4}+\left(1-\sqrt{24}a_{5}\right)x^{5}+\mathcal{O}(x^{6}).
\end{eqnarray}
Equating \eqref{eq:resonant eq ydot1-1-1-1} and \eqref{eq:resonant eq ydot2-1-1-1}
gives
\begin{eqnarray}
\dot{a}_{0} & = & -\sqrt{24}a_{0}+\epsilon\left(\sin t+\sin\sqrt{2}t\right),\quad\dot{a}_{j}=\left(j-\sqrt{24}\right)a_{j}+1,\quad j\in[2,5],\quad\dot{a}_{k}=\left(j-\sqrt{24}\right)a_{k},\nonumber \\
\dot{a}_{k} & = & \left(j-\sqrt{24}\right)a_{k},\quad k\geq6.\label{eq:slowmanifold coefficients-1-1}
\end{eqnarray}
The quasi-periodicity requirements on $(\phi_{1},\phi_{2})$ given
in \eqref{eq:analyticSSM-1-1} define a boundary-value problems for
the PDEs in \eqref{eq:slowmanifold coefficients-1-1}, whose unique
solutions are
\begin{eqnarray*}
a_{0}(\phi_{1},\phi_{2}) & = & \epsilon\left(\frac{\sqrt{24}}{25}\left[\sin\phi_{1}-\frac{1}{\sqrt{24}}\cos\phi_{1}\right]+\frac{\sqrt{24}}{26}\left[\sin\phi_{2}-\frac{\sqrt{2}}{\sqrt{24}}\cos\phi_{2}\right]\right),\\
a_{j}(\phi_{1},\phi_{2}) & \equiv & \frac{1}{\sqrt{24}-j},\quad j=2,3,4,5,\\
a_{k}(\phi_{1},\phi_{2}) & \equiv & 0,\quad k\geq6.
\end{eqnarray*}
At the same time, just as in Example 3, the ODE \eqref{eq:resonant example-1-1-1}
is explicitly solvable: a direct integration gives $x(t)$ which,
upon substitution into the $y$ equation, yields an inhomogeneous
linear ODE for $y(t).$ Combining the expressions for $x(t)$ and
$y(t)$ gives the solutions in the form 
\begin{eqnarray*}
y(x;x_{0},y_{0},t) & = & K(x_{0},y_{0})x^{\sqrt{24}}+\frac{x^{2}}{\sqrt{24}-2}+\frac{x^{3}}{\sqrt{24}-3}+\frac{x^{4}}{\sqrt{24}-4}+\frac{x^{5}}{\sqrt{24}-5}\\
 &  & +\epsilon\left(\frac{\sqrt{24}}{25}\left[\sin t-\frac{1}{\sqrt{24}}\cos t\right]+\frac{\sqrt{24}}{26}\left[\sin\sqrt{2}t-\frac{\sqrt{2}}{\sqrt{24}}\cos\sqrt{2}t\right]\right),\\
K(x_{0},y_{0},t_{0}) & = & \frac{y_{0}-\epsilon\left(\frac{\sqrt{24}}{25}\left[\sin t_{0}-\frac{1}{\sqrt{24}}\cos t_{0}\right]+\frac{\sqrt{24}}{26}\left[\sin\sqrt{2}t_{0}-\frac{\sqrt{2}}{\sqrt{24}}\cos\sqrt{2}t_{0}\right]\right)}{x_{0}^{\sqrt{24}}}\\
 &  & -\frac{x_{0}^{2-\sqrt{24}}}{\sqrt{24}-2}-\frac{x_{0}^{3-\sqrt{24}}}{\sqrt{24}-3}-\frac{x_{0}^{4-\sqrt{24}}}{\sqrt{24}-4}-\frac{x_{0}^{5-\sqrt{24}}}{\sqrt{24}-5},
\end{eqnarray*}
with $(x_{0},y_{0})$ denoting an arbitrary initial condition for
the solution at the initial time $t_{0}$.

Again, all these solutions decay exponentially to a unique quasiperiodic
NNM given by
\[
\left(\begin{array}{c}
x_{\epsilon}(t)\\
y_{\epsilon}(t)
\end{array}\right)=\epsilon\tau(\Omega_{1},\Omega_{2}t;\epsilon)=\epsilon\left(\begin{array}{c}
0\\
\frac{\sqrt{24}}{25}\left[\sin t-\frac{1}{\sqrt{24}}\cos t\right]+\frac{\sqrt{24}}{26}\left[\sin\sqrt{2}t-\frac{\sqrt{2}}{\sqrt{24}}\cos\sqrt{2}t\right]
\end{array}\right).
\]
The graph $y(\,\cdot\,;x_{0},y_{0};t)$ is a time-dependent representation
of all invariant manifolds tangent to the slow subbundle of this NNM.
Any $t=t_{0}$ slice of this subbundle is parallel to the $x$ axis,
i.e., to the spectral subspace $E_{1}$. As in the autonomous case,
these invariant manifolds are generally only of class $C^{4}$, because
the term $K(x_{0},y_{0},t_{0})x^{\sqrt{24}}$ admits only four continuous
derivatives along the NNM (which satisfies $x\equiv0$). The only
exception is the case $K(x_{0},y_{0})=0,$ for which $y(\,\cdot\,;x_{0},y_{0})$
becomes a quintic polynomial in $x$ plus sine and cosine functions
of $t$, all which are analytic. But $K(x_{0},y_{0})=0$ holds only
along the points 
\begin{eqnarray}
y_{0}(x_{0},t_{0}) & = & \frac{x_{0}^{2}}{\sqrt{24}-2}+\frac{x_{0}^{3}}{\sqrt{24}-3}+\frac{x_{0}^{4}}{\sqrt{24}-4}+\frac{x_{0}^{5}}{\sqrt{24}-5}\label{eq:SSMexample-1-1}\\
 &  & +\epsilon\left(\frac{\sqrt{24}}{25}\left[\sin t_{0}-\frac{1}{\sqrt{24}}\cos t_{0}\right]+\frac{\sqrt{24}}{26}\left[\sin\sqrt{2}t_{0}-\frac{\sqrt{2}}{\sqrt{24}}\cos\sqrt{2}t_{0}\right]\right),\nonumber 
\end{eqnarray}
which lie precisely on the $t=t_{0}$ slice (fiber) of the SSM, $W_{1}(x_{\epsilon}(t))$
whose Taylor expansion we computed in \eqref{eq:slowmanifold coefficients-1-1}.
We show the unique analytic SSM for this example in Fig. \ref{fig:resonant_quasiperiodic}.
\begin{figure}[H]
\centering{}\includegraphics[width=0.9\textwidth]{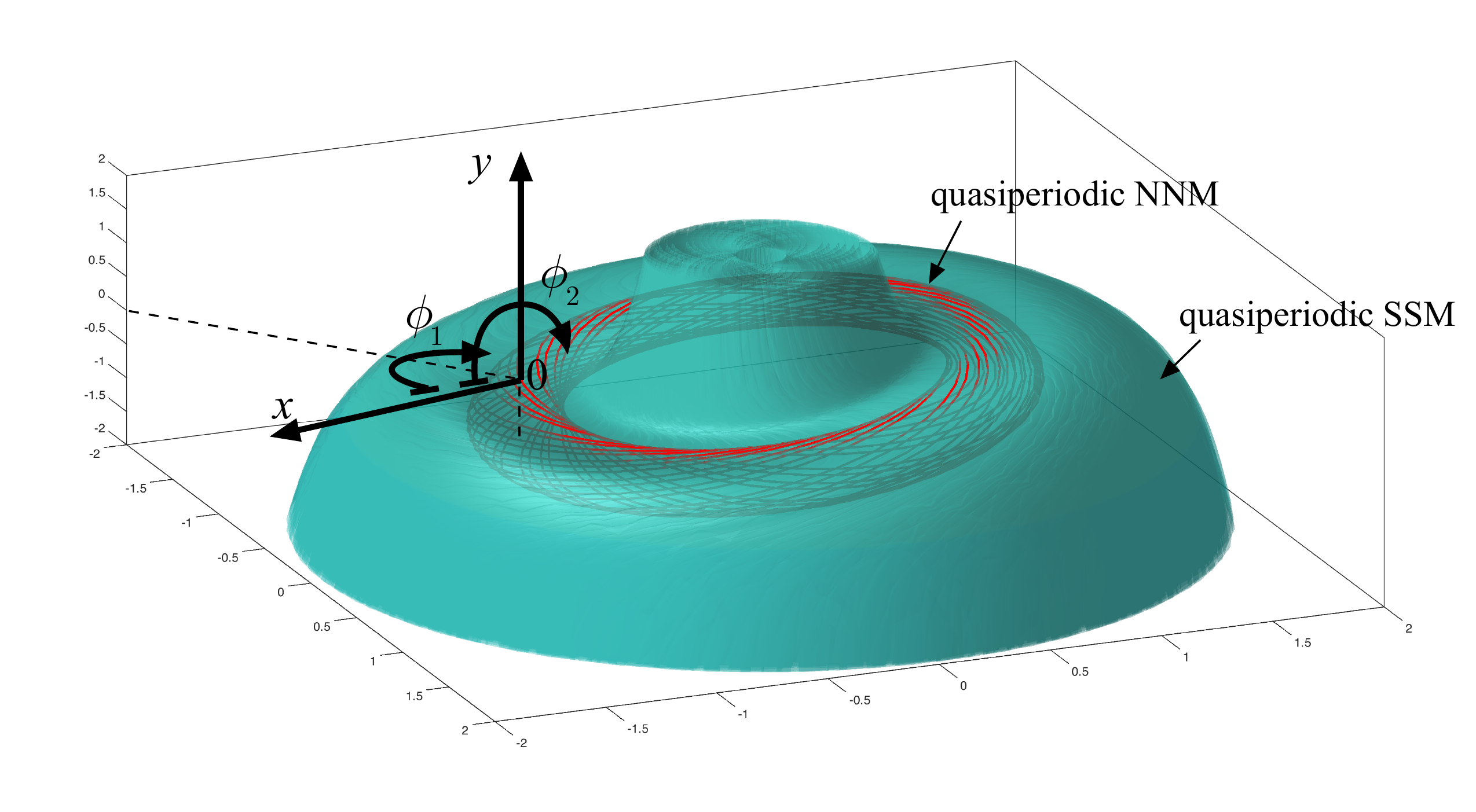}\caption{A projection of system \eqref{eq:resonant example-1-1-1} from the
extended phase space of $(x,y,\phi_{1},\phi_{2})$, where $\phi_{1}=\left(\phi_{10}+\Omega_{1}t\right)\mod2\pi$
and $\phi_{2}=\left(\phi_{20}+\Omega_{2}t\right)\mod2\pi/\sqrt{2}$.
The green surface is the unique analytic, quasiperiodic SSM guaranteed
by Theorem \ref{theo:SSM forced}, emanating from the unique quasiperiodic
NNM (red) guaranteed by Theorem \ref{theo:forced NNM}. The forcing
parameter is selected as $\epsilon=2.$ The specific projection used
in this visualization is $(x,\phi_{1},\phi_{2},y)\protect\mapsto\left(x+0.2\cos\phi_{2}+1)\cos\phi_{1},(x+0.2\cos\phi_{2}+1)\sin\phi_{1},y(x,\phi_{1},\phi_{2})\right).$
\label{fig:resonant_quasiperiodic}}
\end{figure}

\end{example}

\begin{example}
\label{ex:shaw-pierre2}{[}\emph{Illustration of Theorem \ref{theo:SSM forced}
on a mechanical example}{]} As a last example, we reconsider here
Example \ref{ex:shaw-pierre1} with time-dependent forcing. First,
we illustrate the application of Theorem \ref{theo:SSM forced} to
the general case of quasiperiodic forcing. Next, we restrict the forcing
to be periodic and compute the periodic NNM and slow periodic SSM
guaranteed by our results for this case. 

Fig. \ref{fig:shaw-pierre-model2} shows the two-degree-of freedom
system already featured in Fig. \ref{fig:shaw-pierre-model1}, but
now with multi-frequency parametric forcing 
\begin{equation}
\epsilon F_{1}=\epsilon\left(\begin{array}{c}
F_{11}(\Omega_{1}t,\ldots,\Omega_{k}t)\\
F_{12}(\Omega_{1}t,\ldots,\Omega_{k}t)
\end{array}\right)\label{eq:PS forcing}
\end{equation}
acting on both masses, with $k\geq1$ arbitrary frequencies All other
details remain the same as in Example \ref{ex:shaw-pierre1}.
\begin{figure}[H]
\begin{centering}
\includegraphics[width=0.6\textwidth]{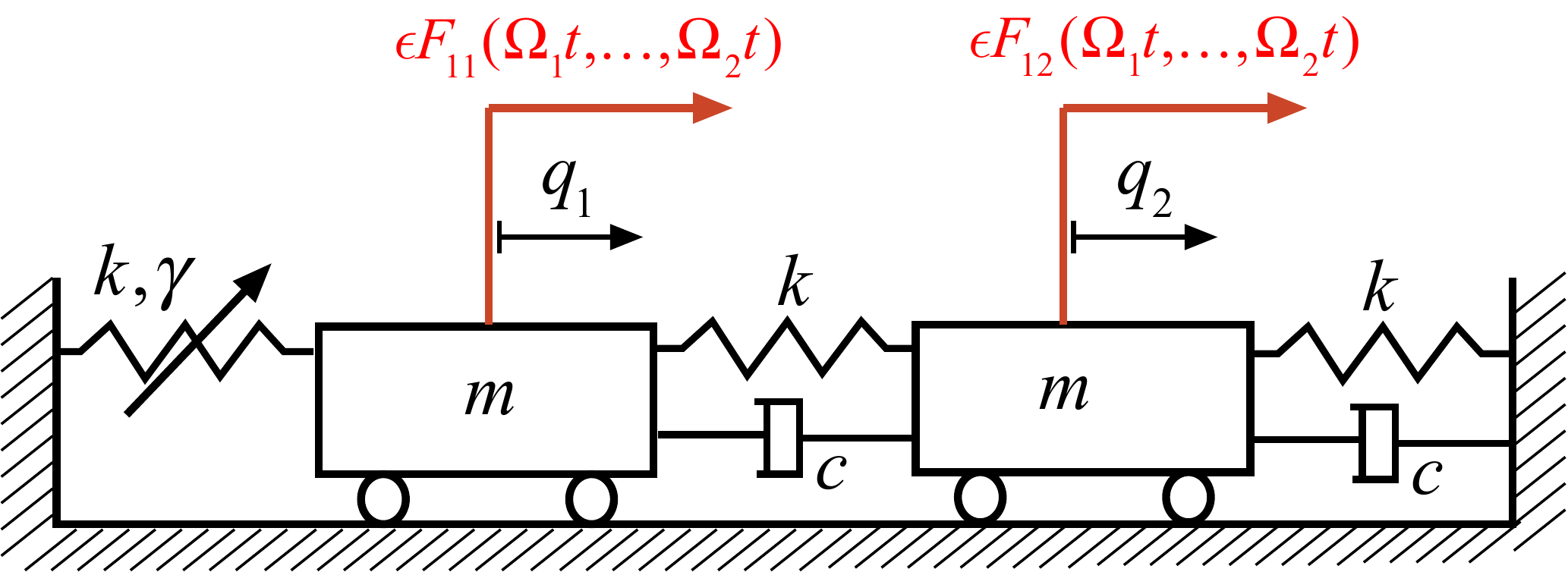}
\par\end{centering}

\caption{The quasiperiodically forced version of Example \ref{ex:shaw-pierre1}.\label{fig:shaw-pierre-model2}}
\end{figure}
The eigenvalues of the linearized, unforced system are again those
listed in \eqref{eq:SP eigenvalues}, yielding the absolute spectral
quotients 
\[
\Sigma(E_{1})=\mathrm{Int}\text{\ensuremath{\left[\frac{\mathrm{Re}\,\lambda_{2}}{\mathrm{Re}\,\lambda_{1}}\right]\emph{=5}},\quad\ensuremath{\Sigma}(\emph{\ensuremath{E_{2}}})=\ensuremath{\mathrm{Int}\text{\ensuremath{\left[\frac{\mathrm{Re}\,\lambda_{2}}{\mathrm{Re}\,\lambda_{2}}\right]}=1. }} }
\]
 By Table 2, the relevant nonresonance condition for the slow non-autonomous
SSM is
\[
-0.0741a_{1}\neq-0.3759,\quad a_{1}=2,3,4,5,
\]
which is very close to being satisfied for $a_{1}=5$. This means
that the existence of a non-autonomous SSM can only be concluded from
Theorem \ref{theo:SSM forced} for very small values of $\epsilon$.
Whenever it exists, the slow SSM is still analytic in the positions
and velocities, and unique among invariant manifolds that are at least
of class $C^{6}$ in these variables. The dependence of this uniqueness
class on the parameters is identical to that shown in Fig. \ref{fig:required order of expansion for SP1}.
As for the fast SSM, Table 2 shows that no non-resonance conditions
are required, because $\Sigma(E_{2})<2.$ This time, however, the
fast SSM can only be concluded to be unique in the function class
$C^{2}$, given that $\Sigma(E_{2})=1$. 

For simplicity, we now restrict our discussion to time-periodic forcing
by selecting the forcing terms \eqref{eq:PS forcing} as 
\end{example}
\[
\epsilon F_{1}=\left(\begin{array}{c}
0\\
\text{sin}(\Omega_{1}t)
\end{array}\right),
\]
with $\Omega_{1}=1$. As concluded above already for more general
forcing terms, Theorem \ref{theo:SSM forced} guarantees the existence
of a unique analytic slow SSM, $W_{1}(x_{\epsilon}(t))$, for $\epsilon>0$
small enough. This SSM is already unique among class $C^{6}$ invariant
manifolds tangent to the slow spectral subbundle of a small-amplitude,
periodic NNM, which is guaranteed to exist by Theorem \ref{theo:forced NNM}.
To compute this NNM and its slow SSM, we again use a linear change
of coordinates \eqref{eq:xyz} to obtain the equations of motion in
the form

\begin{eqnarray}
\dot{y} & = & A_{y}y+f_{0y}(y,z)+\epsilon f_{1y}(\Omega_{1}t;\epsilon)=\left(\begin{array}{cc}
-0.0741 & 1.0027\\
-1.0027 & -0.0741
\end{array}\right)y+\left(\begin{array}{c}
1.0148\\
-0.2162
\end{array}\right)p(y,z)\label{eq:ysep_per}\\
 &  & +\epsilon\left(\begin{array}{c}
1.0016\\
-0.0660
\end{array}\right)\text{\ensuremath{\sin}\,}t,\nonumber 
\end{eqnarray}

\begin{eqnarray}
\dot{z} & = & A_{z}z+f_{0z}(y,z)+\epsilon f_{1z}(\Omega_{1}t;\epsilon)=\left(\begin{array}{cc}
-0.3759 & 1.6812\\
-1.6812 & -0.3759
\end{array}\right)z+\left(\begin{array}{c}
0.8046\\
-0.1685
\end{array}\right)p(y,z)\label{eq:zsep_per}\\
 &  & +\epsilon\left(\begin{array}{c}
-0.7987\\
0.3861
\end{array}\right)\text{\ensuremath{\sin}\,}t,\nonumber 
\end{eqnarray}

\[
p(y,z)=-0.5\left(-0.0374\thinspace y_{1}-0.5055\thinspace y_{2}-0.1526\thinspace z_{1}-0.3052\thinspace z_{2}\right)^{3}.
\]

First, we seek the unique periodic NNM of this system in the form
of a Taylor expansion in the perturbation-parameter $\epsilon$. By
Theorem \ref{theo:forced NNM}, this NNM can be written in the form

\begin{equation}
x_{\epsilon}(t)=\epsilon\tau_{1}(t)+\mathcal{O}(\epsilon^{2}).\label{eq:peturb}
\end{equation}
Substitution of this expression into \eqref{eq:ysep_per}-\eqref{eq:zsep_per}
and collection of the $\mathcal{O}(\epsilon)$ terms gives

\begin{equation}
\dot{\tau}_{1}(t)=A_{yz}\tau_{1}(t)+c\text{\ensuremath{\cdot\sin}}t,\qquad A_{yz}=\left(\begin{array}{cc}
A_{y} & 0\\
0 & A_{z}
\end{array}\right),\qquad c=\left(\begin{array}{c}
1.0016\\
-0.0660\\
-0.7987\\
0.3861
\end{array}\right).\label{eq:inhomoglin}
\end{equation}
The unique, periodic particular solution of this inhomogeneous system
of linear differential equations can be sought in the form

\begin{equation}
\tau_{1}(t)=a\text{\ensuremath{\cdot\sin}}t+b\text{\ensuremath{\cdot\cos}}t,\quad a,b\in\mathbb{R}^{4}.\label{eq:q_part}
\end{equation}
Substituting (\ref{eq:q_part}) into (\ref{eq:inhomoglin}) gives
algebraic equations for the vectors $a$ and $b$, whose solutions
are explicitly computable as

\begin{eqnarray*}
 &  & a=-A_{yz}\left(A_{yz}^{2}+I\right)^{-1}c,\qquad b=-\left(A_{yz}^{2}+I\right)^{-1}c.
\end{eqnarray*}
With the relevant parameter values substituted into \eqref{eq:peturb},
we obtain the leading-order approximation of the attracting periodic
NNM in the form

\begin{equation}
x_{\epsilon}(t)=\epsilon\left(\begin{array}{c}
6.7213\\
-0.9408\\
0.0194\\
0.7253
\end{array}\right)\text{\ensuremath{\sin}}t+\epsilon\left(\begin{array}{c}
0.4402\\
6.7357\\
-0.4134\\
-0.0809
\end{array}\right)\text{\ensuremath{\cos}}t+\mathcal{O}(\epsilon^{2}).\label{eq:xepsilon}
\end{equation}

To obtain the unique slow SSM, $W_{1}(x_{\epsilon}(t))$, guaranteed
by Theorem \ref{theo:SSM forced}, we use the time-periodic Taylor
expansion 

\begin{equation}
z=h(y,t)=\sum_{\left|p\right|=0}^{6}h_{p}(t)y^{p},\quad p=\left(p_{1},p_{2}\right)\in\mathbb{N}^{2},\quad y^{p}=y_{1}^{p_{1}}y_{2}^{p_{2}},\quad h_{p}(t)=h_{p}(t+2\pi)\in\mathbb{R}^{2}.\label{eq:z poly_per}
\end{equation}
Differentiating (\ref{eq:z poly_per}) with respect to time and substituting
$\dot{y}$ and $\dot{z}$ from \eqref{eq:ysep_per}-\eqref{eq:zsep_per}
gives

\[
\frac{\partial h(y,t)}{\partial y}\left[A_{y}y+f_{0y}(y,h(y,t))+\epsilon f_{1y}(\Omega_{1}t;\epsilon)\right]+\sum_{\left|p\right|=0}^{6}\dot{h}_{p}y^{p}=A_{z}h(y,t)+f_{0z}\left(y,h(y,t)\right)+\epsilon f_{1z}(\Omega_{1}t;\epsilon).
\]

Comparing equal powers of $y$ in this last expression leads to a
set of coupled ODEs for $h_{p}(t)$. The 2$\pi$-periodicity requirement
on $h_{p}(t)$ given in (\ref{eq:z poly_per}) defines a boundary-value
problem for these ODEs, which we solve numerically. The slow SSM surface
obtained in this fashion is shown in the extended phase space in Fig.
\ref{fig:3D-proj-periodic-SSM}, along with the periodic NNM (red)
obtained in \eqref{eq:xepsilon}.

\begin{figure}[H]
\begin{centering}
\includegraphics[width=0.6\textwidth]{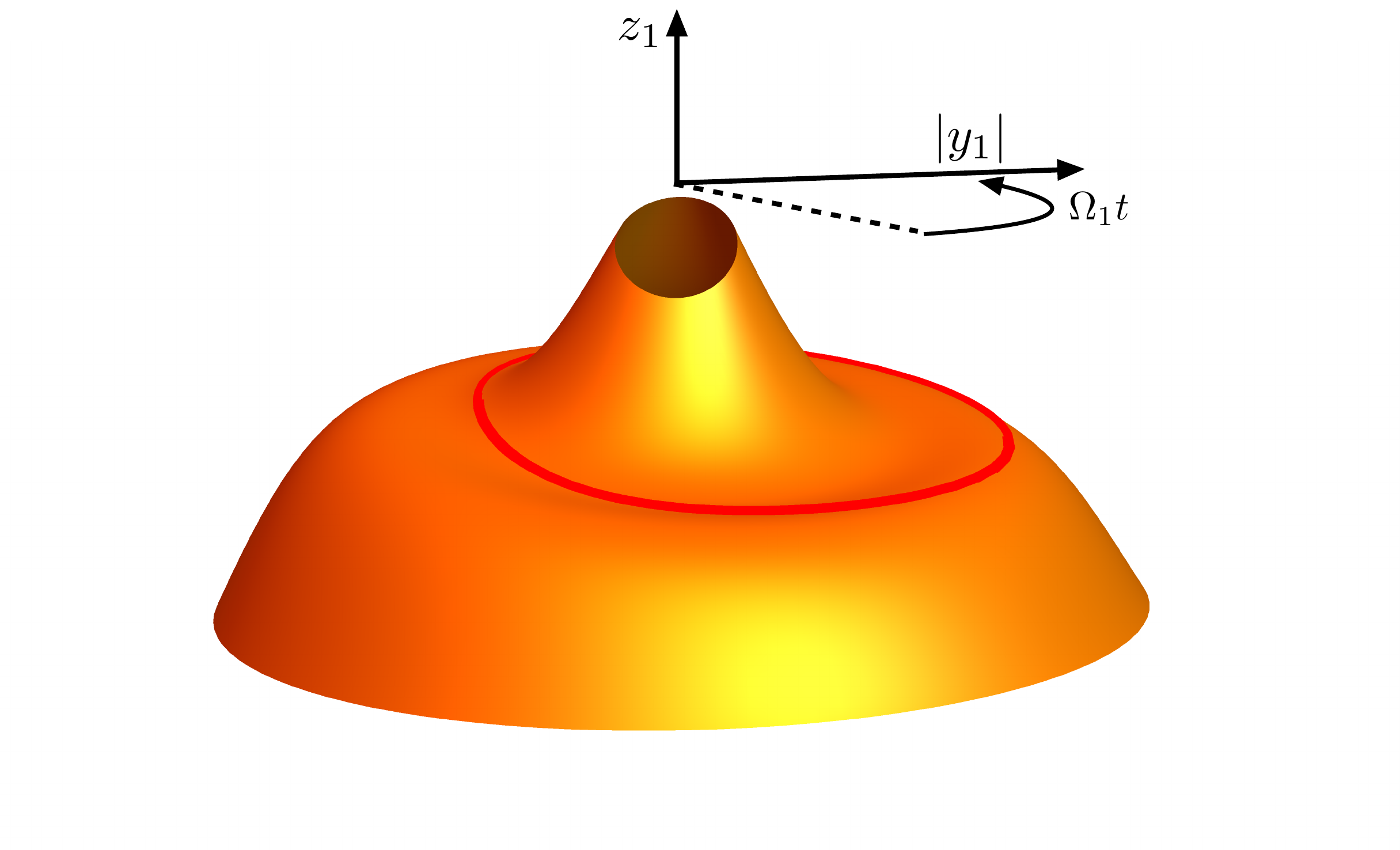}
\par\end{centering}

\caption{Projection of the SSM and NNM for the periodically forced Shaw--Pierre
example from the extended phase space of ($y$, $z$, $\Omega_{1}$t
mod $2\pi$). The forcing parameter is $\epsilon=0.1.$ (See the related
on-line supplemental movie for animation.)\label{fig:3D-proj-periodic-SSM}}
\end{figure}

As an alternative view, an instantaneous projection of the dynamics
on the slow SSM from the four-dimensional $(y_{1},y_{2},z_{1},z_{2})$
phase space is shown in Fig. \ref{fig:3D-proj-periodic-SSM-1}. 
\begin{figure}[H]
\begin{centering}
\includegraphics[width=0.6\textwidth]{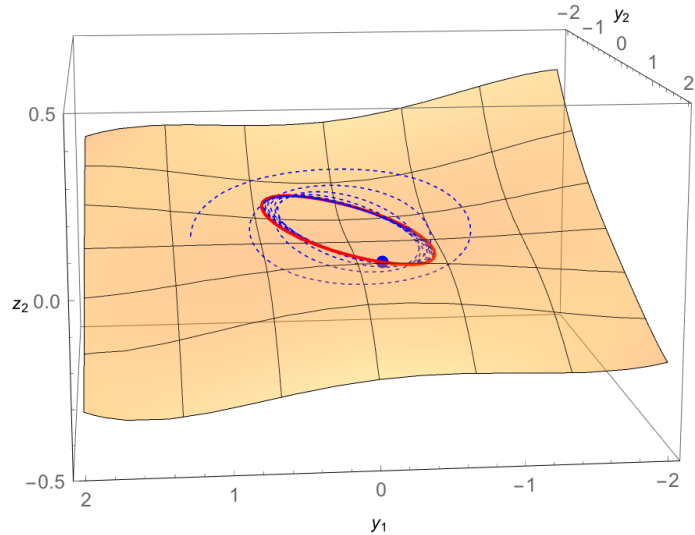}
\par\end{centering}

\caption{Instantaneous projection of the analytically computed periodic NNM
and slow SSM for the periodically forced Shaw--Pierre example from
the $(y_{1},y_{2},z_{1},z_{2})$ phase space for $\epsilon=0.1.$
Shown is an instantaneous position of the SSM surface along with the
history of a trajectory (blue) launched from the SSM at an earlier
time. Note that the trajectory has converged to the analytically computed
approximation to the NNM (red). We find the mean squared error between
the independently computed NNM and its projection onto the SSM to
be $\mathcal{O}(\epsilon^{3})$ over one time period. (See the related
on-line supplemental movie for animation.) \label{fig:3D-proj-periodic-SSM-1}.}
\end{figure}

\section{Relevance for model reduction\label{sec:modelreduction}}

\subsection{Expansions for NNMs\label{sub:Expansions-for-NNMs}}

Theorems \ref{theo:unforced NNM} and \ref{theo:forced NNM} provide
existence, uniqueness and robustness results for NNMs in both the
autonomous and the non-autonomous settings. Specifically, by Theorem
\ref{theo:unforced NNM}, the unique NNM $x_{\epsilon}$ in the autonomous
case ($k=0$) depends on $\epsilon$ in a $C^{r}$ fashion, and hence
can be approximated in the form of a Taylor series
\[
x_{\epsilon}=\epsilon\tau(\epsilon)=\sum_{l=1}^{r}\xi_{l}\epsilon^{l}+o\left(\epsilon^{r}\right),
\]
 with the vector $\xi_{l}\in\mathbb{R}^{N}$ denoting the $l^{th}$
order Taylor coefficient of the function $x_{\epsilon}$. 

In the non-autonomous case $(k>0)$, Theorem \ref{theo:forced NNM}
guarantees a unique NNM, $x_{\epsilon}(t)=\epsilon\tau(\Omega_{1}t,\ldots,\Omega_{k}t;\epsilon)$
in system \eqref{eq:1storder_system} that depends on $\epsilon$
in a $C^{r}$ fashion. Thus $x_{\epsilon}(t)$ can be approximated
in the form of a Taylor--Fourier series 
\begin{eqnarray*}
x_{\epsilon}(t) & = & \epsilon\tau(\Omega_{1}t,\ldots,\Omega_{k}t;\epsilon)=\sum_{l=1}^{r}\epsilon^{l}\xi_{l}(\Omega_{1}t,\ldots,\Omega_{k}t)+o\left(\epsilon^{r}\right)\\
 & = & \sum_{l=1}^{r}\sum_{\left|m\right|=1}^{\infty}\epsilon^{l}\left[A_{m}^{l}\sin(\left\langle m,\Omega\right\rangle t)+B_{m}^{l}\cos(\left\langle m,\Omega\right\rangle t)\right]+o\left(\epsilon^{r}\right),
\end{eqnarray*}
with the vectors $A_{m}^{l},B_{m}^{l}\in\mathbb{R}^{N}$ denoting
the multi-frequency Fourier coefficients of the function $x_{\epsilon}(t)$
corresponding to the multi-index $m=(m_{1},\ldots,m_{k})\in\mathbb{N}^{k}.$

\subsection{Expansions for slow SSMs}

Theorems \ref{theo:SSM unforced} and \ref{theo:SSM forced} provide
a theoretical underpinning for the construction of reduced-order models
over appropriately chosen spectral subspaces of the linearized system.
Specifically, approximations to the flow on an SSM may simplify the
study of long-term system dynamics. 

Of highest relevance for such model reduction are slow SSMs. Since
all linearized solutions decay to an NNM in our setting, slow SSMs
contain the trajectories that resist this trend as much as possible
and remain active for the longest time. These SSMs can be constructed
under the conditions spelled out in the last columns of Tables 1 and
2.

To approximate uniquely a slow SSM, we need to use a Taylor expansion
of at least order $\sigma(E)+1$ or $\Sigma(E)+1$, respectively.
This order depends solely on the damping rates associated with the
fastest- and slowest-decaying modes. Even if the real part of the
whole spectrum of $A$ is close to zero, $\sigma(E)$ and $\Sigma(E)$
may well be large, as seen in the mechanical systems considered in
Examples \ref{ex:shaw-pierre1} and \ref{ex:shaw-pierre2}.
\begin{example}
\label{ex:shaw-pierre2-1}{[}\emph{Illustration of model reduction
on a mechanical example}{]} Here we illustrate the relevance of slow
SSMs in model reduction for the unforced oscillator system in Example
\ref{ex:shaw-pierre1}. Figure \ref{fig:SSM_auto_diff_1} and Fig.
\ref{fig:SSM_auto_diff_2} show different visualization of the fast
convergence of a generic trajectory first to the slow SSM, then to
the stable equilibrium along the SSM.
\end{example}
\begin{figure}[H]
\centering{}\subfloat[\label{fig:SSM_auto_diff_1}]{\begin{centering}
\includegraphics[width=0.45\textwidth]{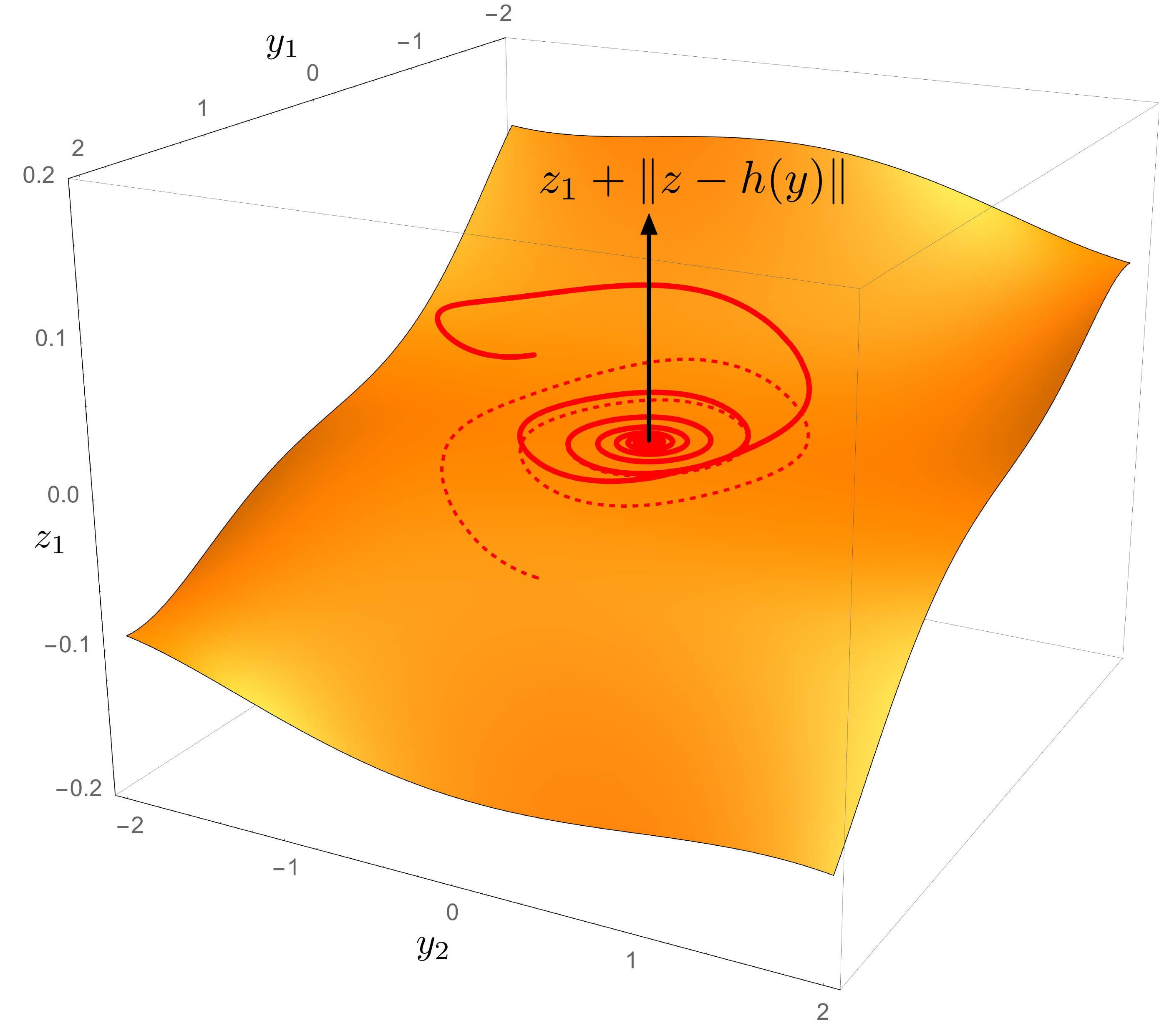}
\par\end{centering}

}\hfill{}\subfloat[\label{fig:SSM_auto_diff_2}]{\centering{}\includegraphics[width=0.5\textwidth]{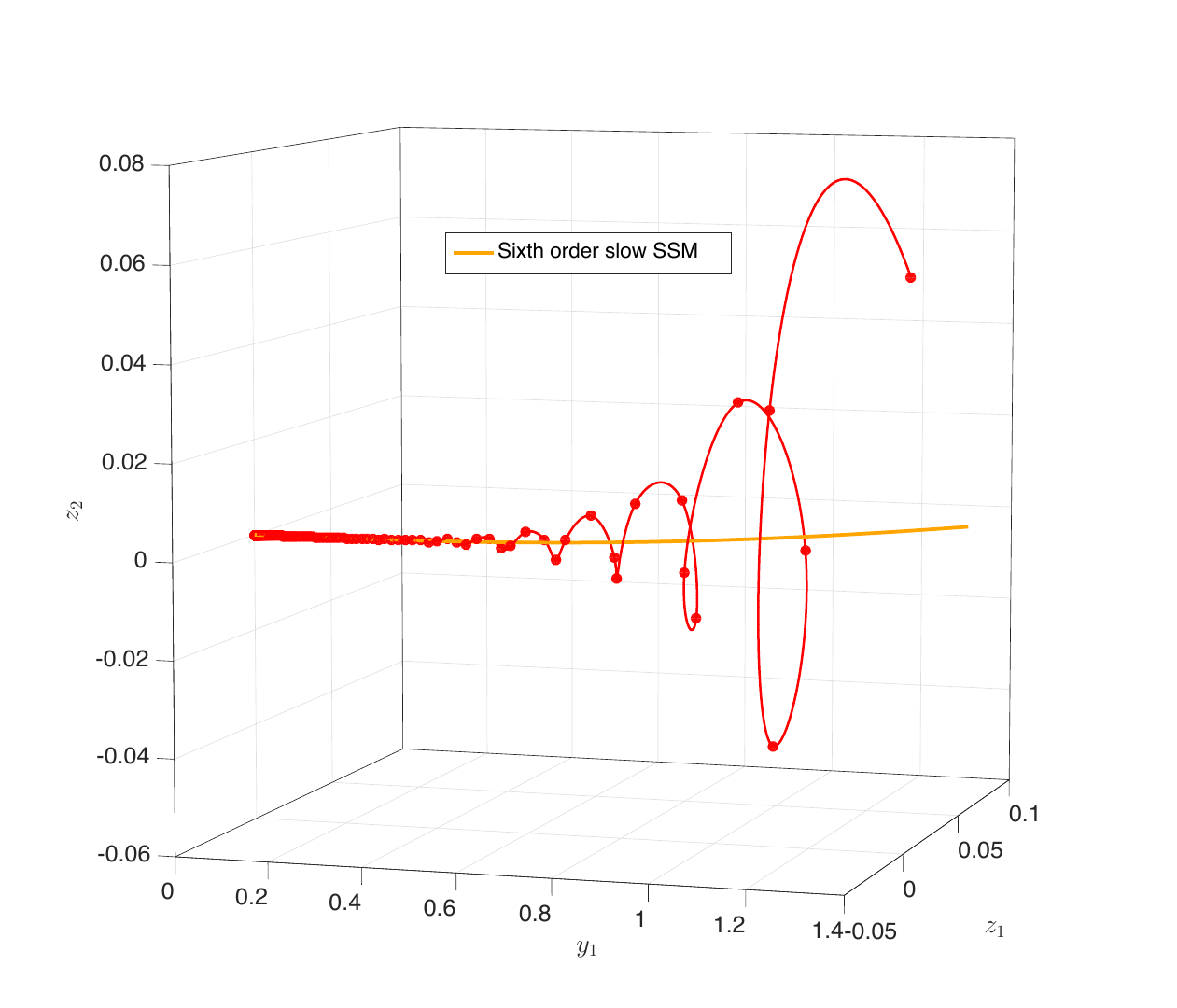}}\caption{(a) Fast convergence of a generic trajectory to the slow SSM, then
subsequently to the equilibrium point along the SSM. Initial conditions
for the trajectory were chosen off the SSM with the coordinates $y_{1}(0)=1.2$,
$y_{2}(0)=0$, $z_{1}(y_{1}(0),y_{2}(0))+\Delta z_{1}=-0.042+0.1$,
$z_{2}(y_{1}(0),y_{2}(0))+\Delta z_{2}=-0.0045+0.1$). The vertical
axis in the figure represents the difference between $(z_{1},z_{2})$
and $h(y_{1},y_{2})$, which decays in time due to attraction to the
SSM. (b) A different view on the same convergences shown by the Poincaré
map already used in Fig. \ref{fig:Poincare map 1}, with the damping
now decreased to $c=0.03$ to increase the number of intersection
with the Poincaré section for clarity. \label{fig:SSM_minipage_2}}
\end{figure}

\subsection{The optimal dimension of the slow SSM}

The integer $q$ in the the choice of the slow spectral subspace $E_{1,\ldots,q}$
is a free parameter. This integer is best selected in a way so that
the resulting slow SSM, $W_{1,\ldots,q}(x_{\epsilon})$, is the most
prevalent low-dimensional attractor containing the underlying NNM
$x_{\epsilon}(t)$ described in Theorems \ref{theo:unforced NNM}
and \ref{theo:forced NNM}. 

Generally, can can construct a nested hierarchy of such prevalent
slow manifolds. At any step in this hierarchy, the remaining slow
spectrum can further be divided along the next largest gap in the
real part of the eigenvalues $\lambda_{j}$ of the linearized system
\eqref{eq:linearization}. Dividing the spectrum along this spectral
gap provides the most readily observable decay rate separation for
the trajectories inside of, and towards, the slow SSM. Defining the
index sequence $q_{j}$ as 
\begin{eqnarray*}
q_{1} & = & \arg\max_{j\in[1,N-1]}\left|\mathrm{Re}\lambda_{j+1}-\mathrm{Re}\lambda_{j}\right|,\\
q_{2} & = & \arg\max_{j\in[1,q_{1}-1]}\left|\mathrm{Re}\lambda_{j+1}-\mathrm{Re}\lambda_{j}\right|,\\
\vdots & \vdots & \vdots\\
q_{l} & = & \arg\max_{j\in[1,q_{k}-1]}\left|\mathrm{Re}\lambda_{j+1}-\mathrm{Re}\lambda_{j}\right|,\\
\vdots & \vdots & \vdots\\
q_{w} & = & 1,
\end{eqnarray*}
gives the nested sequence
\[
E_{1,\ldots,q_{1}}\supset E_{1,\ldots,q_{2}}\supset\ldots\supset E_{1,\ldots,q_{l}}\supset\ldots\supset E_{1}
\]
of $w$ spectral subspaces. If the appropriate nonresonance conditions
of Table 1 or Table 2 are satisfied for each element of this nested
sequence, than a nested sequence of $w$ slow SSMs exists, asymptotic
to an NNM of the full nonlinear system. In the autonomous case, this
nested sequence of slow SSMs is 
\begin{equation}
W_{1,\ldots,q_{1}}(0)\supset W_{1,\ldots,q_{2}}(0)\supset\ldots\supset W_{1,\ldots,q_{l}}(0)\supset\ldots\supset W_{1}(0),\label{eq:nested1}
\end{equation}
 while in the non-autonomous case, we have
\begin{equation}
W_{1,\ldots,q_{1}}(x_{\epsilon}(t))\supset W_{1,\ldots,q_{2}}(x_{\epsilon}(t))\supset\ldots\supset W_{1,\ldots,q_{l}}(x_{\epsilon}(t))\supset\ldots\supset W_{1}(x_{\epsilon}(t)).\label{eq:nested2}
\end{equation}
 In the autonomous case, therefore, the minimal slow SSM is $W_{1}(0)$,
tangent to the slowest eigenspace $E_{1}$ at $x=0$ with $\dim W_{1}(0)=\dim E_{1}$.
In the non-autonomous case, the minimal slow SSM is $W_{1}(x_{\epsilon}(t))$
which is $\mathcal{O}(\epsilon)$ $C^{r}$-close to\emph{ $\left\{ x_{\epsilon}\right\} \times E_{1}$
}in the $x$ variable.

Reducing the full dynamical system \eqref{eq:1storder_system} to
the minimal slow SSM brings the largest reduction in the number of
dimensions: the dimension of the reduced model obtained in this fashion
is equal to the algebraic multiplicity of the eigenvalue $\lambda_{1}$
that lies closest to zero. If this eigenvalue is simple and complex,
then the dimension of the reduced system on the slowest SSM is two.
If the eigenvalue is simple and real, than this reduced dimension
is one. 

Reducing the dynamic to the minimal (slowest) SSM, however, only captures
the correct system dynamics over very long time scales in case the
spectral gap between $\mathrm{Re}\lambda_{1}$ and $\mathrm{Re}\lambda_{2}$
is small. This is because in that case, solution components decaying
transverse the slowest SSM may take a long time to die out. More generally,
the optimal choice of the SSM in the nested sequences \eqref{eq:nested1}-\eqref{eq:nested2}
depends on the time scale over which the approximation of the reduced
flow on the SSM is to be used as a model for the behavior of the full
system. In the absence of a definitive target time scale, a reasonable
choice is $W_{1,\ldots,q_{1}}(0)$ or $W_{1,\ldots,q_{1}}(x_{\epsilon}(t))$,
i.e., the slow SSM corresponding to the largest gap in the real part
of the spectrum of $A$.

\subsection{Implications for the computation of NNMs and slow SSMs}

Theorems \ref{theo:SSM unforced} and \ref{theo:SSM forced} provide
a mathematical foundation for a systematic computation of slow SSMs.
Without going into technical details, we briefly mention the main
computational implications that follow from the application of these
theorems.

\subsubsection{Local Taylor--Fourier expansion for slow SSMs\label{sub:Local-Taylor--Fourier-expansion}}

In our terminology, all slow SSMs are unique and anchored to a unique
NNM, which may be trivial (a fixed point), periodic (a closed orbit)
or quasiperiodic (an invariant torus). The most common nonlinearities
used in mechanical modeling are analytic functions, i.e., have everywhere
convergent Taylor-series expansion in terms of the $x$ and $\epsilon$
variables. To this end, we will assume here that the right-hand side
of the dynamical system \eqref{eq:1storder_system} is analytic near
the origin in all its arguments, i.e., 
\[
f_{0},f_{1}\in C^{a}.
\]
Theorems \ref{theo:SSM unforced} and \ref{theo:SSM forced} then
guarantee that under appropriate low-order nonresonance conditions,
the slow SSMs of the system also admit convergent Taylor expansions
about the NNMs they are anchored to.

Consider a spectral subspace $E_{1,\ldots q}$ with $u:=\dim E_{1,\dots,q}$,
satisfying the nonresonance conditions of Table 1. After a linear
change of coordinates, the variable $x$ can be split as
\[
x=(y,z)\in E_{1,\ldots q}\times E_{q+1,\ldots,N}.
\]
 In these coordinates, system \eqref{eq:1storder_system} takes the
form
\begin{eqnarray}
\dot{y} & = & A_{y}y+f_{0y}(y,z)+\epsilon f_{1y}(y,z,\Omega t;\epsilon),\nonumber \\
\dot{z} & = & A_{z}z+f_{0z}(y,z)+\epsilon f_{1z}(y,z,\Omega t;\epsilon),\label{eq:1st order in modal coordinates}
\end{eqnarray}
with the constant matrices 
\[
A_{y}\in\mathbb{R}^{u\times u},\quad A_{z}\in\mathbb{R}^{(N-u)\times(N-u)},
\]
and with appropriate $C^{r}$ functions $f_{0y},f_{0z},f_{1y}$ and
$f_{1z}.$

In the autonomous case, the unique slow SSM $W_{1,\ldots,q}(0)$ can
then locally be written in the form of a convergent Taylor series
\[
z=h^{0}(y)=\sum_{\left|p\right|=1}^{\infty}h_{p}^{0}y^{p},\quad p=\left(p_{1},\ldots,p_{u}\right),\quad y^{p}:=\left(y_{1}^{p_{1}},\ldots,y_{u}^{p_{u}}\right),\quad h_{p}^{0}\in\mathbb{R}^{N-u}.
\]
By Theorem \ref{theo:SSM unforced}, this expansion can be truncated
at an order
\[
\sigma(E_{1,\ldots,q})+1=\mathrm{Int}\left[\mathrm{Re}\lambda_{N}/\mathrm{Re}\lambda_{1}\right]+1,
\]
as an approximation to the unique slow SSM $W_{1,\ldots,q}(0)$. Lower-order
truncations of $h^{0}(y)$ also approximate a multitude of other invariant
manifolds tangent to $E_{1,\ldots q}$.

In the non-autonomous case, the slow SSM, $W_{1,\ldots,q}(x_{\epsilon},\Omega t)$,
can locally be written in the form of a convergent Fourier--Taylor
series
\begin{eqnarray}
z & = & h^{\epsilon}(y,t)=h^{0}(y)+\epsilon h^{1}(y,\Omega_{1}t,\ldots,\Omega_{k}t;\epsilon)=\sum_{\left|p\right|=1}^{\infty}h_{p}(t)y^{p}\label{eq:ttaylor}\\
 & = & \sum_{\left|p\right|=1}^{\infty}h_{p}^{0}y^{p}+\sum_{l=1}^{\infty}\sum_{\left|p\right|=1}^{\infty}\sum_{\left|m\right|=1}^{\infty}\epsilon^{l}y^{p}\left[C_{lmp}\sin(\left\langle m,\Omega\right\rangle t)+D_{lmp}\cos(\left\langle m,\Omega\right\rangle t)\right],\nonumber 
\end{eqnarray}
with $C_{lmp},D_{lmp}\in\mathbb{R}^{N-u}$. Again, by Theorem \ref{theo:SSM forced},
the convergent power series $h^{0}$ and $h^{1}$ of $y$ can be truncated
at an order
\[
\Sigma(E_{1,\ldots,q})+1=\mathrm{Int}\left[\mathrm{Re}\lambda_{N}/\mathrm{Re}\lambda_{1}\right]+1,
\]
serving as an approximation to the unique slow SSM, $W_{1,\ldots,q}(x_{\epsilon}(t))$.
Lower-order truncations of the series will also approximate an infinity
of other invariant manifolds with similar properties.

For an illustration of these computations in a simple setting, we
refer the reader to Example \ref{ex:shaw-pierre2}. In that example,
the Taylor expansion was carried out up to sixth order, and the Fourier
expansion in formula \eqref{eq:ttaylor} was replaced by the direct
numerical solution of the boundary value problems defining the time-periodic
Taylor coefficients $h_{p}(t)$.

\subsubsection{Local PDEs for slow SSMs\label{sub:Local-PDEs-for}}

Once the existence and uniqueness of the slow SSMs in the appropriate
function class is clarified from Theorems \ref{theo:SSM unforced}
and \ref{theo:SSM forced}, we may also write down a PDE for these
manifolds using their invariance properties. As mentioned in the Introduction
(see also Appendix \ref{sub:Existence-and-uniqueness issues for numerically solved PDEs}),
such PDEs are solved in the literature without specific concern for
the uniqueness of their solution under ill-posed or undetermined boundary
conditions. 

The relevant lesson from Theorems \ref{theo:SSM unforced} and \ref{theo:SSM forced}
is that approximate numerical solutions of these PDE in any set of
basis functions should be constructed in a way that the infinitely
many less smooth invariant manifolds are excluded from consideration.
For instance, in the autonomous case covered by Theorems \ref{theo:SSM unforced},
cost functions penalizing the magnitude of numerically computed derivatives
of order $\sigma(E_{1,\ldots,q})+1$, or $\Sigma(E_{1,\ldots,q})+1$,
respectively, could be employed for a defendable approximation to
the SSM.

\subsubsection{Global parametrization of slow SSMs}

Classic invariant manifold techniques (see, e.g., Fenichel \cite{fenichel71})
construct the invariant surfaces in question as graphs over an appropriate
set of variables. In our present context, this translates to seeking
an SSM as a graph of the form $z=h^{0}(y)$ or $z=h^{\epsilon}(y,t)$,
as assumed in the Taylor--Fourier- and PDE-based approaches discussed
above. Both of these approaches are local in nature, capturing only
a subset of the SSM that can be viewed as a graph over the underlying
$E_{1,\ldots,q}$ spectral subspace. The construction of the SSM,
therefore, breaks down once the SSM develops a fold over $E_{1,\ldots,q}$,
i.e., becomes a multi-valued graph over $E_{1,\ldots,q}$ (cf. Fig.
\ref{fig:graph_vs_embedding})

\begin{figure}[H]
\centering{}\includegraphics[width=0.5\textwidth]{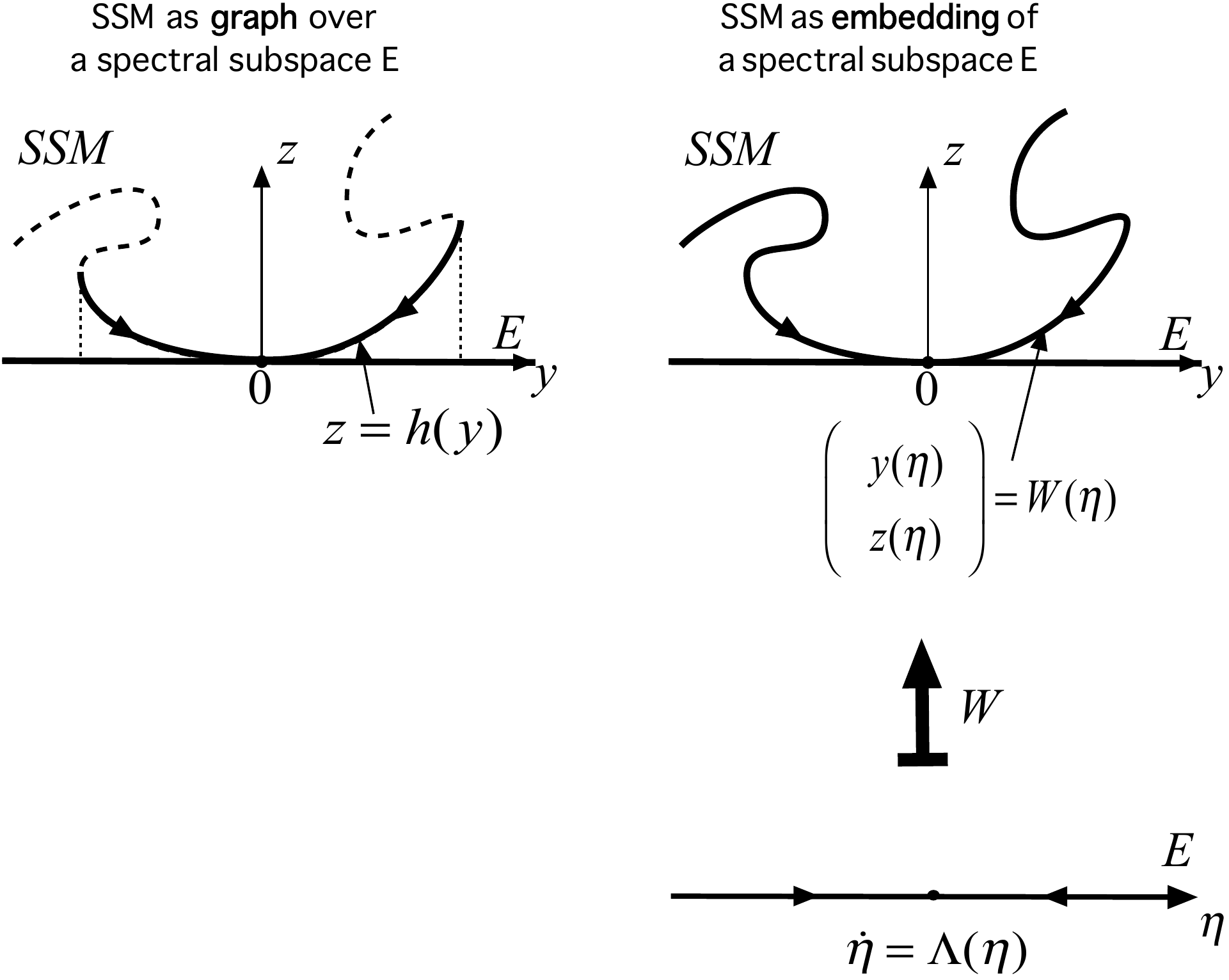}\caption{An illustration of the idea of the parametrization method for autonomous
systems (no dependence on $\phi$): Constructing an SSM as a graph
over a spectral subspace $E$ vs. as an embedding of the spectral
subspace $E$. \label{fig:graph_vs_embedding}}
\end{figure}

The proofs of the results underlying Theorems \ref{theo:SSM unforced}-\ref{theo:SSM forced},
however, do not assume such a graph property. Rather, they construct
the SSM by the parametrization method pioneered by Cabré et al. \cite{cabre03}.
This method renders the SSMs as an embedding of $E_{1,\ldots,q}$
into the phase space $\mathbb{R}^{N}$, rather than a graph over the
subspace $E_{1,\ldots,q}$ of $\mathbb{R}^{N}$. Moreover, the flow
on the SSM is exactly conjugate to a polynomial function of a parametrization
of $E_{1,\ldots,q}$. The order of this polynomial is no larger than
$K=\Sigma(E_{1,\ldots,q}$).

More specifically, with the notation $X(x,\phi)=f_{0}(x)+\epsilon f_{1}(x,\phi,\epsilon),$
our dynamical system \eqref{eq:1storder_system} and its associated
flow map $F^{t}(x,\phi)\colon\mathbb{R}^{N}\times\mathbb{T}^{k}\to\mathbb{R}^{N}$
can be written as
\[
\dot{x}=X(x,\phi),\quad\dot{\phi}=\Omega,\qquad\frac{d}{dt}F^{t}(x,\phi)=X\left(F^{t}(x,\phi),\phi+\Omega t\right),\quad F^{0}(x,\phi)=x.
\]
An SSM can then be sought as the image of $E_{1,\ldots,q}$ under
an embedding
\begin{eqnarray*}
W:E_{1,\ldots,q}\times\mathbb{T}^{k} & \to & \mathbb{R}^{N},\\
(\eta,\phi) & \mapsto & x,
\end{eqnarray*}
such that the reduced model flow on $E_{1,\ldots,q}$ has the associated
differential equation and flow map
\begin{equation}
\dot{\eta}=\Lambda(\eta,\phi),\quad\dot{\phi}=\Omega,\qquad\frac{d}{dt}G^{t}(\eta,\phi)=\Lambda(G^{t}(\eta,\phi),\phi+\Omega t),\quad G^{0}(\eta,\phi)=\eta.\label{eq:reducedeq}
\end{equation}
 Our model flow is defined over all of the spectral subspace $E_{1,\ldots,q}$.
We may seek this model flow map in the form of a Fourier--Taylor expansion
\[
G^{t}(\eta,\phi)=\sum_{\left|j\right|=1}^{K}g_{j}(\phi,t)\eta^{j},
\]
which, substituted into \eqref{eq:reducedeq}, gives
\[
\Lambda\left(\sum_{\left|j\right|=1}^{K}g_{j}(\phi,t)\eta^{j},\phi+\Omega t\right)=\sum_{\left|j\right|=1}^{K}\left[D_{\phi}g_{j}(\phi,t)\Omega+D_{t}g_{j}(\phi,t)\right]\eta^{j}+g_{j}(\phi,t)\Lambda(\eta,\phi).
\]
 The invariance of the SSM can then be expressed by the equation 
\[
F^{t}\left(W(\eta,\phi),\phi\right)=W\left(G^{t}(\eta,\phi),\phi+\Omega t\right).
\]
Differentiating this equation in time and setting $t=0$ yields the
infinitesimal invariance condition
\begin{equation}
X\left(W(\eta,\phi),\phi\right)=D_{\eta}W\left(\eta,\phi\right)\Lambda(\eta,\phi)+D_{\phi}W(\eta,\phi)\Omega.\label{eq:invcond}
\end{equation}
Substituting the analytic Taylor--Fourier expansions
\begin{eqnarray*}
W\left(\eta,\phi\right) & = & \sum_{l=1}^{\infty}\sum_{\left|p\right|=1}^{\infty}\sum_{\left|m\right|=1}^{\infty}\epsilon^{l}\eta^{p}\left[E_{lmp}\sin(\left\langle m,\Omega\right\rangle t)+F_{lmp}\cos(\left\langle m,\Omega\right\rangle t)\right],\\
\Lambda(\eta,\phi) & = & \sum_{l=1}^{\infty}\sum_{\left|p\right|=1}^{\infty}\sum_{\left|m\right|=1}^{\infty}\epsilon^{l}\eta^{p}\left[G_{lmp}\sin(\left\langle m,\Omega\right\rangle t)+H_{lmp}\cos(\left\langle m,\Omega\right\rangle t)\right],
\end{eqnarray*}
into the invariance condition \eqref{eq:invcond}, one can recursively
solve for the coefficients of the embedding $W\left(\eta,\phi\right)$
of the SSMs together with the coefficients of the right-hand side
$\Lambda(\eta,\phi)$ of the differential equation \eqref{eq:reducedeq},
describing the reduced-order dynamics on the slow SSM.

Practical hints on the numerical implementation of the above parametrization
method are described by Haro et al. \cite{haro16} and Mireles--James
\cite{mireles--james15}. As mentioned in the Introduction, Cirillo
et al. \cite{cirillo15a} have recently suggested a computational
technique for a two-dimensional autonomous SSM that is identical to
the parametrization method in their setting.

\section{Conclusions}

We have proposed a unified terminology in the nonlinear modal analysis
of dissipative systems, deriving rigorous existence, uniqueness, smoothness
and robustness results for the nonlinear normal modes (NNMs) and their
spectral submanifolds (SSMs) covered by this terminology. 

The NNMs defined here generalize the original nonlinear normal mode
concept of Rosenberg to dissipative yet eternally recurrent motions
with finitely many frequencies, including fixed points, periodic motions
and quasiperiodic motions. In contrast, the SSMs introduced here are
the smoothest invariant manifolds asymptotic to such generalized NNMs
along their spectral subbundles. As such, SSMs build on the Shaw--Pierre
normal mode concept and clarify its relationship to Rosenberg's concept
in a general dissipative, multi-degree-of-freedom system, possibly
subject to time-periodic or quasipriodic forcing.

In our setting, NNMs are locally unique in the phase space, admitting
a unique SSM over any of their spectral subspaces (or subbundles)
that have no low order resonances with the remaining part of the linearized
spectrum. In the autonomous case, the order of these nonresonance
conditions is fully governed by the relative spectral quotient $\sigma(E)$
of the spectral subspace of interest. In the non-autonomous case,
the role of $\sigma(E)$ is taken over by the absolute spectral quotient
$\Sigma(E)$. Both of these spectral quotients can be a priori determined
from the spectrum of the linearized system (see Tables 1 and 2). 

Our results cover three classes of SSMs: fast, intermediate and slow.
Out of these classes, fast SSMs have unrestricted uniqueness among
all differentiable invariant surfaces in the autonomous case, but
are generally the least relevant for model reduction. In contrast,
slow SSMs are the most relevant for model reduction, but have the
most restricted uniqueness properties. Namely, the minimal order of
a Taylor expansion distinguishing any slow SSM from other invariant
manifolds is the smallest integer that is larger than the ratio of
the strongest and the weakest decay rate of the linearized system.
This spectral ratio may well be large even for weakly damped systems,
thus a careful consideration of damping is essential for rigorous
SSM-based model reduction approaches. 

Our results are meant to aid the construction of formal expansions
and intuitive computations of NNMs and SSMs. As we discussed, most
of these operational approaches tend to hide the fundamental non-uniqueness
of invariant manifolds tangent to modal subspaces. The ambiguity in
the results is inherently small close to the underlying fixed point
but is magnified significantly away from fixed points (see, e.g.,
Fig. \ref{fig:resonant}a), and becomes an obstacle to extending invariant
manifolds in a defendable fashion to larger domains of the phase space.
The use of SSMs eliminate this ambiguity, and should therefore be
useful in expanding the range of nonlinear modal analysis in a well-understood
fashion.\\

\begin{description}
\item [{Acknowledgments}]~
\end{description}
We are grateful to Rafael de la Llave and Alex Haro for detailed technical
explanations on their invariant manifold results, to Ludovic Renson
for clarifying the numerical approach in Ref. \cite{renson14}, and
to Paolo Tiso for helpful discussions on nonlinear normal modes. We
are also thankful to Alireza Hadjighasem for his advice on visualization,
and to Robert Szalai for pointing out typographical errors in an earlier
version of this manuscript. Finally, we are pleased to acknowledge
useful suggestions from the two anonymous reviewers of this work.

\appendix

\section{Existence, uniqueness and analyticity issues for invariant manifolds
tangent to eigenspaces}

\subsection{Modified Euler example of a non-analytic but $C^{\infty}$ center
manifold\label{sub:Euler's-example-of}}

For the system \eqref{eq:Euler's example}, the origin is a fixed
point with eigenvalues $\lambda_{1}=0$ and $\lambda_{2}=-1$ and
corresponding eigenvectors $e_{1}=(1,1)$ and $e_{2}=(0,1)$. Therefore,
the classic center manifold theorem (see, e.g., Guckenheimer and Holmes
\cite{guckenheimer83}) guarantees the existence of a center manifold
$W^{c}(0)$, tangent to the $x$ axis at the origin. We seek $W^{c}(0)$
in the form of a Taylor expansion
\[
y=h(x)=x+\sum_{j=2}^{\infty}a_{j}x^{j},
\]
which we differentiate in time to obtain
\begin{equation}
\dot{y}=\left(1+\sum_{j=2}^{\infty}ja_{j}x^{j-1}\right)\dot{x}=-\left(1+\sum_{j=2}^{\infty}ja_{j}x^{j-1}\right)x^{2}=-x^{2}-\sum_{j=2}^{\infty}ja_{j}x^{j+1}=-\sum_{j=2}^{\infty}(j-1)a_{j-1}x^{j},\label{eq:ydot1}
\end{equation}
where we have let $a_{1}=1.$ At the same, we evaluate the second
equation in \eqref{eq:Euler's example} on the manifold $W^{c}(0)$
to obtain 
\begin{equation}
\dot{y}=-h(x)+x=-\sum_{j=2}^{\infty}a_{j}x^{j}.\label{eq:ydot2}
\end{equation}
Equating \eqref{eq:ydot1} and \eqref{eq:ydot2} gives the recursion
$a_{j}=(j-1)a_{j-1}$ with $a_{1}=1,$ which implies $a_{j}=(j-1)!.$
We therefore obtain the explicit form 
\begin{equation}
h(x)=\sum_{j=1}^{\infty}(j-1)!x^{j}\label{eq:eulermanifold}
\end{equation}
as a formal expansion of the center manifold, as stated in the Introduction.
The formal series $h(x)=\sum_{j=1}^{\infty}(j-1)!x^{j}$, however,
diverges for any $x\neq0$, thus the center manifold is $C^{\infty}$
but not analytic in any open neighborhood of the origin.

\subsection{Uniqueness and analyticity issues for invariant manifolds in linear
systems\label{sub:Uniqueness-and-analyticity}}

Any invariant manifold through the origin of the linearized system
\eqref{eq:diagonalized linearization} is locally a graph over $q$
of the elements of the vector $y$. Such a graph is of the general
form
\begin{equation}
y_{l}=f_{l}(y_{j_{1}},\ldots,y_{j_{q}}),\quad l\notin\left\{ j_{1},\ldots,j_{q}\right\} .\label{eq:general invariant manifold in linearized system}
\end{equation}
By the invariance of these surfaces, one can substitute full trajectories
into \eqref{eq:general invariant manifold in linearized system} and
differentiate in time to obtain the PDE 
\begin{equation}
\lambda_{l}f_{l}=\sum_{i=1}^{q}\lambda_{j_{i}}y_{j_{i}}\partial_{y_{j_{i}}}f_{l},\qquad l\notin\left\{ j_{1},\ldots,j_{q}\right\} .\label{eq:linearPDE}
\end{equation}
This linear PDE can be solved locally by the method of characteristics
(see, e.g., Evans \cite{evans98}), once we prescribe the value of
$f_{l}$ along an appropriate codimension-one set $\Gamma(s_{1},\ldots,s_{q-1})$
of the spectral subspace $E_{j_{1},\ldots,j_{q}}$. Here the real
variables $s=(s_{1},\ldots,s_{q-1})$ parametrize the surface $\Gamma$.
For instance, $\Gamma$ can be selected as a $q-1$ dimensional sphere
in $E_{j_{1},\ldots,j_{q}}$ that surrounds the origin. 

Fixing a boundary condition 
\begin{equation}
f_{l}(\Gamma(s_{1},\ldots,s_{q-1}))=f_{l}^{0}(s_{1},\ldots,s_{q-1})\label{eq:linear BC}
\end{equation}
 gives the equation for characteristics:
\begin{equation}
y_{j_{i}}(t)=\Gamma_{i}(s_{1},\ldots,s_{q-1})e^{\lambda_{j_{i}}t},\qquad i=1,\ldots,q.\label{eq:chareq1}
\end{equation}
 
\begin{equation}
f_{l}(y_{j_{1}}(t),\ldots,y_{j_{q}}(t))=f_{l}^{0}(s_{1},\ldots,s_{q-1})e^{\lambda_{p}t}.\label{eq:chareq2}
\end{equation}

Then, the strategy to obtain a solution for the PDE \eqref{eq:linearPDE}
is the following: express the variables $(s_{1},\ldots,s_{q-1},t)$
as a function of $(y_{j_{1}},\ldots,y_{j_{q}})=(y_{j_{1}}(t),\ldots,y_{j_{q}}(t))$
from the $q$ algebraic equations \eqref{eq:chareq1} in the vicinity
of $\Gamma$, and substitute the result into \eqref{eq:chareq2} to
obtain a solution $f_{l}(y_{j_{1}},\ldots,y_{j_{q}})$ to \eqref{eq:linearPDE}
that satisfies the boundary condition \eqref{eq:linear BC}.

To this end, we rewrite \eqref{eq:chareq1} as 
\begin{equation}
\Gamma_{i}(s_{1},\ldots,s_{q-1})e^{\lambda_{j_{i}}t}-y_{j_{i}}=0,\qquad i=1,\ldots,q,\label{eq:chareq1-1}
\end{equation}
and observe that this system of $q$ algebraic equations is solved
by $t=0$ and $y_{j_{i}}^{0}=y_{j_{i}}(0)=\Gamma_{i}(s_{1},\ldots,s_{q-1})$.
By the implicit function theorem, the variables $(s_{1},\ldots,s_{q-1},t)$
can be expressed from \eqref{eq:chareq1-1} near $\Gamma$ as a function
of $y_{j_{i}}$ if the Jacobian 
\begin{equation}
D_{s_{1},\ldots,s_{q-1},t}\left[\begin{array}{c}
\Gamma_{1}(s_{1},\ldots,s_{q-1})e^{\lambda_{j_{1}}t}-y_{j_{1}}\\
\vdots\\
\Gamma_{q}(s_{1},\ldots,s_{q-1})e^{\lambda_{j_{q}}t}-y_{j_{q}}
\end{array}\right]_{(y_{j_{i}}=y_{j_{i}}^{0},t=0)}=\left[D_{s}\Gamma,\,\,-\Lambda y\vert_{E_{j_{1},\ldots,j_{q}}}\right],\label{eq:impl fun theorem}
\end{equation}
is non-degenerate. In other words, along the surface $\Gamma$, all
tangent vectors of $\Gamma$ should be linearly independent of the
vector field $\Lambda y$ restricted to its invariant subspace $E_{j_{1},\ldots,j_{q}}$.
In the language of linear PDEs, the boundary surface $\Gamma$ should
be a \emph{non-characteristic surface} for a unique, local solution
to exist near $\Gamma$ for any boundary condition posed over $\Gamma$.
This argument just reproduces the classic local existence and uniqueness
result for linear first-order PDEs (see, e.g., Evans \cite{evans98}). 

Under these conditions, therefore, we have a unique, local solution
for any initial function $f_{l}^{0}(s_{1},\ldots,s_{q-1})$ defined
on $\Gamma$. There are infinitely many different choices both for
the surface $\Gamma$ and the boundary values $f_{l}^{0}$. Since
the Jacobian \eqref{eq:impl fun theorem} is non-degenerate for any
$y\neq0$, each of these infinitely many choices leads to a local
invariant surface satisfying \eqref{eq:linearPDE} in the vicinity
of $\Gamma$, which in turn can be propagated all the day to the $y=0$
fixed point along characteristics of the PDE. Accordingly, we obtain
\emph{infinitely many} invariant surfaces tangent to the spectral
subspace $E_{j_{1},\ldots,j_{q}}$ in the linearized system \eqref{eq:diagonalized linearization}.
Applying the more general Theorem \ref{theo:SSM unforced} in the
current linear setting, however, we obtain that only one analytic
solution exists to the PDE \eqref{eq:linearization} for any fixed
subspace $E_{j_{1},\ldots,j_{q}}$ under the nonresonance conditions
detailed in Theorem \ref{theo:SSM unforced}. Since $f_{l}(y_{j_{1}},\ldots,y_{j_{q}})\equiv0$
is analytic, this flat solution must be the unique analytic solution
of \eqref{eq:linearPDE}. All other solutions are only finitely many
times differentiable, and hence are not even $C^{\infty}$.

\subsection{Uniqueness issues for invariant manifolds obtained from numerical
solutions of PDEs\label{sub:Existence-and-uniqueness issues for numerically solved PDEs}}

The PDE approach we described in Section \ref{sub:Uniqueness-and-analyticity}
is broadly used in the literature to compute Shaw--Pierre type invariant
surfaces for nonlinear systems. This approach was originally suggested
by Shaw and Pierre \cite{shaw93}, explored first in detail first
by Peschek et al. \cite{pesheck02}, then developed and applied further
by various authors (see Renson et al. \cite{renson16} for a recent
review). Interestingly, none of these studies reports or discusses
non-uniqueness of solutions, which appears to be in contradiction
with our conclusions in Section \ref{sub:Uniqueness-and-analyticity}.
Here we take a closer look to understand the reason behind this paradox.

In the simplified setting of Section \ref{sub:Uniqueness-and-analyticity},
one may seek invariant manifolds of the form $y_{l}=f_{l}(y_{j_{1}},\ldots,y_{j_{q}}),\quad l\notin\left\{ j_{1},\ldots,j_{q}\right\} $
in a nonlinear system
\begin{equation}
\dot{y}=\Lambda y+g(y),\qquad\Lambda=\mathrm{diag}\left(\lambda_{1},\ldots,\lambda_{N}\right),\quad g(y)=\mathbb{\mathcal{O}}\left(\left|y\right|^{2}\right),\label{eq:nonlin diagonal}
\end{equation}
over a spectral subspace $E_{j_{1},\ldots,j_{q}}$ of the operator
$A$. The same argument we used in the linear case now leads to a
quasilinear version of the linear system of PDEs \eqref{eq:linearPDE}.
This quasilinear system of PDEs is of the form 
\begin{equation}
\lambda_{l}f_{l}+g_{l}(y_{j},f)=\sum_{i=1}^{q}\left[\lambda_{j_{i}}y_{j_{i}}+g_{j_{i}}(y_{j},f)\right]\partial_{y_{j_{i}}}f_{l},\qquad l\notin\left\{ j_{1},\ldots,j_{q}\right\} ,\label{eq:quasilinearPDE}
\end{equation}
 with $y_{j}=(y_{j_{1}},\ldots,y_{j_{q}})$ and $f$ denoting the
vector of the$f_{l}$ functions. 

The local existence and uniqueness theory relevant for this PDE is
identical to that for its linear counterpart (cf. Evans \cite{evans98}).
Specifically, as in Section \ref{sub:Uniqueness-and-analyticity},
boundary conditions 
\begin{equation}
f_{l}(\Gamma(s_{1},\ldots,s_{q-1}))=f_{l}^{0}(s_{1},\ldots,s_{q-1}),\label{eq:quasilinear BC}
\end{equation}
must be posed on a non-characteristic, codimension-one boundary surface
$\Gamma$ inside the subspace $E_{j_{1},\ldots,j_{q}}$ for the PDE
\eqref{eq:quasilinearPDE} to have a unique local solution near $\Gamma$.
Here the required non-characteristic property of $\Gamma$ is that
the projected vector field $\dot{y}_{j}=\left[\Lambda y+g(y)\right]_{j}$
over $E_{j_{1},\ldots,j_{q}}$ should be transverse to $\Gamma$ at
all points. Since this boundary condition is arbitrary, one again
obtains infinitely many local Shaw--Pierre type invariant manifolds
near the boundary surface $\Gamma$ for the nonlinear problem \eqref{eq:nonlin diagonal}:
one for any boundary condition posed over any non-characteristic surface
$\Gamma.$ In the general case, all of these are also global solutions
that extend smoothly to the origin and give a smooth solution to the
PDE \eqref{eq:quasilinearPDE} in a whole neighborhood of the fixed
point. The only exception is when the invariant manifold is sought
as a graph over the $q$ fastest modes. In this case, the strong stable
manifold theorem (Hirsch, Pugh and Shub \cite{hirsch77}) guarantees
the existence of a unique invariant manifold. In this case, while
infinitely many local solutions still exist near a non-characteristic
boundary surface $\Gamma$, these local solutions do not extend smoothly
to the origin.

Surprisingly, all available numerical algorithms aiming to solve \eqref{eq:quasilinearPDE}
in the nonlinear normal modes literature ignore this non-uniqueness
issue. They are typically validated or illustrated on the computation
of two-dimensional invariant manifolds tangent to the single, slowest
decaying spectral subspace ($q=1,$ $\dim E_{1}=2$). Already in this
simplest case, the high-degree of non-uniqueness illustrated in Fig.
\ref{fig:nonuniqueness} definitely applies. This raises the question:
How do these studies obtain a unique invariant manifold? There are
different reasons for each numerical algorithm, as we review next.

Peschek et al. \cite{pesheck02} consider a spectral subspace $E_{1}$
corresponding to a simple, complex conjugate pair of eigenvalues.
They pass to amplitude-phase variables $(a,\varphi)$ by letting $y_{j_{1}}=ae^{i\varphi}$,
and reconsider the quasilinear PDE \eqref{eq:quasilinearPDE} posed
for the unknown functions $f_{l}(a,\varphi)$. As domain boundary
$\Gamma$, they then consider the $a=0$ axis, over which they prescribe
$f_{l}(0,\varphi)=0$ and $\partial_{a}f_{l}(0,\varphi)=0.$ This
is consistent with the fact that the origin $y_{j_{1}}=0$ is mapped,
due to the singularity of the polar coordinate change, to the $a=0$
of the $(a,\phi)$ coordinate space, and hence the surface should
have a quadratic tangency with this line. However,  the $a=0$ line
is invariant under the transformed nonlinear vector field $(\dot{a},\dot{\varphi)}$,
given that it is the image of the fixed point of the original nonlinear
system, which satisfies $\dot{a}=0$. As a consequence, $\Gamma$
is a characteristic surface, and hence local existence and uniqueness
is not guaranteed for the quasilinear PDE \eqref{eq:quasilinearPDE}
with this boundary condition. As we discussed above, the PDE is in
fact known to have infinitely many solutions, all of which have a
quadratic tangency with the origin, and hence satisfy the singular
boundary conditions $f_{l}(0,\phi)=0$ and $\partial_{a}f_{l}(0,\phi)=0$
in polar coordinates. Therefore, the problem considered by Peschek
et al. \cite{pesheck02} only has a unique solution for invariant
manifolds over the fast modes, but not over the slow or intermediate
modes. The same holds true for all other studies utilizing the approach
developed by Peschek et al. \cite{pesheck02}.

Renson et al. \cite{renson14} solve the same quasilinear PDE \eqref{eq:quasilinearPDE}
in the setting of Peschek et al. \cite{pesheck02} (autonomous system
with $q=1$ and with $\dim E_{1}=2$). In the conservative case, they
seek to construct solutions using a closed boundary curve $\Gamma$
to which the nonlinear vector field $\dot{y}_{j}$ is tangent at each
point. For damped systems, they solve the PDE outward from the equilibrium,
first over an elliptic domain, then gradually outwards over a nested
sequence of annuli. The boundaries of all these domains are selected
as non-characteristic curves, thus a unique solution can be constructed
over each domain in the nested sequence. Over the initial (elliptic)
domain boundary, however, the spectral subspace itself is chosen as
initial condition ($f_{l}^{0}(\Gamma)=0$ for all $l>2$), which singles
out one special solution out of the arbitrarily many. The perceived
uniqueness is, therefore, the artifact of the numerical procedure.

Finally, Blanc et al. \cite{blanc13} start out by correctly selecting
a non-characteristic boundary curve $\Gamma$ in the amplitude--phase--coordinate
setting of Peschek et al. \cite{pesheck02} discussed above. This
curve is just the $\varphi=0$ line of the $(a,\varphi)$ coordinate
plane, to which the characteristics of the PDE are transverse in a
neighborhood of the origin, as required for the local existence and
uniqueness of solutions near $\Gamma$. In this case, any initial
profile $f_{l}(a,0)=f_{l}^{0}(a)$ with $f_{l}^{0}(0)=0$ and $f_{l}^{0\prime}(0)=0$
would lead to a Shaw--Pierre type invariant manifold, thereby revealing
the inherent non-uniqueness of this numerical approach. Instead of
realizing this, Blanc et al. \cite{blanc13} assert that there is
a single correct boundary condition that they need to find by an optimization
process. 

In this optimization process, Blanc et al. \cite{blanc13} modify
the initial boundary condition iteratively so that the computed PDE
solution along the line $\varphi=2\pi$, given by $f_{l}(a,2\pi)$,
is as close to $f_{l}(a,0)=f_{l}^{0}(a)$ as possible in the $L^{2}$
norm. Should they enforce the exact periodicity of the solution of
the PDE on the periodic domain $(a,\varphi)\in[0,a_{max}]\times[0,2\pi]$
(say, by a spectral method), they would always have $f_{l}(a,\varphi)\equiv f_{l}(a,0)$
on any solution, so minimizing the error in this identity would lead
to a vacuous process. In other words, the seemingly unique solution
in this approach is the surface along which the error arising from
an inaccurate handling of the periodic boundary conditions is minimal
in a particular norm.

\section{Existence, uniqueness and persistence of non-autonomous NNMs}

We rewrite system \eqref{eq:1storder_system} in the form of a $(N+k)$-dimensional
autonomous system 
\begin{eqnarray}
\dot{x} & = & Ax+f_{0}(x)+\epsilon f_{1}(x,\phi;\epsilon),\label{eq:extended_full_system}\\
\dot{\phi} & = & \Omega,\nonumber 
\end{eqnarray}
defined on the phase space $\mathcal{P}=\mathcal{U}\times\mathbb{T}^{k}$.
For $\epsilon=0$, the trivial normal mode $x=0$ now appears as an
invariant, $k$-dimensional torus
\[
T_{0}=\left\{ (x,\phi)\in\mathcal{P}\,:\,x=0,\,\,\phi\in\mathbb{T}^{k}\right\} 
\]
for system \eqref{eq:extended_full_system}.

Assume that all eigenvalues of $A$ satisfy the condition $\mathrm{Re}\lambda_{i}\neq0.$
This means that all possible exponential contraction and expansion
rates transverse to $T_{0}$ dominate (the zero) expansion and contraction
rates in directions tangent to $T_{0}$, along the $\phi$ coordinates.
In the language of the theory of normally hyperbolic invariant manifolds,
the torus $T_{0}$ is a compact, $r$-normally hyperbolic invariant
manifold for any integer $r\geq1$ (Fenichel \cite{fenichel71}).

Fenichel's general result on invariant manifolds do not allow, however,
to conclude the persistence of $C^{0}$, $C^{\infty}$ or $C^{a}$
normally hyperbolic invariant manifolds. Instead, such persistence
is established by Haro and de la Llave \cite{haro06}, who specifically
study persistence of invariant tori in systems of the form of \eqref{eq:extended_full_system}.

\section{Existence, uniqueness and persistence for autonomous SSMs ($k=0$)}

First, we recall a more abstract results of Cabré, Fontich and de
la Llave \cite{cabre03} on mappings in Banach spaces, which we subsequently
apply to our setting.

\subsection{Spectral submanifolds for mappings on complex Banach spaces}

We denote by $\mathcal{P}$ a real or complex Banach space, and by
$\mathcal{U}\subset\mathcal{P}$ an open set. We let $C^{r}(\mathcal{U},Y)$
denote the set of functions $f:U\to Y$ that have continuos and bounded
derivatives up to order $r$ in $\mathcal{U}$. Let the space $C^{\infty}(\mathcal{U},Y)$
denote the set of those functions $f$ that are in the class $C^{r}(\mathcal{U},Y)$
for every $r\in\mathbb{N}$, and let $C^{a}(\mathcal{U},Y)$ denote
the set of functions $f$ that are bounded and analytic in $U$. 

Let $0\in\mathcal{U}$ be a fixed point for a $C^{r}$ map $\mathcal{\mathcal{F}\colon}\mathcal{U}\to\mathcal{P},$
where $r\in\mathbb{N}\cup\{\infty,a\}.$ We denote the linearized
map at the fixed point by $\mathcal{A}=D\mathcal{F}(0)$ and its spectrum
by $\mathrm{spec}(\mathcal{A}).$ 

We also assume a direct sum decomposition $\mathcal{P}=\mathcal{P}_{1}\oplus\mathcal{P}_{2}$,
with the subspaces $\mathcal{P}_{1}$ and $\mathcal{P}_{2}$ to be
described shortly in terms of the spectral properties of $\mathcal{A}$.
We denote the projections from the full space $\mathcal{P}$ onto
these two subspaces by $\pi_{1}\colon\mathcal{P}\to\mathcal{P}_{1}$
and $\pi_{2}\colon\mathcal{P}\to\mathcal{P}_{2},$ and assume that
both projections are bounded. Finally, for any set $S$ and positive
integer $k$, we will use the notation 
\[
S^{k}=\underbrace{S\times\ldots\times S}_{k}
\]
for the $k$-fold direct product of $S$ with itself. 

Assume now that
\begin{description}
\item [{(0)}] $\mathcal{A}$ is invertible
\item [{(1)}] The subspace $\mathcal{P}_{1}$ is invariant under the map
$\mathcal{A}$, i.e., 
\[
\mathcal{A}\mathcal{P}_{1}\subset\mathcal{P}_{1}.
\]
As a result, we have a representation of $\mathcal{A}$ with respect
to above decomposition as
\begin{equation}
\mathcal{A}=\left(\begin{array}{cc}
\mathcal{A}_{1} & \mathcal{B}\\
0 & \mathcal{A}_{2}
\end{array}\right),\label{eq:decomp_of_A}
\end{equation}
with the operators $\mathcal{A}_{1}=\pi_{1}\mathcal{A}\vert_{\mathcal{P}_{1}},$
$\mathcal{A}_{2}=\pi_{2}\mathcal{A}\vert_{\mathcal{P}_{2}},$ and
$\mathcal{B}=\pi_{1}\mathcal{A}\vert_{\mathcal{P}_{2}}.$ If $\mathcal{P}_{2}$
is also an invariant subspace for $\mathcal{A}$, then we have $\mathcal{B}=0$.
\item [{(2)}] The spectrum of $\mathcal{A}_{1}$ lies strictly inside the
complex unit circle, i.e., $\mathrm{Spect}(\mathcal{A}_{1})\subset\left\{ z\in\mathbb{C}\,:\,\left|z\right|<1\right\} $.
\item [{(3)}] The spectrum of $\mathcal{A}_{2}$ does not contain zero,
i.e., $0\notin\mathrm{Spect}(\mathcal{A}_{2})$.
\item [{(4)}] For the smallest integer $L\geq1$ satisfying 
\begin{equation}
\left[\mathrm{Spect}(\mathcal{A}_{1})\right]^{L+1}\mathrm{Spect}(\mathcal{A}_{2}^{-1})\subset\left\{ z\in\mathbb{C}\,:\,\left|z\right|<1\right\} ,\label{eq:ratecond}
\end{equation}
 we have
\begin{equation}
\left[\mathrm{Spect}(\mathcal{A}_{1})\right]^{i}\cap\mathrm{Spect}(\mathcal{A}_{2})=\emptyset\label{eq:nonresonancecond}
\end{equation}
 for every integer $i\in[2,L]$ (in case $L\geq2)$. 
\item [{(5)}] $L+1\leq r.$ 
\end{description}
We then have the following result:
\begin{thm}
\label{CFL}{[}Theorems 1.1 and 1.2, Cabré, Fontich and de la Llave
\cite{cabre03}{]} Under assumptions (0)-(5):\end{thm}
\begin{description}
\item [{(i)}] There exists a $C^{r}$ manifold $\mathcal{M}_{1}$ that
is invariant under $\mathcal{F}$ and tangent to the subspace $\mathcal{P}_{1}$
at $0$. 
\item [{(ii)}] The invariant manifold $\mathcal{M}_{1}$ is unique among
all $C^{L+1}$ invariant manifolds of $\mathcal{F}$ that are tangent
to the subspace $\mathcal{P}_{1}$ at $0$. That is, every two $C^{L+1}$
invariant manifolds with this tangency property will coincide in a
neighborhood of $0$.
\item [{(iii)}] There exists a polynomial map $R:\mathcal{P}_{1}\to\mathcal{P}_{1}$
of degree not larger than $L$ and a $C^{r}$ map $K\colon\mathcal{U}_{1}\subset\mathcal{P}_{1}\to\mathcal{P}$,
defined over an open neighborhood $\mathcal{U}_{1}$ of $0$, satisfying
\[
R(0)=0,\quad DR(0)=\mathcal{A}_{1},\quad K(0)=0,\quad\pi_{1}DK(0)=I,\quad\pi_{2}DK(0)=0,
\]
 such that $K$ serves as an embedding of $\mathcal{M}_{1}$ from
$\mathcal{P}_{1}$ to $\mathcal{P}$, and $R$ represents the pull-back
of the dynamics on $\mathcal{M}_{1}$ to $\mathcal{U}_{1}$ under
this embedding. Specifically, we have
\[
\mathcal{F}\circ K=K\circ R.
\]

\item [{(iv)}] If, furthermore, $\left[\mathrm{Spec}(\mathcal{A}_{1})\right]^{i}\cap\mathrm{Spec}(\mathcal{A}_{1})=\emptyset$
holds for every integer $i\in[L_{-},L],$ then $R$ can be chosen
to be a polynomial of degree not larger than $L_{-}-1.$ 
\item [{(v)}] Dependence on parameters: If $\mathcal{F}$ is jointly $C^{r}$
in $x$ and a parameter $\mu$, the the invariant manifold $\mathcal{M}_{1}$
is jointly $C^{r-L-1}$ in space and the parameter $\mu$. In particular,
$C^{\infty}$ and analytic maps will have invariant manifolds that
are $C^{\infty}$ and analytic, respectively, with respect to any
parameters in the system. 
\end{description}

\subsection{\label{sub:proof SSM unforced}Proof of Theorem \ref{theo:SSM unforced}}

We now apply Theorem \ref{CFL} to system \eqref{eq:unforced}. In
this context, the space $\mathcal{P}$ is the finite-dimensional,
real vector space $\mathcal{P}=\mathbb{R}^{N},$ and the mapping is
the time-one map $\mathcal{F}=F^{1}\colon\mathcal{U\subset\mathcal{P}}\to\mathcal{P}$
of system \eqref{eq:unforced}. We further have
\begin{equation}
\mathcal{F}(0)=0,\quad\mathcal{A=}D\mathcal{F}(0)=DF^{1}(0)=e^{A},\label{eq:calAdef}
\end{equation}
 and hence $\mathcal{A}$ is invertible. We have the spectra 
\begin{equation}
\mathrm{spec}(\mathcal{A})=\left\{ e^{\lambda_{1}},e^{\bar{\lambda}_{1}},\ldots,e^{\lambda_{N}},e^{\bar{\lambda}_{N}}\right\} ,\qquad\mathrm{spec}(\mathcal{A}^{-1})=\left\{ e^{-\lambda_{1}},e^{-\bar{\lambda}_{1}},\ldots,e^{-\lambda_{N}},e^{-\bar{\lambda}_{N}}\right\} ,\label{eq:specs}
\end{equation}
 where we have ordered the eigenvalues in an increasing order based
on their real parts, i.e., 
\[
\mathrm{Re}\lambda_{N}\leq\ldots\mathrm{\leq\mathrm{Re}\lambda_{1}<0},
\]
and listed purely real elements of the spectrum of $\mathcal{A}$
and $\mathcal{A}^{-1}$ twice to simplify our notation. Equation \eqref{eq:specs}
implies that condition (0) of Theorem \ref{CFL} is always satisfied. 

For a given spectral subspace $E$, we let $\mbox{\ensuremath{\mathcal{P}_{1}}}=E,$
so that assumption (1) of Theorem \ref{CFL} is satisfied. Because
the real part of the spectrum of $A$ is assumed to be strictly negative,
the operator $\mathcal{A}$ defined in \eqref{eq:calAdef} satisfies
assumptions (2)-(3) of Theorem \ref{CFL}.

Next we note that the smallest integer $L$ satisfying 
\[
\left[\mathrm{Spect}(\mathcal{A}_{1})\right]^{L+1}\mathrm{Spect}(\mathcal{A}_{2}^{-1})\subset\left\{ z\in\mathbb{C}\,:\,\left|z\right|<1\right\} ,
\]
 is just the smallest integer that satisfies
\[
\left[e^{\max_{\lambda\in\mathrm{Spect}(A\vert_{E})}\mathrm{Re}\lambda}\right]^{L+1}e^{\min_{\lambda\in\mathrm{Spect}(A)-\mathrm{Spect}(A\vert_{E})}\mathrm{Re}\lambda}<1.
\]
The solution of this inequality for a general real number $L$ is
\[
L>\frac{\min_{\lambda\in\mathrm{Spect}(A)-\mathrm{Spect}(A\vert_{E})}\mathrm{Re}\lambda}{\max_{\lambda\in\mathrm{Spect}(A\vert_{E})}\mathrm{Re}\lambda}-1,
\]
 which, restricted to integer solutions, becomes 
\[
L\geq\sigma(E),
\]
with the relative spectral quotient $\sigma(E)$ defined in \eqref{eq:relative_sigma}.
The nonresonance condition \eqref{eq:nonresonancecond} can then be
written in our setting precisely in the form \eqref{eq:unforced_nonresonance}.
Thus, under the assumptions of Theorem \ref{theo:SSM unforced}, the
conditions of Theorem \ref{CFL} are satisfied, and the statements
of Theorem \ref{theo:SSM unforced} are restatements of Theorem \ref{CFL}
in our present context.

\subsection{Comparison with applicable results for normally hyperbolic invariant
manifolds\label{sub:NHIM proof unforced}}

Out of the three types of SSMs covered by Theorem \ref{theo:SSM unforced},
the existence of the slow SSMs (last column in Table 1) can also be
deduced in a substantially weaker form from the classical theory of
inflowing invariant normally hyperbolic invariant manifolds (Fenichel\cite{fenichel71}).
To show this, we first rescale variables via $x\to\delta x$ in system
\eqref{eq:unforced} to obtain the rescaled autonomous problem
\begin{equation}
\dot{x}=Ax+\delta\tilde{f}_{0}(x;\delta),\qquad\tilde{f}_{0}(x;\delta):=\frac{1}{\delta^{2}}f_{0}(\delta x).\label{eq:rescaled unforced}
\end{equation}

For $\delta=0$, this system coincides with the linearized system
\eqref{eq:linearization}, while for $\delta>0$, it is equivalent
to the full autonomous nonlinear system \eqref{eq:unforced}. 

Assume now that the slow spectral subspace $E_{1,\ldots,q}$ featured
in Table 1 satisfies the strict inequality 
\[
\mathrm{Re}\lambda_{q+1}<\mathrm{Re}\lambda_{q}.
\]
 This implies that $E_{1,\ldots,q}$, is normally hyperbolic, i.e.,
all decay rates of the linearized system within $E_{1,\ldots,q}$
are weaker than any decay rate transverse to $E_{1,\ldots,q}$. Furthermore,
a small compact manifold $\tilde{E}_{1,\ldots,q}\subset E_{1,\ldots,q}$
with boundary can be selected for the unperturbed limit ($\delta=0$)
of system \eqref{eq:rescaled unforced} such that $\dim\tilde{E}_{1,\ldots,q}=\dim E_{1,\ldots,q}$
and $\tilde{E}_{1,\ldots,q}$ is inflowing invariant under the unperturbed
limit of \eqref{eq:rescaled unforced}. This means that $Ax$ points
strictly outwards on the boundary $\partial\tilde{E}_{1,\ldots,q}$.
Then, for $\delta>0$ small enough, the classic results of Fenichel
\cite{fenichel71} imply the existence of an invariant manifold $\tilde{W}(0)$
with boundary in system \eqref{eq:rescaled unforced} that is $C^{1}$-close
to $\tilde{E}_{1,\ldots,q}$. Furthermore, $\dim\tilde{W}(0)=\dim E_{1,\ldots,q}$
and the manifold $\tilde{W}(0)$ is of class $C^{\gamma}$, with 
\begin{equation}
\gamma=\min\left(r,\mathrm{Int}\left[\frac{\mathrm{Re}\lambda_{q+1}}{\mathrm{Re}\lambda_{q}}\right]\right),\label{eq:gamma_unforced}
\end{equation}
which is the minimum of the degree of smoothness of \eqref{eq:rescaled unforced}
and the integer part of the ratio of the weakest decay rate normal
to $\tilde{E}_{1,\ldots,q}$ to the strongest decay rate inside $\tilde{E}_{1,\ldots,q}$.
Since $\delta>0$ has to be selected small in this result to keep
the norm $\delta\left|\tilde{f}(0)\right|$ small enough, the above
conclusion on the existence of $\tilde{W}(0)$ holds in a small enough
neighborhood of $x=0$ in system \eqref{eq:unforced}.

This result might seem attractive at the first sight, as it requires
no nonresonance conditions among the eigenvalues of the operator $A.$
At the same time, the properties of $\tilde{W}(0)$ are substantially
weaker than those obtained for $W_{1,\ldots,q}(0)$ in Theorem \ref{theo:SSM unforced}.
First, the degree $\gamma$ of differentiability for $\tilde{W}(0)$
(cf. formula \eqref{eq:gamma_unforced}) is generally much lower than
$r$, the degree of smoothness of system \eqref{eq:unforced}. In
particular, even if \eqref{eq:unforced} is analytic, the manifold
$\tilde{W}(0)$ may well just be once continuously differentiable,
and hence cannot be sought in the form of a convergent Taylor expansion.
Second, no uniqueness is guaranteed by the normal hyperbolicity results
of Fenichel \cite{fenichel71} for $\tilde{W}(0)$ within any class
of invariant manifolds. Third, the whole argument is only applicable
to slow SSMs, but not to intermediate and fast SSMs.

\subsection{Comparison with results deducible from analytic linearization theorems\label{sub:Poincare proof unforced}}

The analytic linearization theorem of Poincaré \cite{poincare79}
concerns complex systems of differential equations of the form 
\begin{equation}
\dot{y}=\Lambda y+g(y),\qquad g(y)=\mathcal{O}\left(\left|y\right|^{2}\right),\label{eq:poincare starting form}
\end{equation}
where $\Lambda\in\mathbb{C}^{N\times N}$is diagonalizable and $g(y)$
is analytic. If 
\begin{enumerate}
\item all eigenvalues of $\Lambda$ lie in the same open half plane in the
complex plane (e.g, $\mathrm{Re}\lambda_{j}<0$ for all $j$, as in
our case), and 
\item the nonresonance conditions $\left\langle m,\lambda\right\rangle \neq\lambda_{j}$
hold for all $l=1,\ldots,N$ for all integer vectors $m=(m_{1},\ldots,m_{N})$
with $m_{i}\geq0,$ and $\sum_{i}m_{i}\geq2$, 
\end{enumerate}
then there exists an analytic, invertible change of coordinates $z=h(y)$
in a neighborhood of the origin under which system \eqref{eq:poincare starting form}
transforms to the linear system 
\begin{equation}
\dot{z}=\Lambda z.\label{eq:poincare next}
\end{equation}
The spectral subspaces of this linear system are all defined by analytic
functions (trivially, flat graphs over themselves). As we discussed
in Section \ref{sub:Uniqueness-and-analyticity}, the spectral subspaces
of nonresonant linear systems are in fact the only analytic invariant
manifolds that are graphs over spectral subspaces. 

Recall that the composition of two analytic functions is analytic
and the inverse of an invertible analytic function is also analytic.
We can, therefore, transform back the spectral subspaces of \eqref{eq:poincare next}
under the analytic inverse mapping $h^{-1}(z)$ to conclude that \eqref{eq:poincare starting form}
also has unique analytic SSMs tangent at the origin to any selected
spectral subspace of the operator $\Lambda$. (Indeed, if \eqref{eq:poincare starting form}
had more than one such analytic SSMs, then those would have to transform
to nontrivial analytic SSMs of \eqref{eq:poincare next} under $h(y)$,
but no such nontrivial analytic SSMs exist in \eqref{eq:poincare next}.)
The unique analytic SSMs over spectral subspaces of \eqref{eq:poincare starting form}
can in turn be extended to smooth global invariant manifolds under
the reverse flow map of \eqref{eq:poincare starting form} up to the
maximum time of definition of backward solutions. 

Cirillo et al. \cite{cirillo15b} touches on parts of this argument
for the existence of two-dimensional SSMs in autonomous nonlinear
systems, without establishing uniqueness and analyticity in detail.
These authors involve the Koopman operator (cf. Mezi\'{c} \cite{mezic05})
in their arguments, but all spectral subspaces of a linear mapping
are well-defined without the need to view them as zero sets of Koopman
eigenfunctions. (These subspaces are in fact the only invariant manifolds
of the linearized system \eqref{eq:poincare next} out of the infinitely
many that are expressible as zero sets of Koopman eigenfunctions under
the nonresonance conditions given above.) Furthermore, as shown by
the argument above, the restriction to two-dimensional SSMs is not
necessary either.

The line of reasoning we gave above for the existence of autonomous
SSMs is complete but applicable only under assumptions that limit
its applicability in practice. Specifically, SSMs obtained from the
analytic linearization are applicable only when the linear operator
$A$ in system \eqref{eq:linearization} has no resonances, not even
inside any of the spectral subspaces. This latter assumption is a
limitation, as the main motivation in the nonlinear normal mode literature
for multi-mode Pierre--Shaw type invariant surfaces is precisely to
deal with internal resonances inside a spectral subspace $E_{j_{1},\ldots,j_{q}}.$
Furthermore, unlike Theorem \ref{theo:SSM unforced}, Poincaré's result
guarantees uniqueness only for analytic dynamical systems and only
within the class of analytic SSMs. This is again a limitation in practice,
as no finite order can be deduced over which a Taylor expansion will
only approximate the unique SSM. A relaxation of Poincaré's analytic
setting to the case of finite differentiability is available (Sternberg
\cite{sternberg57}). In that setting, however, the uniqueness of
SSMs can no longer be concluded within any function class, given that
the local linearizing transformation $h(y)$ is no longer unique.

\section{Existence, uniqueness and persistence for non-autonomous SSMs ($k>0$)\label{sec:Appendix-D:-Existence,}}

First, we recall a more abstract result of Haro and de la Llave \cite{haro06}
on quasiperiodic mappings and their sub-whiskers, which we subsequently
apply to our setting.

\subsection{\label{sub:Invariant-tori-and-whiskers}Invariant tori and their
spectral sub-whiskers in quasiperiodic maps}

We fix the finite-dimensional phase space $\mathcal{P=\mathbb{R}}^{N}\times\mathbb{T}^{k}$.
On an open subset $\mathcal{U}=U\times\mathbb{T}^{k}\subset\mathcal{P}$
of this phase space, we consider a map $\mathcal{F}_{1}\colon\mathcal{U}\to\mathbb{R}^{N}$.
For some $r\in\mathbb{N}\cup\{\infty,a\}$ and $s\geq2,$ we will
say that the map $\mbox{\ensuremath{\mathcal{F}_{1}} is of class \ensuremath{C^{r,s},} if}$
$\mathcal{F}_{1}(x,\phi)$ is $C^{r}$ in its second argument $\phi\mathbb{\in T}^{k}$,
and jointly $C^{r+s}$ in both of its arguments $(x,\phi)\in U\times\mathbb{T}^{k}$.
In other words, if $\mathcal{F}_{1}\in C^{r,s}$ then $\partial_{\phi}^{i}\partial_{x}^{j}\mathcal{F}_{1}$
exists and is continuous for all indices $(i,j)\in\mathbb{N}^{2}$
satisfying $i\leq r$ and $i+j\leq r+s.$ 

Next we assume that for any $\phi\in\mathbb{T}^{k}$, the map $\mathcal{F}_{1}(\,\cdot\,,\phi$)
is a local diffeomorphism. For a constant phase shift vector $\Delta\in\mathbb{R}^{k}$,
we define the quasiperiodic mapping $\mathcal{F}=\left(\mathcal{F}_{1},\mathcal{F}_{2}\right)\colon\mathcal{U}\times\mathbb{T}^{k}\to\mathcal{P}$
as
\[
\mathcal{F}(x,\phi)=\left(\mathcal{F}_{1}(x,\phi),\mathcal{F}_{2}(\phi)\right):=\left(\mathcal{F}_{1}(x,\phi),\phi+\Delta\right).
\]
Assume that $\mathcal{F}_{1}(0,\phi)=0$, i.e., $\mathcal{K}=\left\{ 0\right\} \times\mathbb{T}^{k}$
is an invariant torus for the map $\mathcal{F}$. Let $K:\mathbb{T}^{k}\to\mathbb{R}^{n}$
be a parametrization of the torus $\mathcal{K}$.

Next, we define the torus-transverse Jacobian 
\begin{equation}
M(\phi)=D_{x}\mathcal{F}_{1}(0,\phi)\label{eq:M_Jacobian}
\end{equation}
of the mapping component $\mathcal{F}_{1}$, and let $\nu\colon\mathbb{T}^{k}\to\mathbb{R}^{N}$
be any bounded mapping from the $k$-dimensional torus into $\mathbb{R}^{N}$.
We then define the \emph{transfer operator $\mathcal{T}_{\Delta}\colon\nu\mapsto\mathcal{T}_{\Delta}\nu$
}as\emph{ }a functional that maps the function $\nu$ into the function\emph{
\begin{equation}
\left[\mathcal{T}_{\Delta}\nu\right](\phi)=D_{x}\mathcal{F}_{1}(0,\phi-\Delta)\nu(\phi-\Delta).\label{eq:transfer_operator}
\end{equation}
}Note that $\mathcal{T}_{\Delta}$ is just the torus-transverse component
of the mapping $(\phi-\Delta,\nu(\phi-\Delta))\mapsto(\phi,\left[\mathcal{T}_{\Delta}\nu\right](\phi))$
which maps the vector $\nu(\phi-\Delta)$, an element of the normal
space of the torus $\mathcal{K}$ at the base point $(0,\phi-\Delta)$,
under the linearized map $D\mathcal{F}$ into a vector in the normal
space of $\mathcal{K}$ at the base point $(0,\phi)$. 

As long as $\nu$ is taken from the class of bounded functions, the
spectrum of the operator $\mathcal{T}_{\Delta}$ does not depend on
the smoothness properties of $\nu$ (see Theorem 2.12, Haro and de
la Llave \cite{haro06}). We will need the \emph{annular hull of the
spectrum} of $\mathcal{T}_{\Delta}$, defined as 
\begin{equation}
\mathcal{A}=\left\{ ze^{i\alpha}\,:\,\,z\in\mathrm{Spect}\mathcal{T}_{\Delta},\,\,\alpha\in\mathbb{R}\right\} .\label{eq:annular_hull_of_full_spectrum}
\end{equation}
This set is a union of circles in the complex plane, with each circle
obtained by rotating an element of the spectrum of $\mathcal{T}_{\Delta}$. 

We make the following assumptions:
\begin{description}
\item [{(0)}] The spectrum of the operator $\mathcal{T}_{\Delta}$ does
not intersect the complex unit circle, i.e.,
\[
\mathrm{Spect}\mathcal{T}_{\Delta}\cap\left\{ z\in\mathbb{C}\,:\,\left|z\right|=1\right\} =\emptyset.
\]

\item [{(1)}] There exists a decomposition of $N$$\mathcal{K}$, the normal
bundle of $\mathcal{K}$, into a direct sum 
\[
N\mathcal{K}=P_{1}\oplus P_{2}
\]
of two $C^{r}$ subbundles, $P_{1},P_{2}\subset N\mathcal{K}$, such
that $P_{1}$ is invariant under $M(\phi)$. As a consequence, a representation
of $M(\phi)$ with respect to this decomposition is given by
\[
M=\left(\begin{array}{cc}
M_{1}(\phi) & B(\phi)\\
0 & M_{2}(\phi)
\end{array}\right).
\]
 The corresponding restrictions of the transfer operator $\mathcal{T}_{\Delta}$
onto functions mapping into $P_{1}$ and $P_{2}$ will be denoted
as $\mathcal{T}_{1,\Delta}$ and $\mathcal{T}_{2,\Delta}$. The annular
hulls $\mathcal{A}_{j}$ of the spectra of these restricted operators
can be defined similarly to $\mathcal{A}$:
\begin{equation}
\mathcal{A}_{j}=\left\{ ze^{i\alpha}\,:\,\,z\in\mathrm{Spect}\mathcal{T}_{j,\Delta},\,\,\alpha\in\mathbb{R}\right\} ,\quad j=1,2,\qquad\mathcal{A}_{1}\cup\mathcal{A}_{2}=\mathcal{A}.\label{eq:annular_hulls_of_subspectra}
\end{equation}

\item [{(2)}] The annular hull of $\mathrm{Spect}(\mathcal{T}_{1,\Delta})$
lies strictly inside the complex unit circle, i.e., $\mathcal{A}_{1}\subset\left\{ z\in\mathbb{C}\,:\,\left|z\right|<1\right\} $ 
\item [{(3)}] For the smallest integer $L\geq1$ satisfying 
\begin{equation}
\mathcal{A}_{1}^{L+1}\mathcal{A}^{-1}\subset\left\{ z\in\mathbb{C}\,:\,\left|z\right|<1\right\} ,\label{eq:ratecond_quasiperiodic}
\end{equation}
 we have
\begin{equation}
\mathcal{A}_{1}^{i}\cap\mathcal{A}_{2}=\emptyset\label{eq:nonresonancecond_quasiperiodic}
\end{equation}
 for every integer $i\in[2,L]$ (in case $L\geq2)$
\item [{(5)}] $L+1\leq s$ 
\end{description}
We then have the following result:
\begin{thm}
{[}Haro and de la Llave, 2006{]} \label{theo:Haro-de-la-Llave Theorem}Under
assumptions (0)-(5):\end{thm}
\begin{description}
\item [{(i)}] There exists an invariant manifold $\mathcal{M}_{1}\subset\mathcal{P}$
that is a $C^{r,s}$ embedding of the subbundle $P_{1}$ into $\mathcal{P}$,
and is tangent to $P_{1}$ along the torus $\mathcal{K}$. 
\item [{(ii)}] The invariant manifold $\mathcal{M}_{1}$ is unique among
all $C^{r,L+1}$ invariant manifolds of $\mathcal{F}$ that are tangent
to the subbundle $P_{1}$ along the torus $\mathcal{K}$. That is,
every two $C^{r,L+1}$ invariant manifolds with this tangency property
will coincide in a neighborhood of $\mathcal{K}$.
\item [{(iii)}] There exists a map $R:P_{1}\to P_{1}$ that is a polynomial
of degree not larger than $L$ in the variable $\Delta$, of class
$C^{r}$ in $x$ and $C^{\infty}$ in $\phi$, and there exists a
$C^{r,s}$ map $W\colon U_{1}\subset P_{1}\to\mathcal{P}$, defined
over an open tubular neighborhood $U_{1}$ of the zero section of
$P_{1}$, satisfying
\[
R(0,\phi)=0,\quad D_{1}R(0,\phi)=M_{1},\quad W(0,\phi)=K(\phi),\quad\pi_{P_{1}}D_{1}W(0,\phi)=I{}_{P_{1}},\quad\pi_{E_{2}}D_{2}W(0,\phi)=0
\]
for all $\phi\mathbb{\in\mathbb{T}}^{k}$, such that $W$ serves as
an embedding of $\mathcal{M}_{1}$ from $P_{1}$ to $\mathcal{P}$,
and $R$ represents the pull-back of the dynamics on $\mathcal{M}_{1}$
to $U_{1}$ under this embedding. Specifically, we have
\[
\mathcal{F}_{1}(W(\eta,\phi),\phi)=W(R(\eta,\phi),\phi+\Delta)
\]
 in the tubular neighborhood $U_{1}$.
\item [{(iv)}] If we further assume that for some integer $L_{-}\geq2,$
we have $\mathcal{A}_{1}^{i}\cap\mathcal{A}_{1}=\emptyset$ for every
integer $i\in[L_{-},L],$ then $R$ can be chosen to be a polynomial
of degree not larger than $L_{-}-1.$ 
\item [{(v)}] If $\mathcal{A}_{2}\cap\left\{ z\in\mathbb{C}\,:\,\left|z\right|=1=\emptyset\right\} $
(i.e., the torus $\mathcal{K}$ is normally hyperbolic), then statements
(i)--(iv) remain valid under small enough $C^{r,s}$ perturbations
of the map $\mathcal{F}_{1}$. In particular, the invariant manifold
$\mathcal{M}_{1}$ and its parametrization persists smoothly under
small enough changes in parameters $\mu\in\mathbb{R}^{p}$ as long
as for the new variable $\tilde{\phi}=\left(\phi,\mu\right)$, the
function $\mathcal{F}_{1}(x,\tilde{\phi})$ is of class $C^{r,s}$.
\end{description}
These results have been collected, with minor notational changes,
from Theorem 4.1 and Remark 4.7 of Haro and de la Llave \cite{haro06}.

\subsection{\label{sub:proof SSM forced}Proofs of Theorem \eqref{theo:SSM forced}}

We consider eq. \eqref{eq:fullsystemagain} but will work with its
equivalent autonomous form 
\begin{eqnarray}
\dot{x} & = & Ax+f_{0}(x)+\epsilon f_{1}(x,\phi,\epsilon),\label{eq:extended_full_system-1}\\
\dot{\phi} & = & \Omega.\nonumber 
\end{eqnarray}
We will state the smoothness assumptions on $f_{0}$ and $f_{1}$
in more detail later. By (v) of Theorem \ref{theo:Haro-de-la-Llave Theorem},
we can first establish the existence of various spectral submanifolds
attached to the invariant torus $\mathcal{K}_{0}=\left\{ 0\right\} \times\Pi^{k}$
of the $\epsilon=0$ limit of \eqref{eq:extended_full_system-1}.
We then conclude the existence of similar submanifolds attached to
the quasiperiodic normal mode $x_{\epsilon}(t),$ represented by a
perturbed invariant torus $\mathcal{K}_{\epsilon}$ for $\epsilon>0$
in the full perturbed system \eqref{eq:extended_full_system-1}.

In the context of the above theorem, we are working on the phase space
$\mathcal{P}\mathbb{=R}^{N}\times\mathbb{T}^{k}$ and an open neighborhood
$\mathcal{U}=U\times\mathbb{T}^{k}$, where $U\subset\mathbb{R}^{N}$
is an open neighborhood of the fixed point $x=0$ of \eqref{eq:fullsystemagain}.
We define the mapping $\mathcal{F}$ as the time-one map of the autonomous
system \eqref{eq:extended_full_system-1} for $\epsilon=0$. , i.e.,
\begin{eqnarray}
\mathcal{F}(x,\phi) & = & \left(F_{0}^{1}(x),\phi+\Omega\right)\colon\mathcal{U}\to\mathcal{P},\nonumber \\
\mathcal{F}_{1}(x) & = & F_{0}^{1}(x),\nonumber \\
\mathcal{F}_{2}(\phi) & = & \phi+\Omega,\label{eq:mappings}
\end{eqnarray}
with the map $F_{0}^{1}$ denoting the time-one map of $\dot{x}=Ax+f_{0}(x).$
By our assumptions, we have $\mathcal{F}_{1}(0)=0$, and hence the
torus $\mathcal{K}_{0}=$ is an invariant torus for the map $\mathcal{F}$
for $\epsilon=0$. 

The Jacobian of the $x$-dynamics at $x=0$, as defined in \eqref{eq:M_Jacobian},
is 
\[
M(\phi)=D_{x}F_{0}^{1}(0)=e^{A},
\]
and the transfer operator defined in \eqref{eq:transfer_operator}\emph{
}takes the form\emph{ 
\[
\left[\mathcal{T}_{\Omega}\nu\right](\phi)=e^{A}\nu(\phi-\Omega).
\]
}We now Fourier expand the general function $\nu\colon\mathbb{T}^{k}\to\mathbb{R}^{N}$
as\emph{
\[
\nu(\phi)=\sum_{\left|m\right|=1}^{\infty}\nu_{m}e^{i\left\langle m,\phi\right\rangle },\qquad m\in\mathbb{Z}^{n}.
\]
 }Be definition, $\lambda\in\mathbb{C}$ is in the spectrum of the
operator\emph{ $\mathcal{T}_{\Omega}$ }if\emph{ $\left[\lambda I-\mathcal{T}_{\Omega}\right]^{-1}$
}does not exist. After Fourier-expanding $\mathcal{T}_{\Omega}\nu$,
we see that the non-invertibility of $\lambda I-\mathcal{T}_{\Omega}$
is equivalent to the non-solvability of
\[
\sum_{\left|m\right|=1}^{\infty}\left(\lambda I-e^{-i\left\langle m,\Omega\right\rangle }e^{A}\right)\nu_{m}e^{i\left\langle m,\phi\right\rangle }=\sum_{\left|m\right|=1}^{\infty}\tilde{\nu}_{m}e^{i\left\langle m,\phi\right\rangle }
\]
for the coefficients $\nu_{m}$, where $\tilde{\nu}_{m}$ is arbitrary
but fixed. This non-solvability arises precisely when
\[
\det\left[e^{A}-\lambda e^{i\left\langle m,\Omega\right\rangle }I\right]=0,
\]
i.e., when $\lambda e^{i\left\langle m,\Omega\right\rangle }$ is
contained in the spectrum $e^{A}$. We conclude that the spectrum
of \emph{$\mathcal{T}_{\Omega}$ }is given by 
\begin{equation}
\mathrm{Spect}\left(\mathcal{T}_{\Omega}\right)=\left\{ e^{\lambda_{j}-i\left\langle m,\Omega\right\rangle }:\,\,j=1,\ldots,d;\,\,\,m\in\mathbb{N}^{k}\right\} ,\label{eq:transfop_spectrum}
\end{equation}
where $\lambda_{j}$ are the eigenvalues of $A$, listed in \eqref{eq:eigenvalue_ordering}.
By the definition \eqref{eq:annular_hull_of_full_spectrum}, the annular
hull of $\mathrm{Spect}\mathcal{T}_{\Omega}$ is therefore
\begin{equation}
\mathcal{A}=\left\{ z\in\mathbb{C}\,:\,\left|z\right|=e^{\mathrm{Re}\lambda_{j}}:\,\,j=1,\ldots,d\right\} .\label{eq:annular hull of spectrum for vibrations}
\end{equation}
 For later reference, the analogous annular hull defined for the inverse
of $A$ is then
\[
\mathcal{A}^{-1}=\left\{ z\in\mathbb{C}\,:\,\left|z\right|=e^{\mathrm{-Re}\lambda_{j}}:\,\,j=1,\ldots,d\right\} .
\]

By assumption \eqref{eq:ass1-1}, eq. \eqref{eq:transfop_spectrum}
implies that hypotheses (0)-(2) of Theorem \ref{theo:Haro-de-la-Llave Theorem}
are satisfied. To verify the remaining assumptions of the theorem,
we note that the smallest integer $L$ satisfying
\[
\mathcal{A}_{1}^{L+1}\mathcal{A}^{-1}\subset\left\{ z\in\mathbb{C}\,:\,\left|z\right|<1\right\} 
\]
 is just the smallest integer that satisfies
\[
\left[e^{\max_{\lambda\in\mathrm{Spect}(A\vert_{E})}\mathrm{Re}\lambda}\right]^{L+1}e^{\min_{\lambda\in\mathrm{Spect}(A)}\mathrm{Re}\lambda}<1.
\]
The solution of this inequality for a general real $L$ is given by
\[
L>\frac{\min_{\lambda\in\mathrm{Spect}(A)}\mathrm{Re}\lambda}{\max_{\lambda\in\mathrm{Spect}(A\vert_{E})}\mathrm{Re}\lambda}-1.
\]
The integer solutions of this inequality therefore satisfy 
\[
L\geq\Sigma(E),
\]
with the absolute spectral quotient $\sigma(E)$ defined in \eqref{eq:absolute_sigma}.
The nonresonance condition \eqref{eq:nonresonancecond_quasiperiodic}
can be written in our setting precisely in the form \eqref{eq:forced_nonresonance}.
Thus, under the assumptions of Theorem \ref{theo:SSM forced}, the
conditions of Theorem \ref{theo:Haro-de-la-Llave Theorem} are satisfied.
The statements of Theorem \ref{theo:SSM forced} are then just restatements
of Theorem \ref{theo:Haro-de-la-Llave Theorem} in our present context.

\subsection{Comparison with applicable results for normally hyperbolic invariant
manifolds\label{sub:NHIM proof forced}}

As in the autonomous case, the existence of slow non-autonomous SSMs
(last column of Table 2) could also be deduced in a substantially
weaker form from the classic theory of inflowing invariant normally
hyperbolic invariant manifolds (Fenichel\cite{fenichel71}). 

Following the approach taken in Appendix \ref{sub:NHIM proof unforced}
for the autonomous case, we let $\delta=\sqrt{\epsilon}$ and use
the rescaling $x\to\delta x$ in system \eqref{eq:extended_full_system-1}
to obtain the equivalent dynamical system

\begin{eqnarray}
\dot{x} & = & Ax+\delta\left[\tilde{f}_{0}(x;\delta)+f_{1}(\delta x,\phi)\right],\label{eq:extended_full_system-1-1}\\
\dot{\phi} & = & \Omega.\nonumber 
\end{eqnarray}
Assume that the slow spectral subspace $E_{1,\ldots,q}$ featured
in row (1) Table 2 satisfies the strict inequality 
\[
\mathrm{Re}\lambda_{q+1}<\mathrm{Re}\lambda_{q}.
\]
 This implies that in the $\delta=0$ limit of system \eqref{eq:extended_full_system-1-1},
the torus bundle $\mathcal{K}_{0}\times E_{1,\ldots,q}$ is a normally
hyperbolic invariant manifold, i.e., all decay rates of the linearized
system within $\mathcal{K}_{0}\times E_{1,\ldots,q}$ are weaker than
any decay rate transverse to $E_{1,\ldots,q}$. Furthermore, a small
compact manifold $\mathcal{K}_{0}\times\tilde{E}_{1,\ldots,q}\subset\mathcal{K}_{0}\times E_{1,\ldots,q}$
with boundary can be selected such that $\dim\tilde{E}_{1,\ldots,q}=\dim E_{1,\ldots,q}$
and $\mathcal{K}_{0}\times\tilde{E}_{1,\ldots,q}$ is inflowing invariant
under the flow of \eqref{eq:extended_full_system-1-1} for $\delta=0$.
This specifically means that the vector field $(Ax,\Omega)$ points
strictly outwards on the boundary $\partial\left(\mathcal{K}_{0}\times\tilde{E}_{1,\ldots,q}\right)=\mathcal{K}_{0}\times\partial\tilde{E}_{1,\ldots,q}$
of $\mathcal{K}_{0}\times\tilde{E}_{1,\ldots,q}$. Then, for $\delta>0$
small enough, the results of Fenichel \cite{fenichel71} imply the
existence of an invariant manifold $\tilde{W}$ with boundary in system
\eqref{eq:extended_full_system-1-1} that is $\mathcal{O}(\delta)$
$C^{1}$-close to $\mathcal{K}_{0}\times\tilde{E}_{1,\ldots,q}$ within
a small neighborhood of $\mathcal{K}_{0}$. Furthermore, $\dim\tilde{W}=\dim\tilde{E}_{1,\ldots,q}+k$
and the manifold $\tilde{W}$ is of class $C^{\gamma}$ with the integer
$\gamma$ defined in \eqref{eq:gamma_unforced}. 

The limitations of this approach are identical to those discussed
in Appendix \ref{sub:NHIM proof unforced}.

\subsection{Comparison with results deducible from analytic linearization theorems\label{sub:Poincare proof forced}}

A time-quasiperiodic extension of the linearization theorem of Poincaré
\cite{poincare79} (cf. Appendix \ref{sub:Poincare proof unforced})
is given by Belaga \cite{belaga62} (cf. Arnold \cite{arnold88}),
covering differential equations of the form 
\begin{eqnarray}
\dot{y} & = & \Lambda y+g(y,\phi),\qquad g(y,\phi)=\mathcal{O}\left(\left|y\right|^{2}\right),\label{eq:poincare starting form-1}\\
\dot{\phi} & = & \Omega,
\end{eqnarray}
where $\Lambda\in\mathbb{C}^{N\times N}$is diagonalizable, $\phi\in\mathbb{T}^{k}$
and $g(y,\phi)$ is analytic. If 
\begin{enumerate}
\item all eigenvalues of $\Lambda$ lie in the same open half plane in the
complex plane (e.g, $\mathrm{Re}\lambda_{j}<0$ for all $j$ in our
setting), and 
\item the nonresonance conditions $\lambda_{l}\neq\left\langle m,\lambda\right\rangle +i\left\langle p,\Omega\right\rangle $
hold for all integer vectors $m\in(m_{1},\ldots,m_{N})$, with $m_{i}\geq0,$
and $\sum_{i}m_{i}\geq2$, and for all $p\in\mathbb{Z}^{k}$ ,
\end{enumerate}
then there exists an analytic, invertible change of coordinates $z=h(y)$
in a neighborhood of the origin under which system \eqref{eq:poincare starting form-1}
transforms to 
\begin{eqnarray}
\dot{z} & = & \Lambda z,\label{eq:poincare next-1}\\
\dot{\phi} & = & \Omega.\nonumber 
\end{eqnarray}

The spectral subbundles of the trivial normal mode $\left\{ z=0\right\} \times\mathbb{T}^{k}$
in this system are all defined by analytic functions, given as direct
products of flat graphs over any spectral subspace of $\Lambda$ with
the torus $\mathbb{T}^{k}$. It follows from our discussion in Section
\ref{sub:Uniqueness-and-analyticity} that these flat subbundles are
the only analytic spectral subbundles of \eqref{eq:poincare next-1}.
Then, following the argument in Section \ref{sub:Poincare proof unforced},
we conclude that \eqref{eq:poincare starting form-1} also has unique
analytic, quasiperiodic SSMs, tangent at the origin to any selected
spectral subspace of the operator $\Lambda$. These unique analytic
SSMs over spectral subspaces of \eqref{eq:poincare starting form-1}
can in turn be extended to smooth global invariant manifolds under
the reverse flow map of \eqref{eq:poincare starting form-1} up to
the maximum time of definition of backward solutions. 

This construct has all the practical limitations already discussed
Appendix \ref{sub:Poincare proof unforced}, plus two more. First,
resonances with the external forcing are also excluded by the above
nonresonance assumptions. Second, the term representing external,
time-dependent forcing must be fully nonlinear in the phase space
variables. The latter is rarely the case in mechanical models.

We close by noting that in the case of $k=1$ (single-frequency forcing),
the above results of Belaga can be extended to cover time-periodic
dependence in the linear operator $\Lambda$ as well (see Arnold \cite{arnold88}).
This is the mechanical setting for the formal manifold calculations
of Sinha et al. \cite{sinha05} and Redkar et al. \cite{redkar08}.
The limitations of the linearization approach discussed above remain
valid for this extension as well. In contrast, a direct application
of Theorem \ref{CFL} to the Poincaré map of \eqref{eq:poincare starting form-1}
with $k=1$ gives sharp existence, persistence and uniqueness results
for SSMs, assuming that the Floquet multipliers associated with the
time-dependent linearization are known. 

Similarly, if $\Lambda$ has quasiperiodic ($k>1)$ dependence on
$\phi$, Theorem \ref{theo:Haro-de-la-Llave Theorem} formally applies
to the quasiperiodic map associated with the linearized system, giving
sharp existence, persistence and uniqueness results for SSMs in the
nonlinear system. In this general case, however, the spectrum of the
transfer operator $\left[\mathcal{T}_{\Delta}\Omega\right](\phi)$
defined in \eqref{eq:transfer_operator} is not known and requires
a case-by-case analysis. For this reason, we have assumed throughout
this paper the common mechanical setting in which the operator $A$
of the linearized system is time-independent.


\begin{thebibliography}{10}
\bibitem{arnold88}Arnold, V. I., \emph{Geometrical Methods in the
Theory of Ordinary Differential Equations. }Springer, New York (1988)

\bibitem{arnold89}Arnold, V. I., \emph{Mathematical Methods of Classical
Mechanics}. Springer, New York (1989).

\bibitem{avramov10}Avramov, K. V., and Mikhlin, Y V., Nonlinear normal
modes for vibrating mechanical systems. Review of theoretical developments
A\emph{SME Applied Mechanics Reviews} \textbf{65} (2010) 060802-1

\bibitem{avramov13}Avramov, K. V., and Mikhlin, Y V., Review of applications
of nonlinear normal modes for vibrating mechanical systems\emph{ ASME
Applied Mechanics Reviews} \textbf{65} (2013) 020801-1

\bibitem{belaga62}E. G. Belaga, On the reducibility of a system of
differential equations in the neighborhood of a quasiperiodic motion,
\emph{Sov. Math. Dokl.} \textbf{143} 2 (1962) 255-258.

\bibitem{blanc13}Blanc, F., Touze, C., Mercier, J.F., Ege, K., and
Bonnet Ben-Dhia, A.S. On the numerical computation of nonlinear normal
modes for reduced-order modelling of conservative vibratory systems.
\emph{Mech Syst Signal Process} 36 (2013) 520\textendash 539.

\bibitem{boivin95-1}Boivin, N., Pierre, C., and Shaw, S. W., Non-linear
normal modes, invariance, and modal dynamics approximations of non-linear
Systems. \emph{Nonlinear Dynamics} \textbf{8 }(1994) 315-346.

\bibitem{boivin95-2}Boivin, N., Pierre, C., and Shaw, S. W., Nonlinear
modal analysis of structural systems featuring internal resonances.
\emph{J. Sound and Vibration} \textbf{182} (1995) 336\textendash 341.

\bibitem{gabale11}Gabale, A.P., and Sinha, S.C., Model reduction
of nonlinear systems with external periodic excitations via construction
of invariant manifolds. \emph{J. Sound and Vibration} \textbf{330}
(2011) 2596\textendash 2607.

\bibitem{cabre03}Cabré, P., Fontich, E., and de la Llave, R., The
parametrization method for invariant manifolds I: Manifolds associated
to non-resonant spectral subspaces. \emph{Indiana University Mathematics
J. }\textbf{52} (2003) 283-328.

\bibitem{cabre05}Cabré, P., Fontich, E., and de la Llave, R., The
parameterization method for invariant manifolds III: overview and
applications. \emph{J. Differential Equations }\textbf{218}\emph{
}(2005) 444\textendash 515.

\bibitem{cirillo15a}Cirillo, G. I., Mauroy, A., Renson, L., Kerschen,
G., Sepulchre, R., Global parametrization of the invariant manifold
defining nonlinear normal modes using the Koopman operator, \emph{in
Proc. ASME 2015 International Design Engineering Technical Conferences
\& Computers and Information in Engineering,} Boston, 2015.

\bibitem{cirillo15b}Cirillo, G. I., Mauroy, A., Renson, L., Kerschen,
G., Sepulchre, R., A spectral characterization of nonlinear normal
modes. \emph{J. Sound and Vibration. }\textbf{377} (2016) 284\textendash 301.

\bibitem{elmegard14}Elmegard, M., \emph{Mathematical Modeling and
Dimension Reduction in Dynamical Systems}, Ph.D. Thesis, Technical
University of Denmark (2014).

\bibitem{euler24}Euler, L, De seriebus divergenti bus, Opera omnia,
Ser. 1,14, Leipzig- Berlin, 247 (1924) 585-617. 

\bibitem{evans98}Evans, L. C., \emph{Partial Differential Equations.}
AMS Press, Rhode Island (1998).

\bibitem{fenichel71}Fenichel, N., Persistence and smoothness of invariant
manifolds for flows. \emph{Indiana U. Math. J. }\textbf{21}\textbf{\emph{
}}(1971) 193-226.

\bibitem{guckenheimer83}Guckenheimer, J., and Holmes, P., \emph{Nonlinar
Oscillations, Dynamical Systems, and Bifurcations of Vector Fields},
Springer, New York (1983).

\bibitem{delallave97}de la Llave, R., Invariant manifolds associated
to nonresonant spectral subspaces. \emph{J. Stat. Phys. }\textbf{\emph{87}}\emph{
}(1997) 211-249\emph{.}

\bibitem{lyapunov92}Lyapunov, A. M., The general problem of the stability
of motion. \emph{Int. Journal of Control}, 55 (1992) 531\textendash 534.

\bibitem{haro06}Haro, A., and de la Llave, R., $\mathcal{A}$ parameterization
method for the computation of invariant tori and their whiskers in
quasi-periodic maps: Rigorous results. \emph{Differential Equations}
\textbf{228} (2006) 530\textendash 579.

\bibitem{haro16}Haro, A., Canadell, M., Figueras, J.-L., Luque, A.,
Mondelo, J.M., \emph{The Parameterization Method for Invariant Manifolds:
From Rigorous Results to Effective Computations. }(to appear) Springer,
New York (2016).

\bibitem{hirsch77}Hirsch, M.W., Pugh, C.C., and Shub, M., \emph{Invariant
Manifolds. }Lecture Notes Math. \textbf{583}, Springer--Verlag, New
York (1977)

\bibitem{hirsch13}Hirsch, W., Smale, S., and Devaney, R., L., \emph{Differential
Equations, Dynamical Systems, and an Introduction to Chaos} (3rd ed.)
Academic Press, Oxford (2013)

\bibitem{jiang05}Jiang, D., Pierre, C., and Shaw, S. W., Nonlinear
normal modes for vibratory systems under harmonic excitation. \emph{J.
Sound Vib.}, \textbf{288} (2005) 791\textendash 812.

\bibitem{kelley69}Kelley, A. F., Analytic two-dimensional subcenter
manifolds for systems with an integral. Pacific J. of Mathematics.
29 (1969) 335-350.

\bibitem{kerschen09}Kerschen, G, Peeters, M, Golinval, J.C., and
Vakakis, A. F. Nonlinear normal modes, Part I: a useful framework
for the structural dynamicist. \emph{Mech. Syst. Signal. Process.}
\textbf{23} (2009) 170\textendash 194. 

\bibitem{kerschen14}Kerschen, G. (ed.), \emph{Modal Analysis of Nonlinear
Mechanical Systems}. Springer, Berlin (2014).

\bibitem{kuether15}Kuether, R. J., , Renson, L., Detroux, T., Grappasonni,
C., Kerschen, G., and $\mathcal{A}$llen, MS., Nonlinear normal modes,
modal interactions and isolated resonance curves. \emph{J. Sound and
Vibration} \textbf{351} (2015) 299\textendash 310.

\bibitem{mezic05}Mezi\'{c}, I., Spectral properties of dynamical
systems, model reduction and decompositions. \emph{Nonlinear Dynamics}
\textbf{41} (2005) 309\textendash 325.

\bibitem{mireles--james15}Mireles--James, J.D., Polynomial approximation
of one parameter families of (un)stable manifolds with rigorous computer
assisted error bounds. \emph{Indagationes Mathematicae }\textbf{26}
(2015) 225\textendash 265.

\bibitem{nayfeh04}Nayfeh, A. H., \emph{Perturbation Methods. }Wiley
(2004). 

\bibitem{neild15}Neild, S. A., Champneys, A. R, Wagg, D.J., Hill
TL, and Cammarano, A. The use of normal forms for analysing nonlinear
mechanical vibrations. \emph{Phil. Trans. R. Soc. }A 373 (2015) 20140404.

\bibitem{peschek01}Peschek, E., Boivin, N., Pierre, C., and Shaw,
S. W., Nonlinear modal analysis of structural systems using multi-mode
invariant manifolds. \emph{Nonlinear Dynamics} \textbf{25} (2001)

\bibitem{pesheck02}Pesheck, E., Pierre, C., and Shaw, S.W. A new
Galerkin-based approach for accurate non-linear normal modes through
invariant manifolds, \emph{J. Sound and Vibration} \textbf{249} (5)
(2002) 971\textendash 993.

\bibitem{peters09}Peeters, M., Viguié, R., Sérandour, G., Kerschen,
G., and Golinval, J.C. Nonlinear normal modes, Part II: toward a practical
computation using numerical continuation techniques. \emph{Mech. Syst.
Signal. Process.} \textbf{23} (2009) 195\textendash 216. 

\bibitem{pierre06}Pierre, C., Jiang, D., and Shaw, S. W., Nonlinear
normal modes and their application in structural dynamics. \emph{Math.
Problems in Engineering }10847 (2006) 1\textendash 15.

\bibitem{poincare79}Poincaré, J.H., Sur les propriétés des fonctions
définies par les équations différences, Gauthier-Villars, Paris (1879).

\bibitem{redkar08}Redkar, S., and Sinha, S.C., A direct approach
to order reduction of nonlinear systems subjected to external periodic
excitations. \emph{J. Computational and Nonlinear Dynamics. }\textbf{3}
(2008) 031011-1

\bibitem{renson14}Renson, L., Del\'{i}ege, G., and Kerschen, G.,
An effective finite-element-based method for the computation of nonlinear
normal modes of nonconservative systems. \emph{Meccanica} \textbf{49}
(2014) 1901\textendash 1916.

\bibitem{renson16}Renson, L., Kerschen, G., and Cochelin, G., Numerical
computation of nonlinear normal modes in mechanical engineering. \emph{J,
Sound and Vibration} \textbf{364} (2016) 177\textendash 206.

\bibitem{rosenberg62}Rosenberg, R. M., The normal modes of nonlinear
$n$-degree-of-freedom systems. \emph{J. Applied Mech.} \textbf{30}
(1962) 7\textendash 14.

\bibitem{sanders85}Sanders, J. A., and Verhulst, F., \emph{Averaging
Methods in Nonlinear Dynamical Systems.} Springer-Verlag, New York
(1985).

\bibitem{shaw93}Shaw, S. W., and Pierre, C., Normal modes for non-linear
vibratory systems. \emph{J. Sound and Vibrations} \textbf{164} (1993)
85-124.

\bibitem{shaw94}Shaw, S. W., and Pierre, C., Normal modes of vibration
for nonlinear continuous systems.\emph{ Journal of Sound and Vibration.}
\textbf{169} (1994) 319-347.

\bibitem{shaw99}Shaw, S. W., Peschek, E., and Pierre, C., Modal analysis-based
reduced-order models for nonlinear structures--An invariant manifold
approach. \emph{Shock and Vibration Digest }\textbf{\emph{31}}\emph{
}(1999) 1-16.

\bibitem{sinha05}Sinha, S.C., Redkar, S., and Butcher, E.A., Order
reduction of nonlinear systems with time periodic coefficients using
invariant manifolds. \emph{J. Sound and Vibration} \textbf{284 }(2005)
985\textendash 1002.

\bibitem{sternberg57}Sternberg, S., Local contractions and a theorem
of Poincaré. \emph{American J. Math}. \textbf{79} (1957) 809-824.

\bibitem{vakakis01}Vakakis, A. (ed.), \emph{Normal Modes and Localization
in Nonlinear Systems}, Kluwer, Dordrecht (2001).

\bibitem{verhulst15}Verhulst, F., Profits and pitfalls of timescales
in asymptotics, \emph{SIAM Review}. \textbf{57}, No. 2 (2015) 255\textendash 274. \end{thebibliography}
\end{document}